\documentclass[12pt, twoside]{article}
\usepackage{amsmath,amsthm,amssymb}
\usepackage{times}
\usepackage{enumerate}

\def\titlerunning#1{\gdef\titrun{#1}}
\makeatletter
\def\author#1{\gdef\autrun{\def\and{\unskip, }#1}\gdef\@author{#1}}
\def\address#1{{\def\and{\\\hspace*{18pt}}\renewcommand{\thefootnote}{}%
\footnote {#1}}%
\markboth{\autrun}{\titrun}}
\makeatother
\def\email#1{e-mail: #1}
\def\subjclass#1{{\renewcommand{\thefootnote}{}%
\footnote{\emph{Mathematics Subject Classification (2010):} #1}}}
\def\keywords#1{\par\medskip
\noindent\textbf{Keywords.} #1}

\newtheorem{proposition}{Proposition}
\newtheorem{lemma}{Lemma}
\newtheorem{theorem}{Theorem}
\newtheorem{corollary}{Corollary}

\theoremstyle{definition}

\newtheorem{example}{Example}
\newtheorem{remark}{Remark}

\newtheorem{assumption}{Assumption}

\numberwithin{equation}{section}

\def\eps{{\varepsilon}}

\def\Bbb E{\mathbb{E}}
\def\Bbb R{\mathbb{R}}
\makeatletter \@addtoreset{equation}{section}

\makeatother

\font\tencmmib=cmmib10 \skewchar\tencmmib '60
\newfam\cmmibfam
\textfont\cmmibfam=\tencmmib

\font\tenmsb=msbm10 
\def\Bbb#1{\hbox{\tenmsb#1}}

\def\lessim{\ \lower4pt\hbox{$
\buildrel{\displaystyle <}\over\sim$}\ }
\def\gessim{\ \lower4pt\hbox{$\buildrel{\displaystyle >}
\over\sim$}\ }

\def\eps{\varepsilon}

\def\go0{\to 0}

\def\leftitem#1{\item{\hbox to\parindent{\enspace#1\hfill}}}

\def\sg{\sigma}

\def\sg2{\sigma^2}

\def\__{_{\infty}}
\begin{document}

%%%%% To ease editing, add:

\baselineskip=17pt

%%%%%%%%%%%%%%%%

%% In the running head, give an abbreviation of the title. 
\titlerunning{Estimation of functionals of covariance operators}

\title{Asymptotically Efficient Estimation 
of Smooth Functionals of Covariance Operators}

\author{Vladimir Koltchinskii}

\date{}

\maketitle

\address{School of Mathematics,
Georgia Institute of Technology, Atlanta, GA 30332-0160, USA;\email{vlad@math.gatech.edu}}


%%%%%%%%

\begin{abstract}
Let $X$ be a centered Gaussian random variable in a separable Hilbert space ${\mathbb H}$
with covariance operator $\Sigma.$  
We study a problem of estimation of a smooth functional 
of $\Sigma$ based on a sample $X_1,\dots ,X_n$ of $n$ independent 
observations of $X.$ More specifically, we are interested in functionals of the form 
$\langle f(\Sigma), B\rangle,$ where $f:{\mathbb R}\mapsto {\mathbb R}$
is a smooth function and $B$ is a nuclear operator in ${\mathbb H}.$ 
We prove concentration and normal approximation bounds for plug-in estimator 
$\langle f(\hat \Sigma),B\rangle,$ $\hat \Sigma:=n^{-1}\sum_{j=1}^n X_j\otimes X_j$
being the sample covariance based on $X_1,\dots, X_n.$ These bounds 
show that $\langle f(\hat \Sigma),B\rangle$ is an asymptotically normal 
estimator of its expectation ${\mathbb E}_{\Sigma} \langle f(\hat \Sigma),B\rangle$ 
(rather than of parameter of interest $\langle f(\Sigma),B\rangle$) with a parametric convergence rate $O(n^{-1/2})$
provided that the effective rank ${\bf r}(\Sigma):= \frac{{\bf tr}(\Sigma)}{\|\Sigma\|}$
(${\rm tr}(\Sigma)$ being the trace and $\|\Sigma\|$ being the operator norm of $\Sigma$)
satisfies the assumption ${\bf r}(\Sigma)=o(n).$ At the same time, we show that the bias of this estimator
is typically as large as $\frac{{\bf r}(\Sigma)}{n}$ 
(which is larger than $n^{-1/2}$ if ${\bf r}(\Sigma)\geq n^{1/2}$). In the case when ${\mathbb H}$ is a finite-dimensional space 
of dimension $d=o(n),$ we develop a method of bias reduction and construct an estimator 
$\langle h(\hat \Sigma),B\rangle$ of $\langle f(\Sigma),B\rangle$ that is asymptotically normal with convergence rate $O(n^{-1/2}).$ Moreover, we study asymptotic properties of the risk of this estimator 
and prove asymptotic minimax lower bounds for arbitrary estimators showing the asymptotic efficiency of $\langle h(\hat \Sigma),B\rangle$ in a semi-parametric sense.

%% Keywords are optional
\end{abstract}

\section{Introduction}\label{Sec:Intro}

Let $X$ be a random variable in a separable Hilbert space ${\mathbb H}$ sampled from a Gaussian distribution with mean $0$ and covariance operator 
$\Sigma:= {\mathbb E}(X\otimes X)$ (denoted in what follows $N(0;\Sigma)$).  
The purpose of this paper is to study 
a problem of estimation of smooth functionals of unknown covariance $\Sigma$
based on a sample $X_1,\dots, X_n$ of i.i.d. observations of $X.$ Specifically,
we deal with the functionals of the form $\langle f(\Sigma), B\rangle,$ where 
$f:{\mathbb R}\mapsto {\mathbb R}$ is a smooth function\footnote{More precisely, the ``smoothness" in 
this paper means that the function belongs to the Besov space $B^{s}_{\infty,1}({\mathbb R})$ for a proper value of 
$s>0,$ see Subsection \ref{SubSec:Overview}.} and $B$ is a nuclear 
operator. The estimation 
of bilinear forms of spectral projection operators of covariance $\Sigma,$ 
which is of importance in the principal component analysis,   
could be easily reduced to this basic problem. Moreover, the estimation of $\langle f(\Sigma),B\rangle$
is a major building block in the development of methods of statistical estimation much more general functionals of covariance such as the functionals of the form 
$\langle f_1(\Sigma), B_1\rangle \dots \langle f_k(\Sigma), B_k\rangle$
and their linear combinations.

Throughout the paper, we use the following notations. Given $A,B\geq 0,$ $A\lesssim B$ means 
that $A\leq CB$ for a numerical (most often, unspecified) constant $C>0;$ $A\gtrsim B$ is equivalent 
to $B\lesssim A;$ $A\asymp B$ is equivalent to $A\lesssim B$ and $B\lesssim A.$  
Sometimes, constants in the above relationships might depend on some parameter(s).
In such cases, the signs $\lesssim,$ $\gtrsim$ and $\asymp$ are provided 
with subscripts: say, $A\lesssim_{\gamma}B$ means that $A\leq C_{\gamma}B$
for a constant $C_{\gamma}>0$ that depends on $\gamma.$

Let ${\mathcal B}({\mathbb H})$ denote the space of all bounded linear 
operators in a separable Hilbert space ${\mathbb H}$ equipped with the 
operator norm and let ${\mathcal B}_{sa}({\mathbb H})$ denote the subspace 
of all self-adjoint operators.\footnote{The main results of the paper are proved in the 
case when ${\mathbb H}$ is a real Hilbert space. However, on a couple of occasions,
especially in auxiliary statements, its complexification ${\mathbb H}^{{\mathbb C}}=\{u+iv:u,v\in {\mathbb H}\}$
with a standard extension of the inner product and complexification of the operators acting in ${\mathbb H}$
is needed. With some abuse of notation, we keep in such cases the notation ${\mathbb H}$ for the complex Hilbert 
space.}
In what follows, $A^{\ast}$ denotes the adjoint of operator $A\in {\mathcal B}({\mathbb H}),$
${\rm tr}(A)$ denotes its trace (provided that $A$ is trace class) and $\|A\|$ denotes its operator norm. 
We use the notation $\|A\|_p$ for the Schatten $p$-norm of $A:$
$\|A\|_p^p := {\rm tr}(|A|^p), |A|= (A^{\ast}A)^{1/2}, p\in [1,\infty].$
In particular, $\|A\|_1$ is the nuclear norm of $A,$ $\|A\|_2$ is its Hilbert--Schmidt
norm and $\|A\|_{\infty}=\|A\|$ is its operator norm. We denote 
the space of self-adjoint operators $A$ with $\|A\|_p<\infty$
($p$-th Schatten class operators) by  
${\mathcal S}_p={\mathcal S}_p({\mathbb H}), 1\leq p\leq \infty.$
The space of compact self-adjoint operators in ${\mathbb H}$
is denoted by ${\mathcal C}_{sa}({\mathbb H}).$
The inner product notation 
$\langle \cdot,\cdot \rangle$ is used both for inner products in the underlying Hilbert space ${\mathbb H}$ and for the Hilbert--Schmidt inner product between 
the operators. Moreover, it is also used to denote bounded linear functionals on the spaces of operators
(for instance, $\langle A,B\rangle,$ where $A$ is a bounded operator and $B$ 
is a nuclear operator, is a value of such a linear functional on the space of bounded 
operators). For $u,v\in {\mathbb H},$ $u\otimes v$ denotes the tensor product of vector $u$ and 
$v:$ $(u\otimes v)x:= u\langle v,x\rangle, x\in {\mathbb H}.$ The operator $u\otimes v$ is of rank $1$
and finite linear combinations of rank one operators are operators of finite rank. The rank of $A$ is denoted by ${\rm rank}(A).$ Finally, ${\mathcal C}_+({\mathbb H})$ denotes the cone of self-adjoint 
positively semi-definite nuclear operators in ${\mathbb H}$ (the covariance operators). 

In what follows, we often use exponential bounds for random
variables of the following form: for all $t\geq 1$ with probability at least $1-e^{-t},$
$\xi \leq Ct.$ Sometimes, our derivation would yield a slightly different probability bound, 
for instance: for all $t\geq 1$ with probability at least $1-3e^{-t},$ $\xi \leq Ct.$
Such bounds could be easily rewritten again as $1-e^{-t}$ by adjusting the value of constant $C:$ 
for $t\geq 1$ with probability at least $1-e^{-t} = 1- 3e^{-t-\log(3)},$
we have $\xi \leq C(t+\log (3))\leq 2\log(3) C t.$  
Such an adjustment of constants will be used in many proofs without further notice. 

\subsection{Sample covariance and effective rank}

Let $\hat \Sigma$ denote the sample 
covariance based on the data $X_1,\dots, X_n:$
$$
\hat \Sigma := n^{-1}\sum_{j=1}^n X_j\otimes X_j.
$$
It is well known that $\hat \Sigma$ is a complete sufficient statistics 
and equals the maximum likelihood estimator in the problem of estimation 
of unknown covariance in the model $X_1,\dots, X_n$ i.i.d. $\sim N(0;\Sigma).$

In what follows, we often use so called {\it effective rank} of covariance $\Sigma$
as a complexity parameter of covariance estimation problem.
It is defined as 
$$
{\bf r}(\Sigma):= \frac{{\rm tr}(\Sigma)}{\|\Sigma\|}.
$$
Note that 
$
{\bf r}(\Sigma)\leq {\rm rank}(\Sigma)\leq {\rm dim}({\mathbb H}).
$
The following result of Koltchinskii and Lounici 
\cite{Koltchinskii_Lounici_arxiv} shows that, in the Gaussian case, the size 
of the random variable $\frac{\|\hat \Sigma-\Sigma\|}{\|\Sigma\|}$
(which is a relative operator norm error of the estimator $\hat \Sigma$
of $\Sigma$) is completely characterized by the ratio $\frac{{\bf r}(\Sigma)}{n}.$

\begin{theorem}
\label{th: operator norm_hatSigma}
The following bound holds:
\begin{equation}
\label{operator norm_hatSigma}
{\mathbb E}\|\hat \Sigma-\Sigma\|\asymp \|\Sigma\|\biggl(\sqrt{\frac{{\bf r}(\Sigma)}{n}}\bigvee 
\frac{{\bf r}(\Sigma)}{n}\biggr).
\end{equation}
Moreover, for all $t\geq 1$ with probability at least $1-e^{-t}$ 
\begin{equation}
\label{conc_oper_hatSigma}
\Bigl|\|\hat \Sigma-\Sigma\|-{\mathbb E}\|\hat \Sigma-\Sigma\|\Bigr|
\lesssim \|\Sigma\|\biggl(\biggl(\sqrt{\frac{{\bf r}(\Sigma)}{n}}\bigvee 1\biggr)\sqrt{\frac{t}{n}}\bigvee 
\frac{t}{n}\biggr).
\end{equation}
\end{theorem}

It follows from the expectation bound \eqref{operator norm_hatSigma} and the concentration inequality \eqref{conc_oper_hatSigma} that, for all $t\geq 1$ with probability at least $1-e^{-t},$
\begin{equation}
\label{operator_hatSigma_exp}
\|\hat \Sigma-\Sigma\|
\lesssim
\|\Sigma\|\biggl(\sqrt{\frac{{\bf r}(\Sigma)}{n}}\bigvee \frac{{\bf r}(\Sigma)}{n}
\bigvee \sqrt{\frac{t}{n}} \bigvee \frac{t}{n}\biggr)
\end{equation}
\label{operator_hatSigma_L_p}
and, for all $p\geq 1,$
\begin{equation}
{\mathbb E}^{1/p}\|\hat \Sigma-\Sigma\|^p\lesssim_p \|\Sigma\|\biggl(\sqrt{\frac{{\bf r}(\Sigma)}{n}}\bigvee 
\frac{{\bf r}(\Sigma)}{n}\biggr).
\end{equation}
To avoid the dependence of the constant on $p,$ the following modification of the above bound will be used on a couple of occasions:
\begin{equation}
\label{bound_p_r_Sigma}
{\mathbb E}^{1/p}\|\hat \Sigma-\Sigma\|^p\lesssim \|\Sigma\|\biggl(\sqrt{\frac{{\bf r}(\Sigma)}{n}}\bigvee 
\frac{{\bf r}(\Sigma)}{n} \bigvee \sqrt{\frac{p}{n}}\bigvee \frac{p}{n}\biggr).
\end{equation}
Since ${\bf r}(\Sigma)\leq d:={\rm dim}({\mathbb H}),$ the bounds in terms of effective rank do imply 
well known bounds in terms of dimension. 
For instance, for all $t\geq 1$ with probability 
at least $1-e^{-t},$ 
\begin{equation}
\label{operator_hatSigma_exp_dimension}
\|\hat \Sigma-\Sigma\|
\lesssim
\|\Sigma\|\biggl(\sqrt{\frac{d}{n}}\bigvee \frac{d}{n}
\bigvee \sqrt{\frac{t}{n}} \bigvee \frac{t}{n}\biggr)
\end{equation}
(see, e.g., \cite{Vershynin}). Of course, bound \eqref{operator_hatSigma_exp_dimension}
is meaningless in the infinite-dimensional case. In the finite-dimensional case, 
it is sharp if $\Sigma$ is isotropic ($\Sigma=cI_d$ for a constant $c$), 
or if it is of {\it isotropic type}, that is,  the spectrum of $\Sigma$
is bounded from above and bounded away from zero (by constants). 
In this case, ${\bf r}(\Sigma)\asymp d,$ which makes \eqref{operator_hatSigma_exp_dimension}
sharp. This is the case, for instance, for popular {\it spiked covariance models} introduced 
by Johnstone \cite{Johnstone} (see also \cite{Johnstone_Lu, Paul_2007, Birnbaumetal}). However, in the case of a fast decay of eigenvalues of $\Sigma,$
the effective rank ${\bf r}(\Sigma)$ could be significantly smaller than $d$ and it becomes 
the right complexity parameter in covariance estimation. 

In what follows, we are interested in the problems in which ${\bf r}(\Sigma)$ is allowed 
to be large, but ${\bf r}(\Sigma)=o(n)$ as $n\to \infty.$ This is a necessary and sufficient 
condition for $\hat \Sigma$ to be an operator norm consistent estimator of $\Sigma,$
which also means that $\hat \Sigma$ is a small perturbation of $\Sigma$ when $n$
is large and methods of perturbation theory could be used to analyze the behavior 
of $f(\hat \Sigma)$ for smooth functions $f.$ 

\subsection{Overview of main results} 
\label{SubSec:Overview}

In this subsection, we state and discuss the main results of the paper concerning 
asymptotically efficient estimation of functionals $\langle f(\Sigma),B\rangle$ for a smooth function 
$f:{\mathbb R}\mapsto {\mathbb R}$ and nuclear operator $B.$
It turns out that the proper notion of smoothness of function $f$
in these problems is defined in terms of Besov spaces and Besov 
norms. The relevant definitions (of spaces $B^{s}_{\infty,1}({\mathbb R})$ and corresponding norms), notations and references are provided in Section \ref{Sec:Entire}. 

A standard approach to asymptotic analysis of plug-in estimators
(in particular, such as $\langle f(\hat \Sigma), B\rangle$) in statistics is the Delta Method
based on the first order Taylor expansion of $f(\hat \Sigma).$ 
Due to a result by Peller (see Section \ref{Sec:Entire}), for any $f\in B_{\infty,1}^1({\mathbb R}),$
the mapping $A\mapsto f(A)$ is Fr\'echet differentiable with respect to the operator norm 
on the space of bounded self-adjoint operators in ${\mathbb H}.$
Let $\Sigma$ be a covariance operator with spectral decomposition 
$\Sigma:= \sum_{\lambda\in \sigma(\Sigma)}\lambda P_{\lambda},$ $\sigma(\Sigma)$ being the spectrum 
of $\Sigma,$ $\lambda$ being an eigenvalue 
of $\Sigma$ and $P_{\lambda}$ being the corresponding spectral projection 
(the orthogonal projection onto the eigenspace of $\Sigma$). Then the 
derivative $Df(\Sigma)(H)=Df(\Sigma;H)$ of operator function $f(A)$ at $A=\Sigma$ 
in the direction $H$ is given by the following formula:
$$
Df(\Sigma;H)= \sum_{\lambda,\mu\in \sigma(\Sigma)}f^{[1]}(\lambda, \mu)P_{\lambda}HP_{\mu},
$$  
where $f^{[1]}(\lambda,\mu)= \frac{f(\lambda)-f(\mu)}{\lambda-\mu}, \lambda\neq \mu$
and $f^{[1]}(\lambda,\mu)=f'(\lambda), \lambda=\mu$ (see Section \ref{Sec:Entire}). 
Moreover, if, for some $s\in (1,2],$ $f\in B_{\infty,1}^s({\mathbb R}),$
then the following first order Taylor expansion holds 
$$
f(\hat \Sigma) - f(\Sigma)= Df(\Sigma; \hat \Sigma-\Sigma)+ S_f(\Sigma; \hat \Sigma-\Sigma)
$$
with the linear term 
$
Df(\Sigma; \hat \Sigma-\Sigma)= n^{-1}\sum_{j=1}^n Df(\Sigma; X_j\otimes X_j-\Sigma)
$
and the remainder $S_f(\Sigma; \hat \Sigma-\Sigma)$ satisfying the bound 
$$
\|S_f(\Sigma; \hat \Sigma-\Sigma)\|\lesssim_s \|f\|_{B_{\infty,1}^s} \|\hat \Sigma-\Sigma\|^s
$$
(see \eqref{bound_s_1}). Since the linear term $Df(\Sigma; \hat \Sigma-\Sigma)$ 
is the sum of i.i.d. random variables, it is easy to check (for instance, using Berry-Esseen bound)
that $\sqrt{n}\langle Df(\Sigma; \hat \Sigma-\Sigma),B\rangle$ is asymptotically normal 
with the limit mean equal to zero and the limit variance 
$$
\sigma_f^2 (\Sigma;B):= 2\|\Sigma^{1/2}Df(\Sigma;B)\Sigma^{1/2}\|_2^2. 
$$
Using exponential bound \eqref{operator_hatSigma_exp}
on $\|\hat \Sigma-\Sigma\|,$ one can easily conclude that the remainder 
$\langle S_f(\Sigma; \hat \Sigma-\Sigma),B\rangle$ is asymptotically negligible (that is, of the 
order $o(n^{-1/2})$) if $\Bigl(\frac{{\bf r}(\Sigma)}{n}\Bigr)^{s/2}=o(n^{-1/2}),$ or, 
equivalently, ${\bf r}(\Sigma)= o(n^{1-\frac{1}{s}}).$ In the case when $s=2,$
this means that ${\bf r}(\Sigma)=o(n^{1/2}).$ This implies that $\langle f(\hat \Sigma),B\rangle$
is an asymptotically normal estimator of $\langle f(\Sigma), B\rangle$ with convergence rate $n^{-1/2}$ and limit normal distribution $N(0;\sigma_f^2(\Sigma;B))$ (under the assumption that ${\bf r}(\Sigma)= o(n^{1-\frac{1}{s}})$).  
The above perturbation analysis is essentially the same as for spectral projections of $\hat \Sigma$ in the case of fixed 
finite dimension (see Anderson \cite{Anderson}), or in the infinite-dimensional case 
when the ``complexity" of the problem (characterized by ${\rm tr}(\Sigma)$ or ${\bf r}(\Sigma)$)
is fixed (see Dauxois, Pousse and Romain \cite{DPR}). 
Note also that the bias of estimator $\langle f(\hat \Sigma),B\rangle,$ 
$$
\langle {\mathbb E}_{\Sigma}f(\hat \Sigma)-f(\Sigma),B\rangle =  \langle {\mathbb E}_{\Sigma}S_f(\Sigma;\hat \Sigma-\Sigma), B\rangle,
$$
is upper bounded by  $\lesssim \|f\|_{B_{\infty,1}^s} \|B\|_1\Bigl(\frac{{\bf r}(\Sigma)}{n}\Bigr)^{s/2},$
so, it is of the order $o(n^{-1/2})$ (asymptotically negligible) under the same condition ${\bf r}(\Sigma)= o(n^{1-\frac{1}{s}}).$ Moreover, it is easy to see that this bound on the bias is sharp for generic 
smooth functions $f.$ For instance, if $f(x)=x^2$ and $B= u\otimes u,$ then one can check by 
a straightforward computation that 
$$
\sup_{\|u\|\leq 1}|\langle {\mathbb E}_{\Sigma}f(\hat \Sigma)-f(\Sigma),u\otimes u\rangle|
= \frac{\|{\rm tr}(\Sigma)\Sigma+\Sigma^2\|}{n}\asymp \|\Sigma\|^2 \frac{{\bf r}(\Sigma)}{n}.
$$
This means that, as soon as ${\bf r}(\Sigma)\geq n^{1/2},$ 
one can choose a vector $u$ from the unit ball (for which the supremum is ``nearly attained") such that 
both the bias and the remainder 
are not asymptotically negligible and, moreover, it turns out that, if 
$\frac{{\bf r}(\Sigma)}{n^{1/2}}\to \infty,$ then $\langle f(\hat \Sigma),B\rangle$
is not even a $\sqrt{n}$-consistent estimator of  $\langle f(\Sigma),B\rangle.$
If, in addition, the operator norm $\|\Sigma\|$ is bounded by a constant $R>0,$ one 
can find a function in the space $B^2_{\infty,1}({\mathbb R})$ that coincides 
with $f(x)=x^2$ in a neighborhood of the interval $[0,R],$ and the above 
claims hold for this function, too (see also Remark \ref{rem_Besov_embed} below).

Our first goal is to show that $\langle f(\hat \Sigma),B\rangle$ is an asymptotically normal 
estimator of its own expectation $\langle {\mathbb E}_{\Sigma}f(\hat \Sigma),B\rangle$
with convergence rate $n^{-1/2}$ and limit variance $\sigma_f^2(\Sigma;B)$ in the class 
of covariances with effective rank of the order $o(n).$   
Given $r>1$ and $a>0,$ define
$
{\mathcal G}(r;a):=\Bigl\{\Sigma: {\bf r}(\Sigma)\leq r, \|\Sigma\|\leq a\Bigr\}.
$

\begin{theorem}
\label{Th:normal_approximation_uniform}
Suppose, for some $s\in (1,2],$ $f\in B_{\infty,1}^s({\mathbb R}).$
Let $a>0, \sigma_0>0.$ Suppose that $r_n>1$ and 
$r_n=o(n)$ as $n\to \infty.$
Then 
\begin{equation}
\label{normal_approximation_uniform}
\sup_{\Sigma\in {\mathcal G}(r_n;a), \|B\|_1\leq 1, \sigma_f(\Sigma;B)\geq \sigma_0}\sup_{x\in {\mathbb R}}
\biggl|{\mathbb P}_{\Sigma}\biggl\{\frac{n^{1/2}\langle f(\hat \Sigma)-{\mathbb E}_{\Sigma}f(\hat \Sigma),B\rangle}{\sigma_f(\Sigma;B)}\leq x\biggr\}-\Phi(x)\biggr|\to 0
\end{equation}
as $n\to\infty,$ where $\Phi(x):= \frac{1}{\sqrt{2\pi}}\int_{-\infty}^x e^{-t^2/2}dt, x\in {\mathbb R}.$
\end{theorem}

This result is a consequence of Corollary \ref{corr_f_B_norm} proved in Section \ref{Sec:Normal approximation bounds} that provides an explicit bound on the accuracy of normal approximation.
Its proof is based on a concentration bound for the remainder 
$\langle S_f(\Sigma;\hat \Sigma-\Sigma),B\rangle $ of the first order Taylor expansion 
developed in Section \ref{Sec:Concentration bounds for the remainder}. 
This bound essentially shows that the centered remainder 
$$
\langle S_f(\Sigma;\hat \Sigma-\Sigma),B\rangle - {\mathbb E}\langle S_f(\Sigma;\hat \Sigma-\Sigma),B\rangle
$$ 
is of the order $\Bigl(\frac{{\bf r}(\Sigma)}{n}\Bigr)^{(s-1)/2}\sqrt{\frac{1}{n}},$ which is 
$o(n^{-1/2})$ as soon as ${\bf r}(\Sigma)=o(n).$

Theorem \ref{Th:normal_approximation_uniform} shows that 
the naive plug-in estimator $\langle f(\hat \Sigma),B\rangle$ ``concentrates" around 
its expectation with approximately standard normal distribution of random variables 
$$
\frac{n^{1/2}\langle f(\hat \Sigma)-{\mathbb E}_{\Sigma}f(\hat \Sigma),B\rangle}{\sigma_f(\Sigma;B)}.
$$
At the same time, as we discussed above, the plug-in estimator has a large bias when the effective rank of $\Sigma$ is sufficiently large (say, ${\bf r}(\Sigma)\geq n^{1/2}$ for functions $f$ of smoothness $s=2$).
In the case when $\Sigma\in {\mathcal G}(r_n;a)$
with $r_n=o(n^{1/2})$ and $\sigma_f(\Sigma;B)\geq \sigma_0,$ the bias is negligible and 
$\langle f(\hat \Sigma),B\rangle$ becomes an asymptotically 
normal estimator of $\langle f(\Sigma),B\rangle.$ 
Moreover, we will also derive asymptotics of the risk 
of plug-in estimator for loss functions satisfying the following 
assumption:

\begin{assumption}
\label{assump_loss}
Let $\ell : {\mathbb R}\mapsto {\mathbb R}_+$ be a loss function such that
$\ell(0)=0,$ $\ell(u)=\ell(-u), u\in {\mathbb R},$ $\ell$ is nondecreasing and convex on ${\mathbb R}_+$
and, for some constants $c_1,c_2>0$
$
\ell (u)\leq c_1 e^{c_2 u}, u\geq 0.
$
\end{assumption}

\begin{corollary}
\label{r_small}
Suppose, for some $s\in (1,2],$ $f\in  B_{\infty,1}^s({\mathbb R}).$
Let $a>0, \sigma_0>0.$ Suppose that $r_n>1$ and 
$r_n=o(n^{1-\frac{1}{s}})$ as $n\to \infty.$
Then 
\begin{equation}
\label{normal_approximation_uniform_r_small}
\sup_{\Sigma\in {\mathcal G}(r_n;a), \|B\|_1\leq 1,\sigma_f(\Sigma;B)\geq \sigma_0}\sup_{x\in {\mathbb R}}
\biggl|{\mathbb P}_{\Sigma}\biggl\{\frac{n^{1/2}(\langle f(\hat \Sigma),B\rangle-\langle f(\Sigma),B\rangle)}{\sigma_f(\Sigma;B)}\leq x\biggr\}-\Phi(x)\biggr|\to 0
\end{equation}
as $n\to\infty.$ Moreover, 
under the same assumptions on $f$ and $r_n,$ and for any loss function $\ell$ satisfying Assumption \ref{assump_loss},
\begin{equation}
\label{normal_approximation_loss_f}
\sup_{\Sigma\in {\mathcal G}(r_n;a), \|B\|_1\leq 1, \sigma_f(\Sigma;B)\geq \sigma_0}
\biggl|{\mathbb E}_{\Sigma}\ell\biggl(\frac{n^{1/2}\Bigl(\langle f(\hat \Sigma),B\rangle-\langle f(\Sigma),B\rangle\Bigr)}{\sigma_f(\Sigma;B)}\biggr)-{\mathbb E}\ell(Z)\biggr|\to 0
\end{equation}
as $n\to\infty,$ where $Z$ is a standard normal random variable. 
\end{corollary}

The main difficulty in asymptotically efficient estimation of functional $\langle f(\Sigma),B\rangle$ is related to the development of bias reduction methods. 
We will discuss now an approach to this problem in the case when ${\mathbb H}$
is a finite-dimensional space of dimension $d=d_n=o(n)$ and the covariance operator $\Sigma$ is of isotropic type (the spectrum of $\Sigma$ is bounded from above and bounded 
away from zero by constants that do not depend on $n$). In this case, the effective rank 
${\bf r}(\Sigma)$ is of the same order as the dimension $d,$ so, $d$ will be used as a 
complexity parameter. The development of a similar approach in a more general setting
(when the effective rank ${\bf r}(\Sigma)$ is a relevant complexity parameter) remains 
an open problem. 

Consider the following integral operator 
$$
{\mathcal T}g(\Sigma) := {\mathbb E}_{\Sigma}g(\hat \Sigma)=
\int_{{\mathcal C}_{+}({\mathbb H})}g(S)P(\Sigma;dS), \Sigma \in {\mathcal C}_{+}({\mathbb H})
$$
where ${\mathcal C}_{+}({\mathbb H})$ is the cone of positively semi-definite self-adjoint 
operators in ${\mathbb H}$ (covariance operators) and $P(\Sigma;\cdot)$ is the distribution of the 
sample covariance $\hat \Sigma$ based on $n$ i.i.d. observations sampled from $N(0;\Sigma)$
(which is a rescaled Wishart distribution). In what follows, ${\mathcal T}$ will be called 
{\it the Wishart operator}. We will view it as an operator acting on bounded measurable functions 
on the cone ${\mathcal C}_{+}({\mathbb H})$ taking values either in real line, or in the space 
of self-adjoint operators. Such operators play an important role 
in the theory of Wishart matrices (see, e.g., James \cite{James_60, James, James_64}, Graczyk, Letac and Massam \cite{GLM, GLM-1}, Letac and Massam \cite{Letac-Massam}). Their properties will be discussed in detail in Section \ref{Sec:Wishart}. 
To find an unbiased estimator $g(\hat \Sigma)$ of $f(\Sigma),$ one has to solve 
the integral equation ${\mathcal T}g(\Sigma)=f(\Sigma), \Sigma\in {\mathcal C}_{+}({\mathbb H})$
({\it the Wishart equation}). 
Let ${\mathcal B}:= {\mathcal T}-{\mathcal I},$ ${\mathcal I}$ being the identity operator. 
Then, the solution of Wishart equation can be formally written as the Neumann series
$$
g(\Sigma) = ({\mathcal I}+{\mathcal B})^{-1} f(\Sigma) =
({\mathcal I}-{\mathcal B}+{\mathcal B}^2- \dots ) f(\Sigma)=
\sum_{j=0}^{\infty} (-1)^j {\mathcal B}^j f(\Sigma). 
$$
We do not use this representation in what follows and do not need any facts about 
the convergence of the series. 
Instead, we will define an approximate solution of Wishart equation in terms of a partial sum of 
Neumann series 
$$
f_k(\Sigma):= \sum_{j=0}^k (-1)^j {\mathcal B}^j f(\Sigma), \Sigma\in {\mathcal C}_{+}({\mathbb H}).
$$
With this definition, we have 
$$
{\mathbb E}_{\Sigma} f_k(\hat \Sigma)- f(\Sigma)= (-1)^{k} {\mathcal B}^{k+1}f(\Sigma),
\Sigma\in {\mathcal C}_{+}({\mathbb H}).
$$
It remains to show that $\langle {\mathcal B}^{k+1}f(\Sigma),B\rangle$
is small for smooth enough functions $f,$ which would imply that 
the bias $\langle {\mathbb E}_{\Sigma} f_k(\hat \Sigma)- f(\Sigma),B\rangle$
of estimator $\langle f_k(\hat \Sigma),B\rangle$ of $\langle f(\Sigma),B\rangle$
is also small. [Very recently, a similar approach was considered in the paper by 
Jiao,  Han  and Weissman \cite{Jiao} in the case of estimation of function $f(\theta)$
of the parameter $\theta$ of binomial model $B(n;\theta), \theta\in [0,1].$ In this case, 
${\mathcal T}f$ is the Bernstein polynomial of degree $n$ approximating function $f,$ and some results of classical approximation 
theory (\cite{Gonska}, \cite{Totik}) were used in \cite{Jiao} to control ${\mathcal B}^k f.$]

Note that $P(\cdot;\cdot)$ is a Markov kernel and it could be viewed as the transition 
kernel of a Markov chain $\hat \Sigma^{(t)}, t=0,1,\dots$ in the cone ${\mathcal C}_{+}({\mathbb H}),$
where $\hat \Sigma^{(0)}=\Sigma,$
$\hat \Sigma^{(1)}=\hat \Sigma,$ and, in general, for any $t\geq 1,$ $\hat \Sigma^{(t)}$ is 
the sample covariance based on $n$ i.i.d. observations sampled from the distribution 
$N(0;\hat \Sigma^{(t-1)})$ (conditionally on $\hat \Sigma^{(t-1)}$). In other words, the 
Markov chain $\{\hat \Sigma^{(t)}\}$ is based on iterative applications of bootstrap, 
and it will be called in what follows {\it the bootstrap chain}. 
As a consequence of bound \eqref{operator_hatSigma_exp_dimension}, with a high 
probability (conditionally on $\hat \Sigma^{(t-1)}$),
$
\|\hat \Sigma^{(t)}-\hat \Sigma^{(t-1)}\|\lesssim \|\hat \Sigma^{(t-1)}\|\sqrt{\frac{d}{n}},
$
so, when $d=o(n),$ the Markov chain $\{\hat \Sigma^{(t)}\}$ moves in ``small steps" 
of the order $\asymp \sqrt{\frac{d}{n}}.$
Clearly, with the above definitions, 
$$
{\mathcal T}^k f(\Sigma) = {\mathbb E}_{\Sigma} f(\hat \Sigma^{(k)}).
$$
Note that, by Newton's binomial formula, 
$$
{\mathcal B}^k f(\Sigma)= ({\mathcal T}-{\mathcal I})^k f(\Sigma)
=\sum_{j=0}^k (-1)^{k-j} {k\choose j} {\mathcal T}^{j}f(\Sigma)
= {\mathbb E}_{\Sigma}\sum_{j=0}^k (-1)^{k-j} {k\choose j} f(\hat \Sigma^{(j)}).
$$
The expression $\sum_{j=0}^k (-1)^{k-j} {k\choose j} f(\hat \Sigma^{(j)})$ 
could be viewed as the $k$-th order difference of function $f$ along the Markov 
chain $\{\hat \Sigma^{(t)}\}.$ It is well known that, for a $k$ times continuously 
differentiable function $f$ in real line, the $k$-th order difference $\Delta^k_h f(x)$
(where $\Delta_hf(x):= f(x+h)-f(x)$) is of the order $O(h^k)$ for a small increment $h.$ 
Thus, at least heuristically, one can expect that ${\mathcal B}^k f(\Sigma)$ would 
be of the order $O\Bigl(\Bigl(\frac{d}{n}\Bigr)^{k/2}\Bigr)$ (since $\sqrt{\frac{d}{n}}$
is the size of the ``steps" of the Markov chain $\{\hat \Sigma^{(t)}\}$). This means 
that, for $d$ much smaller than $n,$ one can achieve a significant bias reduction 
in a relatively small number of steps $k.$ The justification 
of this heuristic is rather involved. It is based on a representation 
of operator function $f(\Sigma)$ in the form ${\mathcal D}g(\Sigma):=\Sigma^{1/2}Dg(\Sigma)\Sigma^{1/2},$ where $g$ is a
real valued function on the cone ${\mathcal C}_{+}({\mathbb H})$
invariant with respect to the orthogonal group. The properties 
of orthogonally invariant functions are then used to derive an 
integral representation for the function 
$
{\mathcal B}^k f(\Sigma)={\mathcal B}^k{\mathcal D}g(\Sigma) = {\mathcal D}{\mathcal B}^k g(\Sigma)
$
that implies, for a sufficiently smooth $f,$ bounds on ${\mathcal B}^k f(\Sigma)$ of the order 
$O\Bigl(\Bigl(\frac{d}{n}\Bigr)^{k/2}\Bigr)$ 
and, as a consequence, bounds on the bias of estimator 
$\langle f_k(\hat \Sigma),B\rangle$ of $\langle f(\Sigma),B\rangle$
of the order $o(n^{-1/2}),$ provided that $d=o(n)$ and $k$ is sufficiently large
(see \eqref{B_k_repres} in Section \ref{Sec:Wishart} and Theorem \ref{bias-iter},
Corollary \ref{Bias_bound} in Section \ref{Sec:bias-iter}). 

The next step in analysis of estimator $\langle f_k(\hat \Sigma),B\rangle$
is to derive normal approximation bounds for 
$\langle f_k(\hat \Sigma),B\rangle-{\mathbb E}_{\Sigma}\langle f_k(\hat \Sigma),B\rangle.$
To this end, we study in Section \ref{Sec:smooth} smoothness properties of functions 
${\mathcal D} {\mathcal B}^k g(\Sigma)$ for a smooth orthogonally invariant function $g$
that are later used to prove proper smoothness
of such functions as $\langle f_k(\Sigma), B\rangle$ and derive concentration bounds on the 
remainder $\langle S_{f_k}(\Sigma;\hat \Sigma-\Sigma), B\rangle$ of the first order Taylor expansion 
of $\langle f_k(\hat \Sigma),B\rangle,$ which is 
the main step in showing that the centered remainder is asymptotically negligible and proving the normal approximation. 
In addition, we show that the limit variance in the normal approximation of 
$\langle f_k(\hat \Sigma),B\rangle-{\mathbb E}_{\Sigma}\langle f_k(\hat \Sigma),B\rangle$ 
coincides with $\sigma_f^2(\Sigma;B)$
(which is exactly the same as the limit variance in the normal approximation of $\langle f(\hat \Sigma),B\rangle-{\mathbb E}_{\Sigma}\langle f(\hat \Sigma),B\rangle$). 
This, finally, yields normal approximation bounds of theorems \ref{th-main-inv} and \ref{th-main-function}  in Section \ref{Sec:norm_appr}.

Given $d>1$ and $a\geq 1,$ denote by 
${\mathcal S}(d;a)$
the set of all covariance 
operators in a $d$-dimensional space ${\mathbb H}$ such that 
$\|\Sigma\|\leq a,$ $\|\Sigma^{-1}\|\leq a.$ 
The following result on uniform normal approximation of 
estimator $\langle f_k(\hat \Sigma),B\rangle$ of $\langle f(\Sigma),B\rangle$ is an immediate 
consequence of Theorem \ref{th-main-function}.

\begin{theorem}
\label{th-main-function-uniform}
Let $a\geq 1, \sigma_0>0.$
Suppose that, for some $\alpha \in (0,1),$ 
$
1<d_n \leq n^{\alpha}, n\geq 1.
$
Suppose also that, for some $s> \frac{1}{1-\alpha},$
$f\in B_{\infty,1}^{s}({\mathbb R}).$ 
Let $k$ be an integer number such that, for some $\beta\in (0,1],$  
$\frac{1}{1-\alpha}<k+1+\beta\leq s.$ 
Then 
\begin{equation}
\label{normal_approximation_uniform_reduced_bias}
\sup_{\Sigma\in {\mathcal S}(d_n;a), \|B\|_1\leq 1, \sigma_f(\Sigma;B)\geq \sigma_0}\sup_{x\in {\mathbb R}}
\biggl|{\mathbb P}_{\Sigma}\biggl\{\frac{n^{1/2}\Bigl(\langle f_k(\hat \Sigma),B\rangle-\langle f(\Sigma),B\rangle\Bigr)}{\sigma_f(\Sigma;B)}\leq x\biggr\}-\Phi(x)\biggr|\to 0
\end{equation}
as $n\to\infty.$ Moreover, if $\ell$ is a loss function satisfying Assumption \ref{assump_loss},
then 
\begin{equation}
\label{normal_approximation_loss}
\sup_{\Sigma\in {\mathcal S}(d_n;a), \|B\|_1\leq 1, \sigma_f(\Sigma;B)\geq \sigma_0}
\biggl|{\mathbb E}_{\Sigma}\ell\biggl(\frac{n^{1/2}\Bigl(\langle f_k(\hat \Sigma),B\rangle-\langle f(\Sigma),B\rangle\Bigr)}{\sigma_f(\Sigma;B)}\biggr)-{\mathbb E}\ell(Z)\biggr|\to 0
\end{equation}
as $n\to\infty.$
\end{theorem}

\begin{remark}
Note that for $\alpha\in (0,1/2)$ and $s>\frac{1}{1-\alpha},$ one can choose 
$k=0,$ implying that $f_k(\hat \Sigma)=f(\hat \Sigma)$ in Theorem \ref{th-main-function-uniform} is a usual plug-in 
estimator (compare this with Corollary \ref{r_small}).
However, for $\alpha=\frac{1}{2},$ we have to assume that $s>2$
and choose $k=1$ to satisfy the condition $k+1+\beta>\frac{1}{1-\alpha}=2.$
Thus, in this case, the bias correction is already nontrivial. For larger values 
of $\alpha,$ even more smoothness of $f$ is required and more iterations $k$
in our bias reduction method is needed. 
\end{remark}

\begin{remark}
\label{rem_Besov_embed}
It easily follows from well known embedding theorems for Besov spaces 
(see, e.g., \cite{Triebel}, Section 2.3.2) that, for $s'>s>0,$ 
the H\"older space $C^{s'}({\mathbb R})\subset B^{s}_{\infty,1}({\mathbb R}).$
Moreover, it is easy to see that any $C^{s'}$-function defined locally in a neighborhood 
of the spectrum of $\Sigma$ could be extended to a function from $C^{s'}({\mathbb R}).$
These observations show that Theorem \ref{th-main-function-uniform} 
could be applied to all $C^{s}$ functions defined in a neighborhood of the spectrum 
of $\Sigma$ for all $s>\frac{1}{1-\alpha}.$ 
\end{remark}

To show the asymptotic efficiency of estimator $\langle f_k(\hat \Sigma),B\rangle,$ it remains 
to prove a minimax lower bound on the risk of an arbitrary estimator $T_n(X_1,\dots, X_n)$ of the functional $\langle f(\Sigma),B\rangle $ that would imply the optimality of the variance $\sigma_f^2(\Sigma;B)$ in normal 
approximation \eqref{normal_approximation_uniform_reduced_bias}, 
\eqref{normal_approximation_loss}.
Consider a function $f\in B_{\infty,1}^s({\mathbb R})$ for some $s\in (1,2].$
Given $a>1,$ let $\mathring{{\mathcal S}}(d;a)$ be the set of all covariance operators 
in a Hilbert space ${\mathbb H}$ of dimension $d={\rm dim}({\mathbb H})$ such that 
$\|\Sigma\|<a, \|\Sigma^{-1}\|<a.$ Given $\sigma_0>0,$ denote 
$$
\mathring{{\mathcal S}}_{f,B}(d;a;\sigma_0):= \mathring{{\mathcal S}}(d;a)\cap \{\Sigma: \sigma_f(\Sigma;B)>\sigma_0\}.
$$ 
Note that the set $\mathring{{\mathcal S}}_{f,B}(d;a;\sigma_0)$ is open in operator norm topology,
which easily follows from the continuity of functions $\Sigma \mapsto \|\Sigma\|,$ 
$\Sigma \mapsto \|\Sigma^{-1}\|$ (on the set of non-singular operators) and 
$\Sigma \mapsto \sigma_f^2(\Sigma;B)$ (see Lemma \ref{Lemma:sigmaf_lipsch} in Section \ref{Sec:Lower bounds}) 
with respect to the operator norm. This set could be empty. For instance, since 
$
\sigma^2_f(\Sigma;B) \leq 2\|f'\|_{L_{\infty}}^2 \|\Sigma\|^2  \|B\|_2^2,
$
we have that $\mathring{{\mathcal S}}_{f,B}(d;a;\sigma_0)=\emptyset$ if $\sigma_0^2 > 2\|f'\|_{L_{\infty}}^2 \|\Sigma\|^2  \|B\|_2^2.$ Denote 
$$
{\frak B}_f(d;a;\sigma_0):=\Bigl\{B: \|B\|_1\leq 1, \mathring{{\mathcal S}}_{f,B}(d;a;\sigma_0)\neq \emptyset\Bigr\}.
$$

The following theorem provides an asymptotic minimax lower bound on the mean squared error of estimation of functionals 
$\langle f(\Sigma), B\rangle, \|B\|_1\leq 1.$ By convention, it will be assumed that 
$\inf \emptyset =+\infty.$

\begin{theorem}
\label{th:main_lower_bound}
Let $a>1,$ $\sigma_0^2>0$ and let $\{d_n\}$ be an arbitrary sequence 
of integers $d_n\geq 2.$
Then, for all $a'\in (1,a)$ and $\sigma_0'>\sigma_0,$ 
\begin{equation}
\label{main_lower_bound}
\liminf_{n\to\infty} \inf_{T_n} \inf_{B\in {\frak B}_f(d_n;a';\sigma_0'
)}\sup_{\Sigma\in \mathring{{\mathcal S}}(d_n;a), \sigma_f(\Sigma;B)>\sigma_0}
\frac{n {\mathbb E}_{\Sigma}(T_n-\langle f(\Sigma),B\rangle)^2}{\sigma_f^2(\Sigma;B)}
\geq 1,
\end{equation}
where the first infimum is taken over all statistics $T_n=T_n(X_1,\dots, X_n)$ based on i.i.d. observations
$X_1,\dots, X_n$ sampled from $N(0;\Sigma).$ 
\end{theorem}

The proof of this theorem is given in Section \ref{Sec:Lower bounds}.

\begin{remark}
If $C\subset \sigma(\Sigma)$ is a ``component" of the spectrum of $\Sigma$ 
such that the distance ${\rm dist}(C;\sigma(\Sigma)\setminus C)$ from $C$ to
the rest of the spectrum is bounded away from zero by a sufficiently large gap
and $P_C$ is the orthogonal projection 
on the direct sum of eigenspaces of $\Sigma$ corresponding to the eigenvalues from $C,$
then it is easy to represent $P_C$ as $f(\Sigma)$ for a smooth function $f$ that is equal to 
$1$ on $C$ and vanishes outside of a neighborhood of $C$ that does not contain other eigenvalues.
The problem of efficient estimation of linear functionals of spectral projection $P_C$ 
(such as its matrix entries in a given basis or general bilinear forms) is of importance 
in principal component analysis. A related problem of estimation 
of linear functionals of principal components was recently studied in \cite{Koltchinskii_Nickl} 
in the case of one-dimensional spectral projections. The methods of efficient estimation 
developed in \cite{Koltchinskii_Nickl} are rather specialized and they could not be easily 
extended even to spectral projections of higher rank than $1.$ This, in part, was our motivation 
to study the problem for more general smooth functionals and to develop a more general
approach to the problem of efficient estimation. Similarly, one can represent operator $P_C \Sigma P_C$
as a smooth function of $\Sigma$ and use the approach of the current paper to develop efficient  
estimators of bilinear forms or matrix entries of such operators. This could be of interest in the 
case of covariance matrices of the form $\Sigma=\Sigma_0+\sigma^2 I_d,$ where $\Sigma_0$
is a low rank covariance matrix (say, the covariance matrix whose eigenvectors are ``spikes"
of a spiked covariance model). If $C$ is the set of top eigenvalues of $\Sigma$ that correspond 
to its ``spikes", then estimation of matrix $\Sigma_0$ could be reduced to estimation of $P_C \Sigma P_C.$ 
\end{remark}

\begin{remark}
The results of the paper could not be directly applied to estimation of functionals of the form 
${\rm tr}(f(\Sigma))$ since in this case $B$ is the identity operator and its nuclear norm is not 
bounded by a constant. In such cases, $\sqrt{n}$-consistent estimators do not always exist 
in high-dimensional problems, minimax optimal convergence rates are slower than $n^{-1/2}$
and they do depend on the dimension (see, for instance, \cite{Cai_Liang_Zhou} for an example 
of estimation of the log-determinant $\log {\rm det}(\Sigma)={\rm tr}(\log (\Sigma))$). Although 
some elements of our approach (in particular, the bias reduction method) could be useful in this 
case, a comprehensive theory of estimation of functionals $\langle f(\Sigma), B\rangle$
in the case of unbounded nuclear norm of operator $B$ remains an open problem and it is beyond the 
scope of this paper.   
\end{remark}

\begin{remark}
In this paper, the problem was studied only in the case of Gaussian models with known mean 
(without loss of generality, it is set to be zero) and unknown covariance operators. 
In \cite{Koltchinskii_Zhilova}, a similar problem of efficient estimation of smooth functionals 
of unknown mean in Gaussian shift models with known 
covariance was studied. 
The problem becomes more complicated when both mean and covariance are unknown
(in particular, it would require a more difficult analysis of operators ${\mathcal T}$ and ${\mathcal B}$ involved in the bias reduction method). 
\end{remark}

\begin{remark}
The computation of estimators $f_k(\hat \Sigma)$ could be based on a Monte Carlo simulation
of the bootstrap chain. To this end, one has to simulate a segment of this chain of length $k+1$
starting at the sample covariance $\hat \Sigma.$ This would allow us to compute the sum  
$\sum_{j=0}^k (-1)^{k-j} {k\choose j} f(\hat \Sigma^{(j+1)}).$ Averaging such sums over a sufficiently 
large number $N$ of independent copies of bootstrap chain provides a Monte Carlo approximation  
of ${\mathcal B}^{k}f(\hat \Sigma),$ which allows us to approximate $f_k(\hat \Sigma).$ 
A total of $(k+1)N$ computations of the function $f$ of covariance operators (each of them based on a singular value decomposition) would be required to implement this procedure.
\end{remark}

\subsection{Related results}

Up to our best knowledge, the problem of efficient estimation for general classes of smooth 
functionals of covariance operators in the setting of the current paper has not been studied before.
However, many results in the literature on nonparametric, semiparametric and high-dimensional 
statistics as well as some results in random matrix theory are relevant in our context. 
We provide below a very brief discussion of some of these results.

Asymptotically efficient estimation of smooth functionals of infinite-dimensional parameters 
has been an important topic in nonparametric statistics for a number of years that 
also has deep connections to efficiency in semiparametric estimation  
(see, e.g., \cite{BKRW}, \cite{Gine_Nickl} and references therein). 
The early references include Levit \cite{Levit_1, Levit_2} and the book of 
Ibragimov and Khasminskii \cite{Ibragimov}. 
In the paper by Ibragimov, Nemirovski and Khasminskii \cite{Ibragimov_Khasm_Nemirov}
and later in paper \cite{Nemirovski_1990} and in Saint-Flour Lectures \cite{Nemirovski} by Nemirovski,  
sharp results on efficient estimation of general smooth functionals of parameters of  
Gaussian white noise models were obtained, precisely describing the dependence between 
the rate of decay of Kolmogorov's diameters of parameter space (used as a measure of its complexity)
and the degree of smoothness of functionals for which efficient estimation is possible.
A general approach to construction of efficient estimators of smooth 
functionals in Gaussian white noise models was also developed in these papers. The result of Theorem \ref{th-main-function-uniform}
is in the same spirit with the growth rate $\alpha$ of the dimension of the space being the complexity
parameter instead of the rate of decay of Kolmogorov's diameters. At this point, we do not know whether the smoothness threshold $s>\frac{1}{1-\alpha}$ for efficient estimation obtained 
in this theorem is sharp (although the sharpness of the same smoothness threshold was proved 
in \cite{Koltchinskii_Zhilova}) in the case of Gaussian shift model).

More recently, there has been a lot of interest 
in semi-parametric efficiency properties of regularization-based estimators 
(such as LASSO) in various models of high-dimensional statistics, see, e.g., \cite{vdgBuhlmannRitovDezeureAOS}, \cite{Montanari}, \cite{Zhang_Zhang}, 
\cite{Jankovavdg} as well as in minimax optimal rates of estimation 
of special functionals (in particular, linear and quadratic) in such models 
\cite{Cai_Low_2005a}, \cite{Cai_Low_2005b}, \cite{C_C_Tsybakov}.

In a series of pioneering papers in the 80s--90s, Girko obtained a number of results on asymptotically normal estimation of many special functionals of covariance matrices in high-dimensional setting, in particular, on estimation of the Stieltjes transform of spectral function 
${\rm tr}((I+t\Sigma)^{-1})$ (see \cite{Girko} and also \cite{Girko-1}
and references therein). His estimators were typically functions of sample covariance $\hat \Sigma$ defined in terms of certain equations 
(so called $G$-estimators) and the proofs of their asymptotic normality were largely based 
on martingale CLT. The centering and normalizing parameters in the limit theorems in these papers 
are often hard to interpret and the estimators were not proved to be asymptotically efficient.   

Asymptotic normality of so called linear spectral statistics ${\rm tr}(f(\hat \Sigma))$  
centered either by their own expectations, or by the integral of $f$ with respect to Marchenko-Pastur
type law has been an active subject of research in random matrix theory both in the 
case of high-dimensional sample covariance (or Wishart matrices) and in other random matrix models such as Wigner matrices, see, e.g., Bai and Silverstein \cite{Bai}, Lytova and Pastur \cite{Lytova},
Sosoe and Wong \cite{Sosoe}. Although these results do not have direct statistical implications since 
${\rm tr}(f(\hat \Sigma))$ does not ``concentrate" around the corresponding population parameter,
probabilistic and analytic techniques developed in these papers are highly relevant. 

There are many results in the literature on special cases of the above problem, such as asymptotic normality of statistic 
$\log {\rm det}(\hat \Sigma)={\rm tr}(\log (\hat \Sigma))$ (the log-determinant). 
If $d=d_n\leq n,$ then it was shown that the sequence 
$$\frac{\log {\rm det}(\hat \Sigma)-a_{n,d}-\log {\rm det}(\Sigma)}{b_{n,d}}$$ 
converges in distribution to a standard normal random variable for explicitly given sequences $a_{n,d}, b_{n,d}$ that depend only on the sample size $n$ and on the dimension $d.$ This means that $\log {\rm det}(\hat \Sigma)$ is an asymptotically normal estimator of $\log {\rm det}(\Sigma)={\rm tr}(\log (\Sigma))$ subject to a simple bias correction (see, e.g., Girko \cite{Girko} and more recent paper by 
Cai, Liang and Zhou \cite{Cai_Liang_Zhou}). The convergence rate of this estimator is typically slower than 
$n^{-1/2}:$ for instance, if $d=n^{\alpha}$ for $\alpha \in (0,1),$ then the convergence rate is $\asymp n^{-(1-\alpha)/2}$
(and, for $\alpha=1,$ the estimator is not consistent). 
In this case, the problem is relatively simple since 
$\log {\rm det}(\hat \Sigma)-\log {\rm det}(\Sigma)=\log {\rm det}(W),$ where $W$ is the sample covariance based on a sample of $n$ i.i.d.   
standard normal random vectors. 
 
In a recent paper by Koltchinskii and Lounici \cite{Koltchinskii_Lounici_bilinear}
(see also \cite{Koltchinskii_Lounici_AOS, Koltchinskii_Lounici_Sankhya}),
the problem of estimation of bilinear forms of spectral projections of covariance operators
was studied in the setting when the effective rank ${\bf r}(\Sigma)=o(n)$ as $n\to \infty.$
\footnote{For other recent results on covariance estimation under assumptions on its effective rank see \cite{NaumovSpokoinyUlyanov, ReissWahl}.}
Normal approximation and concentration results for bilinear forms centered by their expectations 
were proved using first order perturbation expansions for empirical spectral projections
and concentration inequalities for their remainder terms (which is similar to the approach 
of the current paper). Special properties of the bias of these estimators were studied
that, in the case of one dimensional spectral projections, led to the development of 
a bias reduction method based on sample splitting that resulted 
in a construction of $\sqrt{n}$-consistent 
and asymptotically normal estimators of linear forms of eigenvectors of the true covariance 
(principal components) in the case when ${\bf r}(\Sigma)=o(n)$ as $n\to \infty.$ 
This approach has been further developed in a very recent paper by Koltchinskii, Loeffler and Nickl \cite{Koltchinskii_Nickl} in which asymptotically efficient estimators of linear forms of eigenvectors  
of $\Sigma$ were studied. 

Other recent references on estimation of functionals of covariance include Fan, Rigollet and Wang \cite{Fan} (optimal rates 
of estimation of special functionals of covariance under sparsity assumptions), Gao and Zhou \cite{Gao}
(Bernstein-von Mises theorems for functionals of covariance), Kong and Valiant \cite{Kong_Valiant}
(estimation of ``spectral moments" ${\rm tr}(\Sigma^k)$).

\section{Analysis and operator theory preliminaries}\label{Sec:Entire}

In this section, we discuss several results in operator theory concerning 
perturbations of smooth functions of self-adjoint operators in Hilbert spaces.
They are simple modifications of known results due to several authors 
(see recent survey by Aleksandrov and Peller \cite{Aleksandrov_Peller}).

\subsection{Entire functions of exponential type and Besov spaces}

Let $f:{\mathbb C}\mapsto {\mathbb C}$ be an entire function and let $\sigma >0.$
It is said that $f$ is of exponential type $\sigma$ (more precisely, $\leq \sigma$)
if for any $\eps>0$ there exists $C=C(\eps,\sigma, f)>0$ such that 
$$
|f(z)|\leq Ce^{(\sigma+\eps)|z|}, z\in {\mathbb C}. 
$$ 
In what follows, ${\mathcal E}_{\sigma}={\mathcal E}_{\sigma}({\mathbb C})$
denotes the space of all entire functions of exponential type $\sigma.$
It is straightforward to see (and well known) that $f\in {\mathcal E}_{\sigma}$ if and only if 
$$
\limsup_{R\to\infty} \frac{\log \sup_{\varphi\in [0,2\pi]} |f(Re^{i\varphi})|}{R}
=:\sigma(f)\leq \sigma.
$$
With a little abuse of notation, the restriction $f_{\restriction {\mathbb R}}$ of function $f$ 
to ${\mathbb R}$ will be also denoted by $f;$ ${\mathcal F}f$ will denote the Fourier 
transform of $f: {\mathcal F} f(t)=\int_{\mathbb R}e^{-itx}f(x)dx$ (if $f$ is not square integrable, 
its Fourier transform could be understood in the sense of tempered distributions). 
According to {\it Paley-Wiener theorem},  
$$
{\mathcal E}_{\sigma}\bigcap L_{\infty}({\mathbb R}) =
\{f\in L_{\infty}({\mathbb R}): {\rm supp}({\mathcal F}f)\subset [-\sigma,\sigma]\}.
$$
It is also well known that $f\in {\mathcal E}_{\sigma}\bigcap L_{\infty}({\mathbb R})$
if and only if $|f(z)|\leq \|f\|_{L_{\infty}({\mathbb R})}e^{\sigma |{\rm Im}(z)|}, z\in {\mathbb C}.$

We will use the following {\it Bernstein inequality} 
$
\|f^{\prime}\|_{L_{\infty}({\mathbb R})} \leq \sigma \|f\|_{L_{\infty}({\mathbb R})}
$
that holds for all functions $f\in {\mathcal E}_{\sigma}\bigcap L_{\infty}({\mathbb R}).$
Moreover, since $f\in {\mathcal E}_{\sigma}$ implies $f^{\prime}\in {\mathcal E}_{\sigma},$ 
we also have 
$
\|f^{\prime \prime}\|_{L_{\infty}({\mathbb R})} \leq \sigma^2 \|f\|_{L_{\infty}({\mathbb R})},
$
and similar bounds hold for all the derivatives of $f.$ Next elementary lemma 
is a corollary of Bernstein inequality. It provides bounds on the remainder of the 
first order Taylor expansion of a function $f\in {\mathcal E}_{\sigma}\bigcap L_{\infty}({\mathbb R}).$

\begin{lemma}
\label{taylor_remainder}
Let $f\in {\mathcal E}_{\sigma}\bigcap L_{\infty}({\mathbb R}).$ Denote  
$$
S_f(x;h):= f(x+h)-f(x)-f^{\prime}(x)h, x,h\in {\mathbb R}.
$$
Then  
$$
|S_f(x;h)| \leq \frac{\sigma^2}{2}\|f\|_{L_{\infty}({\mathbb R})} h^2,\ x,h\in {\mathbb R}
$$
and 
$$
|S_f(x;h')-S_f(x;h)|\leq \sigma^2\|f\|_{L_{\infty}({\mathbb R})}
\delta (h,h')|h'-h|,\ x,h,h'\in {\mathbb R}.
$$
where
$
\delta(h,h'):=(|h|\wedge |h'|)+ \frac{|h'-h|}{2}.
$
\end{lemma}

We also need an extension of Bernstein inequality to functions 
of many complex variables. Let $f:{\mathbb C}^k \mapsto {\mathbb C}$
be an entire function and let $\sigma:=(\sigma_1,\dots, \sigma_k),$
$\sigma_j> 0.$ Function $f$ is of exponential type $\sigma=(\sigma_1,\dots, \sigma_k)$ 
if for any $\eps>0$ there exists $C=C(\eps,\sigma, f)>0$ such that 
$$
|f(z_1,\dots, z_k)|\leq Ce^{\sum_{j=1}^k(\sigma_j+\eps)|z_j|}, z_1,\dots, z_k\in {\mathbb C}.
$$ 
Let ${\mathcal E}_{\sigma_1,\dots, \sigma_k}$ be the set of all such functions. 
The following extension of Bernstein inequality could be found in the paper by Nikolsky \cite{Nikolsky},
who actually proved it for an arbitrary $L_p$-norm, $1\leq p\leq \infty.$
If $f\in {\mathcal E}_{\sigma_1,\dots, \sigma_k}\cap L_{\infty}({\mathbb R}),$
then for any $m\geq 0$ and any $m_1,\dots, m_k\geq 0$ such that 
$\sum_{j=1}^km_j=m,$
\begin{equation}
\label{Bernstein-Nikolsky}
\biggl\|\frac{\partial^m f}{\partial x_1^{m_1}\dots \partial x_k^{m_k}}\biggr\|_{L_{\infty}({\mathbb R}^k)} \leq \sigma_1^{m_1}\dots \sigma_k^{m_k} 
\|f\|_{L_{\infty}({\mathbb R}^k)}.
\end{equation}

Let $w\geq 0$ be a $C^{\infty}$ function 
in real line with ${\rm supp}(w)\subset [-2,2]$ such that 
$w(t)=1, t\in [-1,1]$ and $w(-t)=w(t), t\in {\mathbb R}.$
Define $w_0(t):= w(t/2)-w(t), t\in {\mathbb R}$
which implies that ${\rm supp}(w_0)\subset \{t: 1\leq |t|\leq 4\}.$ Let $w_j(t):= w_0(2^{-j}t), 
t\in {\mathbb R}$ with ${\rm supp}(w_j)\subset \{t: 2^{j}\leq |t|\leq 2^{j+2}\},$
$j=0,1,\dots .$ These definitions immediately imply that 
$$
w(t)+ \sum_{j\geq 0}w_j(t)= 1, t\in {\mathbb R}.
$$
Finally, define functions $W, W_j\in {\mathcal S}({\mathbb R})$ (the Schwartz space of functions in ${\mathbb R}$) by their Fourier 
transforms as follows:
$$
w(t)=({\mathcal F}W)(t),\ w_j(t)= ({\mathcal F}W_j)(t), t\in {\mathbb R}, j\geq 0.
$$
For a tempered distribution $f\in {\mathcal S}'({\mathbb R}),$ one can define 
its Littlewood-Paley dyadic decomposition as the family of functions 
$f_0:= f\ast W, \ 
f_{n}: = f\ast W_{n-1}, n\geq 1$ with compactly supported 
Fourier transforms. Note that, by Paley-Wiener theorem, 
$f_n\in {\mathcal E}_{2^{n+1}}\bigcap L_{\infty}({\mathbb R}).$ 
It is well known that $\sum_{n\geq 0}f_n=f$ with convergence of the 
series in the space ${\mathcal S}'({\mathbb R}).$
We use the following Besov norms
$$
\|f\|_{B_{\infty,1}^s}:= \sum_{n\geq 0}2^{ns}\|f_n\|_{L_{\infty}({\mathbb R})}, s\in {\mathbb R}
$$
and define the corresponding Besov spaces as 
$$
B_{\infty,1}^s({\mathbb R}):=\Bigl\{f\in {\mathcal S}'({\mathbb R}):\|f\|_{B_{\infty,1}^s}<\infty\Bigr\}.
$$
We do not use in what follows the whole scale of Besov spaces 
$B_{p,q}^s({\mathbb R}),$ but only the spaces $B_{\infty,1}^s({\mathbb R})$
for $p=\infty, q=1$ and $s\geq 0.$ 
Note that Besov norms $\|\cdot\|_{B_{p,q}^{s}}$ are equivalent 
for different choices of function $w$ and the corresponding Besov spaces 
coincide. If $f\in B_{\infty,1}^s({\mathbb R})$ for some $s\geq 0,$
then the series $\sum_{n\geq 0}f_n$ converges uniformly to $f$ in ${\mathbb R},$
which easily implies that $f\in C_u({\mathbb R}),$ 
where $C_u({\mathbb R})$ is the space of all bounded uniformly continuous 
functions in ${\mathbb R}$ and $\|f\|_{L_{\infty}}\leq \|f\|_{B_{\infty,1}^s}.$
Thus, for $s\geq 0,$ the space $B_{\infty,1}^s({\mathbb R})$ is continuously 
embedded in $C_u({\mathbb R}).$ Moreover, if $C^s({\mathbb R})$ denotes the 
H\"older space of smoothness $s>0,$ then, for all $s'>s>0,$ 
$C^{s'}({\mathbb R})\subset B_{\infty,1}^s({\mathbb R})\subset C^s({\mathbb R})$
(see \cite{Triebel}, section 2.3.2, 2.5.7). Further details on Besov spaces could be also 
found in \cite{Triebel}.

\subsection{Taylor expansions for operator functions}

For a continuous (and even for a Borel measurable)
function $f$ in ${\mathbb R}$ and $A\in {\mathcal B}_{sa}({\mathbb H}),$ the operator 
$f(A)$ is well defined and self-adjoint (for instance, by the spectral theorem). 
By a standard holomorphic functional calculus, the operator $f(A)$ is well 
defined for $A\in {\mathcal B}({\mathbb H})$ and for any function $f:G\subset {\mathbb C}\mapsto{\mathbb C}$ holomorphic in a neighborhood $G$
of the spectrum $\sigma(A)$ of $A.$ It is given by the following Cauchy 
formula:
$$
f(A):= -\frac{1}{2\pi i}\oint_{\gamma} f(z)R_A(z)dz,
$$
where $R_A(z):=(A-zI)^{-1}, z\not\in \sigma(A)$ is the resolvent of $A$
and $\gamma\subset G$ is a contour surrounding $\sigma(A)$ with a counterclockwise 
orientation. In particular, this holds for all entire functions $f$ and the mapping 
${\mathcal B}({\mathbb H})\ni A\mapsto f(A)\in {\mathcal B}({\mathbb H})$ is Fr\'echet differentiable with derivative 
\begin{equation}
\label{holomorphic}
Df(A;H)= \frac{1}{2\pi i}\oint_{\gamma} f(z) R_A(z)HR_A(z)dz, H\in {\mathcal B}({\mathbb H}).
\end{equation}
The last formula easily follows from the perturbation series for the resolvent 
$$
R_{A+H}(z)= \sum_{k=0}^{\infty} (-1)^k (R_A(z)H)^k R_A(z), z\in {\mathbb C}\setminus \sigma(A)
$$
that converges in the operator norm as soon as 
$
\|H\|< \frac{1}{\|R_A(z)\|}=\frac{1}{{\rm dist}(z,\sigma(A))}.
$

We need to extend bounds of Lemma \ref{taylor_remainder} to functions of operators 
establishing similar properties for the remainder of the first order Taylor expansion
$$
S_f(A;H) := f(A+H)-f(A)-Df(A;H), A,H\in {\mathcal B}_{sa}({\mathbb H}),
$$
where $f$ is an entire function of exponential type $\sigma. $ 
This is related to a circle of problems studied in operator theory 
literature concerning operator Lipschitz and operator differentiable 
functions (see, in particular, a survey on this subject by Aleksandrov and Peller \cite{Aleksandrov_Peller}). 

We will need the following lemma.

\begin{lemma}
\label{taylor_operator}
Let $f\in {\mathcal E}_{\sigma}\bigcap L_{\infty}({\mathbb R}).$
Then, for all $A,H, H'\in  {\mathcal B}_{sa}({\mathbb H}),$
\begin{equation}
\label{oper_Lipschitz}
\|f(A+H)-f(A)\|\leq \sigma \|f\|_{L_{\infty}({\mathbb R})}\|H\|,
\end{equation}
\begin{equation}
\label{derivA}
\|Df(A;H)\|\leq 
\sigma \|f\|_{L_{\infty}({\mathbb R})}\|H\|,
\end{equation}
\begin{equation}
\label{first_order}
\|S_f(A;H)\|\leq \frac{\sigma^2}{2}\|f\|_{L_{\infty}({\mathbb R})}\|H\|^2
\end{equation}
and 
\begin{equation}
\label{first_order_Lipschitz}
\|S_f(A;H')-S_f(A;H)\|\leq \sigma^2\|f\|_{L_{\infty}({\mathbb R})}
\delta(H,H')\|H'-H\|,
\end{equation}
where
$
\delta(H,H'):= (\|H\|\wedge \|H'\|)+ \frac{\|H'-H\|}{2}.
$
\end{lemma}

Bound (\ref{oper_Lipschitz}) and (\ref{derivA}) are well known, see Aleksandrov and 
Peller \cite{Aleksandrov_Peller} (in fact, bound (\ref{oper_Lipschitz}) means that, for
$f\in {\mathcal E}_{\sigma}\bigcap L_{\infty}({\mathbb R}),$ ${\mathcal B}_{sa}({\mathbb H})\ni A\mapsto f(A)\in {\mathcal B}_{sa}({\mathbb H})$ is operator Lipschitz 
with respect to the operator norm). The proof of bounds (\ref{first_order}) and (\ref{first_order_Lipschitz}) is also based on a very nice approach by Aleksandrov and Peller \cite{Aleksandrov_Peller_2010, Aleksandrov_Peller} developed to prove the operator Lipschitz property. We are giving this proof for completeness.

\begin{proof} Let $E$ be a complex Banach space and let ${\mathcal E}_{\sigma}(E)$
be the space of entire functions $F:{\mathbb C}\mapsto E$ of exponential type 
$\sigma,$ that is, entire functions $F$ such that for any $\eps>0$ there exists a constant $C=C(\eps,\sigma,F)>0$ for which
$
\|F(z)\|\leq Ce^{(\sigma+\eps)|z|}, z\in {\mathbb C}.
$
If $F\in {\mathcal E}_{\sigma}(E)$ and $\sup_{x\in {\mathbb R}}\|F(x)\|<+\infty,$
then Bernstein inequality holds for function $F:$
\begin{equation}
\label{F_one}
\sup_{x\in {\mathbb R}}\|F^{\prime}(x)\|\leq \sigma \sup_{x\in {\mathbb R}}\|F(x)\|.
\end{equation}
Indeed, for any $l \in E^{\ast},$ $l(F(\cdot))\in {\mathcal E}_{\sigma}\bigcap L_{\infty}({\mathbb R}),$ which implies that  
$$
\sup_{x\in {\mathbb R}}\|F^{\prime}(x)\|=
\sup_{\|l\|\leq 1}\sup_{x\in {\mathbb R}}|l(F^{\prime}(x))| \leq \sigma \sup_{\|l\|\leq 1}
\sup_{x\in {\mathbb R}}|l(F(x))|=\sigma \sup_{x\in {\mathbb R}}\|F(x)\|
$$
and 
\begin{equation}
\label{F_two}
\|F(x+h)-F(x)\|\leq \sigma \sup_{x\in {\mathbb R}}\|F(x)\||h|.
\end{equation}
A similar simple argument (now based on Lemma \ref{taylor_remainder}) shows that for $S_F(x;h):=F(x+h)-F(x)-F^{\prime}(x)h,$
we have
\begin{equation}
\label{F_three}
\|S_F(x;h)\| \leq \frac{\sigma^2}{2} \sup_{x\in {\mathbb R}}\|F(x)\|h^2,\ x,h\in {\mathbb R}
\end{equation}
and 
\begin{equation}
\label{F_four}
\|S_F(x;h')-S_F(x;h)\|\leq \sigma^2 \sup_{x\in {\mathbb R}}\|F(x)\|
\biggl[|h||h'-h|+ \frac{|h'-h|^2}{2}\biggr],\ x,h,h'\in {\mathbb R}.
\end{equation}

Next, for given $A,H\in {\mathcal B}_{sa}({\mathbb H})$ and 
$f\in {\mathcal E}_{\sigma}\bigcap L_{\infty}({\mathbb R}),$
define 
$F(z):=f(A+zH), z\in {\mathbb C}.$
Then, $F\in {\mathcal E}_{\sigma \|H\|}({\mathcal B}({\mathbb H})).$ Indeed, $F$ is complex-differentiable 
at any point $z\in {\mathbb C}$ with derivative $F'(z)=Df(A+zH;H),$ so, 
it is an entire function with values in $E={\mathcal B}({\mathbb H}).$ 
In addition, by von Neumann theorem (see, e.g., \cite{Davies}, Theorem 9.5.3), 
$$
\|F(z)\|=\|f(A+zH)\|\leq \sup_{|\zeta|\leq \|A\|+|z|\|H\|}|f(\zeta)|
\leq \|f\|_{L_{\infty}({\mathbb R})} e^{\sigma \|A\|}e^{\sigma \|H\||z|}, z\in {\mathbb C},
$$ 
implying that $F$ is of exponential type $\sigma \|H\|.$ 
Note also that 
$$
\sup_{x\in {\mathbb R}}\|F(x)\|=\sup_{x\in {\mathbb R}}\|f(A+xH)\|
\leq \sup_{x\in {\mathbb R}}|f(x)| = \|f\|_{L_{\infty}({\mathbb R})}.
$$
Hence, bounds (\ref{F_one}) and (\ref{F_two}) imply that 
$$
\|f(A+H)-f(A)\|=\|F(1)-F(0)\|\leq \sup_{x\in {\mathbb R}}\|F^{\prime}(x)\|\leq 
\sigma \|H\| \sup_{x\in {\mathbb R}}\|F(x)\|\leq \sigma \|f\|_{L_{\infty}({\mathbb R})}
\|H\|
$$
and 
$$
\|Df(A;H)\|=\|F^{\prime}(0)\|\leq \sigma \|f\|_{L_{\infty}({\mathbb R})}
\|H\|,
$$
which proves bounds (\ref{oper_Lipschitz}) and (\ref{derivA}). 
Similarly, using (\ref{F_three}), we get
$$
\|S_f(A;H)\|= \|f(A+H)-f(A)-Df(A;H)\|= \|F(1)-F(0)-F^{\prime}(0)(1-0)\|
$$
$$
=\|S_F(0,1)\|\leq \frac{\sigma^2 \|H\|^2}{2}\sup_{x\in {\mathbb R}}\|F(x)\|
\leq \frac{\sigma^2}{2}\|f\|_{L_{\infty}({\mathbb R})}\|H\|^2,
$$
proving (\ref{first_order}).

To prove bound (\ref{first_order_Lipschitz}), define 
$$
F(z):= f(A+H+z(H'-H))-f(A+z(H'-H)), z\in {\mathbb C}.
$$
As in the previous case, $F$ is an entire function with values in 
${\mathcal B}({\mathbb H}).$ The bound 
$$
\|F(z)\|\leq \|f\|_{L_{\infty}({\mathbb R})} \Bigl(e^{\sigma \|A+H\|}+ e^{\sigma \|A\|}\Bigr)
e^{\sigma \|H'-H\||z|}
$$
implies that $F\in {\mathcal E}_{\sigma \|H'-H\|}({\mathcal B}({\mathbb H})).$
Clearly, we also have $\sup_{x\in {\mathbb R}}\|F(x)\|\leq 2\|f\|_{L_{\infty}({\mathbb R})}.$

Note that 
\begin{equation}
\label{Sf}
S_f(A;H')-S_f(A;H)= Df(A+H;H'-H)-Df(A;H'-H)
+S_f(A+H;H'-H)
\end{equation}
and bound (\ref{first_order}) implies 
$$
\|S_f(A+H;H'-H)\| \leq \frac{\sigma^2}{2}\|f\|_{L_{\infty}({\mathbb R})}\|H'-H\|^2.
$$
On the other hand, we have (by Bernstein inequality)
$$
\|Df(A+H;H'-H)-Df(A;H'-H)\| = \|F^{\prime}(0)\| 
\leq 
\sigma \|H'-H\|\sup_{x\in {\mathbb R}}\|F(x)\|
$$
and (\ref{oper_Lipschitz}) implies that 
$$
\sup_{x\in {\mathbb R}}\|F(x)\|=
\sup_{x\in {\mathbb R}}\|f(A+H+x(H'-H))-f(A+x(H'-H))\|\leq 
\sigma \|f\|_{L_{\infty}({\mathbb R})} \|H\|.
$$
Now, it follows from (\ref{Sf}) that 
$$
\|S_f(A;H')-S_f(A;H)\|\leq \sigma^2\|f\|_{L_{\infty}({\mathbb R})}
\biggl(\|H\|+\frac{\|H'-H\|}{2}\biggr)\|H'-H\|,
$$
which implies (\ref{first_order_Lipschitz}).

\end{proof}

\begin{remark}
In addition to (\ref{first_order_Lipschitz}), the following bound follows from (\ref{oper_Lipschitz}) and (\ref{derivA})
\begin{equation}
\label{first_order_Lipschitz_B}
\|S_f(A;H')-S_f(A;H)\|\leq 2\sigma\|f\|_{L_{\infty}({\mathbb R})}
\|H'-H\|.
\end{equation}
Note also that 
$
\delta (H,H')\leq \|H\|+\|H'\|, H,H'\in {\mathcal B}_{sa}({\mathbb H}).
$
\end{remark}

Following Aleksandrov and Peller \cite{Aleksandrov_Peller}, we use Littlewood-Paley dyadic decomposition 
and the corresponding family of Besov norms to extend the bounds of Lemma  \ref{taylor_operator} to functions in Besov classes. 
It would be more convenient for 
our purposes to use inhomogeneous Besov norms instead of homogeneous 
norms used in \cite{Aleksandrov_Peller}.
Peller \cite{Peller_87} proved that any function $f\in B_{\infty,1}^{1}(\mathbb R)$ 
\footnote{Peller, in fact, used modified homogeneous Besov 
classes instead of inhomogeneous Besov spaces we use in this paper.}
is operator Lipschitz 
and operator differentiable on the space of self-adjoint operators with respect to the operator norm (in Aleksandrov and Peller \cite{Aleksandrov_Peller}, these facts were proved 
using Littlewood-Paley theory and extensions of Bernstein inequality for operator functions, 
see also their earlier paper \cite{Aleksandrov_Peller_2010}). We will state Peller's results in the next lemma 
in a convenient form for our purposes along with some additional bounds on 
the remainder of the first order Taylor expansion $S_f(A;H)=f(A+H)-f(A)-Df(A;H)$
for $f$ in proper Besov spaces.

\begin{lemma}
\label{lemma_Lipsch_Bes}
If $f\in B_{\infty,1}^1({\mathbb R}),$ then for all $A,H\in {\mathcal B}_{sa}({\mathbb H}),$ 
\begin{equation}
\label{Lipschitz_Bes}
\|f(A+H)-f(A)\|\leq 2\|f\|_{B_{\infty,1}^1({\mathbb R})}\|H\|.
\end{equation}
Moreover, the function ${\mathcal B}_{sa}({\mathbb H})\ni A\mapsto f(A)\in {\mathcal B}_{sa}({\mathbb H})$ is Fr\'echet differentiable with respect to the operator norm
with derivative given by the following series (that converges in the operator norm):
\begin{equation}
\label{derivative_series}
Df(A;H)=\sum_{n\geq 0}Df_n(A;H).
\end{equation}
If $f\in B_{\infty,1}^s({\mathbb R})$
for some $s\in [1,2],$ then, for all $A,H, H'\in {\mathcal B}_{sa}({\mathbb H}),$
\begin{equation}
\label{bound_s_1}
\|S_f(A;H)\|\leq 2^{3-s}\|f\|_{B_{\infty, 1}^s} \|H\|^s
\end{equation}
and 
\begin{equation}
\label{bound_s_2}
\|S_f(A;H')-S_f(A;H)\|\leq 4\|f\|_{B_{\infty, 1}^s} (\delta(H,H'))^{s-1} \|H'-H\|.
\end{equation}
\end{lemma}
 
\begin{proof}
Recall that, for $f\in B_{\infty,1}^1({\mathbb R}),$
the series $\sum_{n\geq 0}f_n$ converges uniformly in ${\mathbb R}$
to function $f.$ 
Since $A, A+H, A+H'$ are bounded self-adjoint operators, we also get 
\begin{equation}
\label{gnf}
\sum_{n\geq 0} f_n(A)=f(A),\ \sum_{n\geq 0}f_n(A+H)=f(A+H),\ 
\sum_{n\geq 0}f_n(A+H')=f(A+H')
\end{equation}
with convergence of the series in the operator norm. 

To prove bound (\ref{Lipschitz_Bes}), observe that 
\begin{align*}
&
\|f(A+H)-f(A)\|
=\biggl\|\sum_{n\geq 0}[f_n(A+H)-f_n(A)]\biggr\|
\\
&
\leq \sum_{n\geq 0}\|f_n(A+H)-f_n(A)\|\leq \sum_{n\geq 0}2^{n+1}
\|f_n\|_{L_{\infty}({\mathbb R})}\|H\|= 2\|f\|_{B_{\infty,1}^1}\|H\|,
\end{align*}
where we used (\ref{oper_Lipschitz}). 

By bound (\ref{derivA}), 
$$
\sum_{n\geq 0} \|Df_n(A;H)\|\leq \sum_{n\geq 0}2^{n+1}\|f_n\|_{L_{\infty}({\mathbb R})}\|H\|= 2\|f\|_{B_{\infty,1}^1} \|H\|<\infty,
$$
implying the convergence in the operator norm of the series 
$\sum_{n\geq 0}Df_n(A;H).$
We will define 
\begin{equation}
\label{def_deriv}
Df(A;H):=\sum_{n\geq 0}Df_n(A;H)
\end{equation}
and prove that this yields the Fr\'echet derivative of $f(A).$ To this end, 
note that (\ref{gnf}) and (\ref{def_deriv}) implies that 
\begin{equation}
\label{Sfrepr}
S_f(A;H)= \sum_{n\geq 0}[f_n(A+H)-f_n(A)- Df_n(A;H)]=
\sum_{n\geq 0}S_{f_n}(A;H).
\end{equation}
As a consequence,
$$
\|S_f(A;H)\|\leq \sum_{n\leq N} \|S_{f_n}(A;H)\|
+ \sum_{n>N} \|f_n(A+H)-f_n(A)\|+ \sum_{n>N}
\|D f_n(A;H)\|
$$
$$
\leq \sum_{n\leq N} 2^{2(n+1)}\|f_n\|_{L_{\infty}({\mathbb R})}\|H\|^2
+ 2\sum_{n>N} 2^{n+1}\|f_n\|_{L_{\infty}({\mathbb R})}\|H\|,
$$
where we used bounds (\ref{oper_Lipschitz}), (\ref{derivA}) and (\ref{first_order}).
Given $\eps>0,$ take $N$ so that 
$
\sum_{n>N} 2^{n+1}\|f_n\|_{L_{\infty}}\leq \frac{\eps}{4}
$
and suppose $H$ satisfies 
$
\|H\|\leq \frac{\eps}{2\sum_{n\leq N}2^{2(n+1)} \|f_n\|_{L_{\infty}({\mathbb R})}}. 
$
This implies that $\|S_f(A;H)\|\leq \eps \|H\|$ and Fr\'echet differentiability
of $f(A)$ with derivative $Df(A;H)$ follows.

To prove (\ref{bound_s_2}), use (\ref{first_order_Lipschitz}) and (\ref{first_order_Lipschitz_B}) to get 
\begin{align*}
&
\|S_{f_n}(A;H')-S_{f_n}(A;H)\|
\\
& 
\leq 2^{2(n+1)}\|f_n\|_{L_{\infty}({\mathbb R})}\delta(H,H')\|H'-H\|
\bigwedge 2^{n+2}\|f_n\|_{L_{\infty}({\mathbb R})}\|H'-H\| 
\\
&
=2^{n+2}\|f_n\|_{L_{\infty}({\mathbb R})}(2^{n}\delta(H,H')\wedge 1)
\|H'-H\|.
\end{align*}
It follows that 
\begin{align*}
&
\|S_f(A;H')-S_f(A;H)\|\leq \sum_{n\geq 0} \|S_{f_n}(A;H')-S_{f_n}(A;H)\|
\\
&
\leq \sum_{n\geq 0} 2^{n+2}\|f_n\|_{L_{\infty}({\mathbb R})}(2^{n}\delta(H,H')\wedge 1)
\|H'-H\|
\\
&
= 
4\biggl(\sum_{2^{n}\leq 1/\delta(H,H')} 
2^{2n}\|f_n\|_{L_{\infty}({\mathbb R})}\delta(H,H')
+\sum_{2^{n}> 1/\delta(H,H')} 
2^{n}\|f_n\|_{L_{\infty}({\mathbb R})}\biggr)
\|H'-H\|
\\
&
\leq 
4\biggl(\sum_{2^{n}\leq 1/\delta(H,H')} 
2^{sn}\|f_n\|_{L_{\infty}({\mathbb R})}\biggl(\frac{1}{\delta(H,H')}\biggr)^{2-s}
\delta(H,H')
\\
&
+\sum_{2^{-n}<\delta(H,H')} 
2^{sn}\|f_n\|_{L_{\infty}({\mathbb R})}(\delta(H,H'))^{s-1}\biggr)
\|H'-H\|
\\
&
\leq 
4\biggl(\sum_{2^{n}\leq 1/\delta(H,H')} 
2^{sn}\|f_n\|_{L_{\infty}({\mathbb R})}
+\sum_{2^{n}> 1/\delta(H,H')} 
2^{sn}\|f_n\|_{L_{\infty}({\mathbb R})}\biggr)
(\delta(H,H'))^{s-1}\|H'-H\|
\\
&
=4\|f\|_{B_{\infty,1}^s}(\delta(H,H'))^{s-1}\|H'-H\|,
\end{align*}
which yields (\ref{bound_s_2}). Bound (\ref{bound_s_1})
follows from (\ref{bound_s_2}) when $H'=0.$

\end{proof} 

Suppose $A\in {\mathcal B}_{sa}({\mathbb H})$ is a compact operator 
with spectral representation
$
A= \sum_{\lambda \in \sigma(A)} \lambda P_{\lambda},
$
where $P_{\lambda}$ denotes the spectral projection corresponding 
to the eigenvalue $\lambda.$ The following formula for the derivative 
$Df(A;H), f\in B_{\infty,1}^1({\mathbb R})$ is well known (see \cite{Bhatia}, Theorem V.3.3 for a finite-dimensional version):
\begin{equation}
\label{Loewner}
Df(A;H)= \sum_{\lambda, \mu\in \sigma(A)} f^{[1]}(\lambda,\mu)
P_{\lambda} H P_{\mu},
\end{equation}
where $f^{[1]}(\lambda,\mu):=\frac{f(\lambda)-f(\mu)}{\lambda-\mu}$
for $\lambda \neq \mu$ and 
$f^{[1]}(\lambda,\mu):=f^{\prime}(\lambda)$ for $\lambda=\mu.$
In other words, the operator $Df(A;H)$ can be represented in the basis of eigenvectors 
of $A$ as a Schur product of Loewner matrix $(f^{[1]}(\lambda,\mu))_{\lambda,\mu\in \sigma(A)}$
and the matrix of operator $H$ in this basis.  
We will need this formula only in the case of discrete spectrum, but 
there are also extensions for more general operators $A$ with continuous  
spectrum (with the sums being replaced by double operator integrals), see Aleksandrov and Peller \cite{Aleksandrov_Peller}, theorems 3.5.11 and 1.6.4. 

Finally, we need some extensions of the results stated above for higher order 
derivatives (see \cite{Skripka}, \cite{Azamov}, \cite{Kissin} and references therein for a number of subtle 
results in this direction). If 
$g: {\mathcal B}_{sa}({\mathbb H})\mapsto {\mathcal B}_{sa}({\mathbb H})$
is a $k$ times Fr\'echet differentiable function, its $k$-th derivative $D^k g (A)$
$A\in {\mathcal B}_{sa}({\mathbb H})$ can be viewed as a symmetric multilinear 
operator valued form 
$$
D^k g(A)(H_1,\dots, H_k)=D^k g(A;H_1,\dots, H_k), H_1,\dots, H_k\in  {\mathcal B}_{sa}({\mathbb H}).
$$ 
Given such a form $M: {\mathcal B}_{sa}({\mathbb H})\times \dots \times {\mathcal B}_{sa}({\mathbb H})\mapsto {\mathcal B}_{sa}({\mathbb H}),$ define its operator norm as
$$
\|M\|:= \sup_{\|H_1\|,\dots ,\|H_k\|\leq 1} \|M(H_1,\dots, H_k)\|.
$$
The derivatives $D^k g(A)$ are defined iteratively: 
$$
D^k g(A)(H_1,\dots, H_{k-1}, H_k)= D (D^{k-1} g (A)(H_1,\dots, H_{k-1}))(H_k).
$$

For $f\in {\mathcal E}_{\sigma}\bigcap L_{\infty}({\mathbb R}),$ 
the $k$-th derivative $D^k f(A)$ is given by the following formula:
$$
D^kf(A;H_1,\dots, H_k)= 
\frac{(-1)^{k+1}}{2\pi i}\sum_{\pi\in S_k}\oint_{\gamma} f(z)R_A(z)H_{\pi(1)}R_A(z)H_{\pi(2)} \dots 
R_A(z) H_{\pi(k)}R_A(z) dz, 
$$
$$
H_1,\dots, H_k\in {\mathcal B}_{sa}({\mathbb H}).
$$ 
where $\gamma\subset {\mathbb C}$ is a contour surrounding $\sigma(A)$ with a counterclockwise orientation. 

The following lemmas hold.

\begin{lemma}
Let $f\in {\mathcal E}_{\sigma}\bigcap L_{\infty}({\mathbb R}).$
Then, for all $k\geq 1,$
\begin{equation}
\label{Dkzero}
\|D^k f(A)\|\leq \sigma^k \|f\|_{L_{\infty}(\mathbb R)}, A\in {\mathcal B}_{sa}({\mathbb H}). 
\end{equation}
\end{lemma}

\begin{proof}
Given $A,H_1,\dots H_k\in {\mathcal B}_{sa}({\mathbb H}),$ denote
$$
F(z_1,\dots, z_k)= f(A+z_1 H_1+\dots +z_k H_k), (z_1,\dots, z_k)\in {\mathbb C}^k.
$$
Then, $f$ is an entire operator valued function of exponential type 
$(\sigma \|H_1\|,\dots, \sigma \|H_k\|):$ 
$$
\|F(z_1,\dots, z_k)\| \leq \sup_{|\zeta|\leq \|A+z_1 H_1+\dots +z_k H_k\|}|f(\zeta)|
\leq e^{\|A\|} \exp\Bigl\{\sigma \|H_1\||z_1|+\dots + \sigma \|H_k\||z_k|\Bigr\}.
$$
By Bernstein inequality \eqref{Bernstein-Nikolsky} (extended to Banach space valued functions
as it was done at the beginning of the proof of Lemma \ref{taylor_operator}),
we get 
$$
\biggl\|\frac{\partial^k F(x_1,\dots, x_k)}{\partial x_1 \dots \partial x_k}\biggr\|
\leq \sigma^k \|H_1\|\dots \|H_k\| \sup_{x_1,\dots, x_k\in {\mathbb R}}\|F(x_1,\dots, x_k)\|.
$$
Therefore,
$$
\|D^k f(A+x_1 H_1+\dots x_k H_k)(H_1,\dots, H_k)\|\leq 
\sigma^k \|H_1\|\dots \|H_k\|\|f\|_{L_{\infty}({\mathbb R})}.
$$
For $x_1=\dots =x_k=0,$ this yields
$$
\|D^k f(A)(H_1,\dots, H_k)\|\leq 
\sigma^k \|H_1\|\dots \|H_k\|\|f\|_{L_{\infty}({\mathbb R})},
$$
implying the claim of the lemma.

\end{proof}

\begin{lemma}
Let $f\in {\mathcal E}_{\sigma}\bigcap L_{\infty}({\mathbb R}).$
Then, for all $k\geq 1$ and all $A,H_1,\dots, H_k, H\in {\mathcal B}_{sa}({\mathbb H}),$
\begin{equation}
\label{Dkone}
\|D^k f(A+H; H_1,\dots, H_k)-D^k f(A;H_1,\dots, H_k)\|
\leq  
\sigma^{k+1} \|f\|_{L_{\infty}({\mathbb R})} \|H_1\|\dots \|H_k\|\|H\|
\end{equation}
and 
\begin{equation}
\label{Dktwo}
\|S_{D^k f(\cdot; H_1,\dots, H_k)}(A;H)\|
\leq \frac{\sigma^{k+2}}{2} \|f\|_{L_{\infty}({\mathbb R})} \|H_1\|\dots \|H_k\|\|H\|^2.
\end{equation}
\end{lemma}

\begin{proof}
Bound (\ref{Dkone}) easily follows from (\ref{Dkzero}) (applied to the derivative 
$D^{k+1} f$). The proof of bound (\ref{Dktwo}) relies on Bernstein 
inequality \eqref{Bernstein-Nikolsky} and on a slight modification of the proof of bound (\ref{first_order}).

\end{proof}

\begin{lemma}
\label{lemma_s}
Suppose $f\in B^{k}_{\infty,1}({\mathbb R}).$ 
Then the function ${\mathcal B}_{sa}({\mathbb H})\ni A\mapsto f(A)\in {\mathcal B}_{sa}({\mathbb H})$ is $k$ times Fr\'echet differentiable and 
\begin{equation}
\label{DjfA}
\|D^j f(A)\|\leq 2^{j} \|f\|_{B^{j}_{\infty,1}}, A\in {\mathcal B}_{sa}({\mathbb H}), j=1,\dots, k. 
\end{equation}
Moreover, if for some $s\in (k,k+1],$ 
$f\in B^{s}_{\infty,1}({\mathbb R}),$
then 
\begin{equation}
\label{Djf_lip}
\|D^k f(A+H)-D^k f(A)\|\leq 2^{k+1} \|f\|_{B^{s}_{\infty,1}}\|H\|^{s-k}, 
A, H\in {\mathcal B}_{sa}({\mathbb H}).
\end{equation}
\end{lemma}

\begin{proof}
As in the proof of Lemma \ref{lemma_Lipsch_Bes}, we use Littlewood-Paley 
decomposition of $f.$ Since, by (\ref{Dkzero}), for all $j=1,\dots, k$ 
\begin{align}
\label{DjDjDj}
&
\nonumber
\sum_{n\geq 0}\|D^j f_n (A;H_1,\dots, H_j)\|
\leq \sum_{n\geq 0} 2^{(n+1)j} \|f_n\|_{L_{\infty}({\mathbb R})}\|H_1\|\dots \|H_j\|
\\
&
\leq 2^j \|f\|_{B^j_{\infty,1}} \|H_1\| \dots \|H_j\|<+\infty,
\end{align}
the series $\sum_{n\geq 0} D^j f_n (A;H_1,\dots, H_j)$
converges in operator norm and we can define symmetric $j$-linear
forms
$$
D^j f (A;H_1,\dots, H_j):=\sum_{n\geq 0} D^j f_n (A;H_1,\dots, H_j),
j=1,\dots, k.
$$
By the same argument as in the proof of claim (\ref{derivA}) of Lemma \ref{lemma_Lipsch_Bes} and using bounds (\ref{Dkone}) and (\ref{Dktwo}), we can now prove by induction that 
$D^j f (A;H_1,\dots, H_j), j=1,\dots, k$ are the consecutive 
derivatives of $f(A).$ Indeed,
for $j=1,$ it was already proved in Lemma \ref{lemma_Lipsch_Bes}. 
Assuming that it is true for some $j<k,$ we have to prove that it is also 
true for $j+1.$ To this end, note that 
\begin{align*}
&
\|D^{j}f(A+H; H_1,\dots, H_j) - D^{j}f(A; H_1,\dots, H_j)- D^{j+1}f(A; H_1,\dots, H_j, H)\|
\\
&
\leq \sum_{n\leq N} \|S_{D^{j}f_n (\cdot; H_1,\dots, H_j)}(A;H)\|
\\
&
+ \sum_{n>N} \|D^j f_n(A+H; H_1,\dots, H_j)-D^j f_n(A; H_1,\dots, H_j)\|
\\
&
+ \sum_{n>N}\|D^{j+1}{f_n}(A;H_1,\dots, H_j, H)\|
\\
&
\leq \sum_{n\leq N} \frac{2^{(j+2)(n+1)}}{2}\|f_n\|_{L_{\infty}({\mathbb R})}
\|H_1\|\dots \|H_j\| \|H\|^2
\\
&
+ 2\sum_{n>N} 2^{(j+1)(n+1)}\|f_n\|_{L_{\infty}({\mathbb R})}
\|H_1\|\dots \|H_j\|\|H\|.
\end{align*}
Given $\eps>0,$ take $N$ so that 
$
\sum_{n>N} 2^{(j+1)(n+1)}\|f_n\|_{L_{\infty}}\leq \frac{\eps}{4},
$
which is possible for $f\in B^{j+1}_{\infty,1}({\mathbb R}),$
and suppose $H$ satisfies 
$
\|H\|\leq \frac{\eps}{\sum_{n\leq N}2^{(j+2)(n+1)} \|f_n\|_{L_{\infty}({\mathbb R})}}. 
$
Then, we have 
$$
\|D^{j}f(A+H; H_1,\dots, H_j) - D^{j}f(A; H_1,\dots, H_j)- D^{j+1}f(A; H_1,\dots, H_j, H)\|
\leq \eps \|H_1\|\dots \|H_k\| \|H\|.
$$
Therefore, the function $A\mapsto D^j f(A;H_1,\dots, H_j)$ is Fr\'echet differentiable 
with derivative $D^{j+1}f(A;H_1,\dots, H_j, H).$ 
 
Bounds (\ref{DjfA}) now follow from (\ref{DjDjDj}).

To prove (\ref{Djf_lip}), note that 
\begin{align*}
&
\|D^k f(A+H)(H_1,\dots, H_k)-D^k f(A)(H_1,\dots, H_k)\|
\\
&
\leq \sum_{n\geq 0}
\|D^k f_n(A+H)(H_1,\dots, H_k)-D^k f_n(A)(H_1,\dots, H_k)\|.
\end{align*}
Using bounds (\ref{Dkzero}) and (\ref{Dkone}), we get 
\begin{align*}
&
\|D^k f(A+H)(H_1,\dots, H_k)-D^k f(A)(H_1,\dots, H_k)\|
\leq 
\\
&
\sum_{2^{n}\leq \frac{1}{\|H\|}} 
2^{(n+1)(k+1)} \|f_n\|_{L_{\infty}({\mathbb R})}\|H\| \|H_1\|\dots \|H_k\|
+ 
2\sum_{2^{n}>\frac{1}{\|H\|}} 
2^{(n+1)k} \|f_n\|_{L_{\infty}({\mathbb R})}\|H_1\|\dots \|H_k\|
\\
&
\leq 2^{k+1}\|H_1\|\dots \|H_k\|
\biggl[
\sum_{2^{n}\leq \frac{1}{\|H\|}} 
2^{ns} \|f_n\|_{L_{\infty}({\mathbb R})}2^{n(k+1-s)} \|H\|
+  \sum_{2^{n}>\frac{1}{\|H\|}} 
 2^{ns}\|f_n\|_{L_{\infty}({\mathbb R})}2^{n(k-s)}
\biggr]
\\
&
\leq 
2^{k+1}\|H_1\|\dots \|H_k\| \|H\|^{s-k}
\biggl[
\sum_{2^{n}\leq \frac{1}{\|H\|}} 
2^{ns}\|f_n\|_{L_{\infty}({\mathbb R})}
+  \sum_{2^{n}>\frac{1}{\|H\|}} 
 2^{ns}\|f_n\|_{L_{\infty}({\mathbb R})}
\biggr]
\\
&
=2^{k+1} \|f\|_{B^{s}_{\infty,1}}\|H\|^{s-k}\|H_1\|\dots \|H_k\|,
\end{align*}
which implies (\ref{Djf_lip}).

\end{proof}

In what follows, we use the definition of H\"older space norms of 
functions of bounded self-adjoint operators. 
For an open set $G\subset {\mathcal B}_{sa}({\mathbb H}),$
a $k$-times Fr\'echet differentiable functions $g: G\mapsto {\mathcal B}_{sa}({\mathbb H})$
and, for $s=k+\beta, \beta \in (0,1],$ define
\begin{equation}
\label{define_C^s}
\|g\|_{C^s(G)}:= \max_{0\leq j\leq k} \sup_{A\in G}\|D^j g(A)\|
\bigvee \sup_{A, A+H\in G, H\neq 0}
\frac{\|D^k g(A+H)-D^k g(A)\|}{\|H\|^{\beta}}.
\end{equation}
Similar definition applies to $k$-times Fr\'echet differentiable functions $g: G\mapsto {\mathbb R}$
(with $\|D^j g(A)\|$ being the operator norm of $j$-linear form). 
In both cases, $C^s(G)$ denotes the space of functions $g$ on $G$ (operator valued or real valued)
with $\|g\|_{C^s(G)}<\infty.$
In particular, these norms apply to operator functions ${\mathcal B}_{sa}({\mathbb H})\ni A\mapsto f(A)\in {\mathcal B}_{sa}({\mathbb H}),$
where $f$ is a function in real line. With a little abuse of notation,
we write the norm of such operator functions as $\|f\|_{C^s({\mathcal B}_{sa}({\mathbb H}))}.$ 
The next result immediately follows from Lemma \ref{lemma_s}.

\begin{corollary}
\label{remark_diff}
Suppose that, for some $k\geq 0$ and $s\in (k,k+1],$ we have
$f\in B_{\infty,1}^{s}({\mathbb R}).$ Then 
$
\|f\|_{C^s({\mathcal B}_{sa}({\mathbb H}))}
\leq 2^{k+1}\|f\|_{B_{\infty,1}^{s}}.
$ 
\end{corollary} 
 
\section{Concentration bounds for the remainder of the first order Taylor expansion}
\label{Sec:Concentration bounds for the remainder}

Let $g:{\mathcal B}_{sa}({\mathbb H})\mapsto {\mathbb R}$ be a Fr\'echet 
differentiable function with respect to the operator norm with derivative 
$Dg(A;H), H\in {\mathcal B}_{sa}({\mathbb H}).$ Note that $Dg(A;H), 
H\in {\mathcal B}_{sa}({\mathbb H})$ is a bounded linear functional 
on ${\mathcal B}_{sa}({\mathbb H})$ and its restriction to the subspace 
${\mathcal C}_{sa}({\mathbb H})\subset {\mathcal B}_{sa}({\mathbb H})$
of compact self-adjoint operators in ${\mathbb H}$ can be represented as 
$Dg(A,H)= \langle Dg(A), H\rangle, H\in {\mathcal C}_{sa}({\mathbb H}),$
where $Dg(A)\in {\mathcal S}_1$ is a trace class operator in ${\mathbb H}.$
Let $S_g(A;H)$ be the remainder of the first order Taylor expansion of $g:$ 
$$
S_g(A;H):= g(A+H)-g(A)-Dg(A;H), A,H\in {\mathcal B}_{sa}({\mathbb H}).
$$ 
Our goal is to obtain concentration inequalities for random variable $S_g(\Sigma;\hat \Sigma-\Sigma)$ around its expectation. It will be done under the following assumption 
on the remainder $S_g(A;H):$

\begin{assumption}
\label{assume_Lipschitz}
Let $s\in [1,2].$ Assume there exists a constant $L_{g,s}>0$ such that, for all $\Sigma\in {\mathcal C}_+({\mathbb H}), H,H'\in {\mathcal B}_{sa}({\mathbb H}),$
$
|S_g(\Sigma;H')-S_g(\Sigma;H)|\leq L_{g,s} (\|H\|\vee \|H'\|)^{s-1}\|H'-H\|.
$
\end{assumption}

Note that Assumption \ref{assume_Lipschitz} implies (for $H'=0$) that 
$
|S_g(\Sigma;H)|\leq L_{g,s} \|H\|^s, \Sigma\in {\mathcal C}_+({\mathbb H}), H\in {\mathcal B}_{sa}({\mathbb H}).
$

\begin{theorem}
\label{th-conc-med}
Suppose Assumption \ref{assume_Lipschitz} holds for some $s\in (1,2].$
Then there exists a constant $K_s>0$ such that for all $t\geq 1$ with probability 
at least $1-e^{-t}$
\begin{align}
\label{conc_med_state}
&
|S_g(\Sigma;\hat \Sigma-\Sigma)-{\mathbb E} S_g(\Sigma;\hat \Sigma-\Sigma)|
\\
&
\nonumber
\leq K_s L_{g,s} \|\Sigma\|^s \biggl(
\Bigl(\frac{{\bf r}(\Sigma)}{n}\Bigr)^{(s-1)/2} \bigvee 
\Bigl(\frac{{\bf r}(\Sigma)}{n}\Bigr)^{s-1/2}
\bigvee 
\Bigl(\frac{t}{n}\Bigr)^{(s-1)/2}\bigvee \Bigl(\frac{t}{n}\Bigr)^{s-1/2}\biggr)
\sqrt{\frac{t}{n}}.
\end{align}
\end{theorem}
 
\begin{proof} 
Let $\varphi : {\mathbb R}\mapsto {\mathbb R}$ be such that $\varphi (u)=1, u\leq 1,$ 
$\varphi (u)=0, u\geq 2$ and $\varphi (u)= 2-u, u\in (1,2).$  
Denote $E:=\Hat \Sigma-\Sigma$ and, given $\delta>0,$ define  
\begin{equation}
\label{define_function_h}
h(X_1,\dots, X_n):= S_g(\Sigma;E)\varphi \biggl(\frac{\|E\|}{\delta}\biggr).
\end{equation}
We start with deriving a concentration bound for the function $h(X_1,\dots, X_n)$
of Gaussian random variables $X_1,\dots, X_n.$  To this end, we will show 
that $h(X_1,\dots, X_n)$ satisfies a Lipschitz condition. With a minor abuse 
of notation, we will assume for a while that $X_1,\dots, X_n$ are non-random 
points of ${\mathbb H}$ and let $X_1',\dots, X_n'$ be another set of such points.
Denote by $\hat \Sigma':= n^{-1}\sum_{j=1}^n X_j'\otimes X_j'$ the sample 
covariance based on $X_1',\dots, X_n'$ and let $E':=\hat \Sigma'-\Sigma.$

The following lemma establishes a Lipschitz condition for $h.$

\begin{lemma}
\label{lemma_Lipschitz}
Suppose Assumption \ref{assume_Lipschitz} holds with some $s\in (1,2].$
Then, for an arbitrary $\delta>0$ and $h$ defined by \eqref{define_function_h}, the following bound 
holds with some constant $C_s>0$ for all $X_1,\dots, X_n, X_1',\dots, X_n'\in {\mathbb H}:$
\begin{align}
\label{lipg_04}
&
|h(X_1,\dots, X_n)-h(X_1',\dots, X_n')|
\leq \frac{C_s L_{g,s} (\|\Sigma\|^{1/2} + \sqrt{\delta})\delta^{s-1}}
{\sqrt{n}}\biggl(\sum_{j=1}^n \|X_j-X_j'\|^2\biggr)^{1/2}.
\end{align} 
\end{lemma}

\begin{proof} 
Using the fact that $\varphi$
takes values in $[0,1]$ and it is a Lipschitz function with constant $1,$  
and taking into account Assumption \ref{assume_Lipschitz},
we get 
\begin{equation}
\label{bdg}
|h(X_1,\dots,X_n)|
\leq  |S_g(\Sigma;E)|I(\|E\|\leq 2\delta) \leq L_{g,s}\|E\|^s I(\|E\|\leq 2\delta)
\leq 2^s L_{g,s} \delta^s
\end{equation}
and similarly 
\begin{equation}
\label{bdg'}
|h(X_1',\dots,X_n')|
\leq 2^s L_{g,s} \delta^s.
\end{equation}
We also have
\begin{align}
\label{lipg}
\nonumber
&
|h(X_1,\dots, X_n)-h(X_1',\dots, X_n')|
\leq |S_g(\Sigma, E)-S_g(\Sigma, E')|+ \frac{1}{\delta}|S_g(\Sigma,E')|\|E-E'\|
\\
&
\leq L_{g,s} (\|E\|\vee \|E'\|)^{s-1} \|E'-E\|+
L_{g,s} \frac{1}{\delta} \|E'\|^s \|E'-E\|.
\end{align}
If both $\|E\|\leq 2\delta$ and $\|E'\|\leq 2\delta,$ then (\ref{lipg}) implies 
\begin{equation}
\label{lipg_01}
|h(X_1,\dots, X_n)-h(X_1',\dots, X_n')|\leq 
(2^{s-1}+2^s) L_{g,s} \delta^{s-1} \|E'-E\|.
\end{equation}
If both $\|E\|>2\delta$ and $\|E'\|>2\delta,$
then $\varphi \Bigl(\frac{\|E\|}{\delta}\Bigr)=\varphi \Bigl(\frac{\|E'\|}{\delta}\Bigr)=0,$
implying that $h(X_1,\dots, X_n)=h(X_1', \dots, X_n')=0.$
If $\|E\|\leq 2\delta,$ $\|E'\|>2\delta$ and $\|E'-E\|>\delta,$
then 
$$
|h(X_1,\dots, X_n)-h(X_1',\dots, X_n')|= 
|h(X_1,\dots, X_n)|\leq 2^s L_{g,s} \delta^s 
\leq 2^s L_{g,s} \delta^{s-1}\|E'-E\|.
$$
If $\|E\|\leq 2\delta,$ $\|E'\|>2\delta$ and $\|E'-E\|\leq \delta,$
then $\|E'\|\leq 3\delta$ and, similarly to  (\ref{lipg_01}), we get   
\begin{equation}
\label{lipg_02}
|h(X_1,\dots, X_n)-h(X_1',\dots, X_n')|\leq 
(3^{s-1}+3^s) L_{g,s} \delta^{s-1} \|E'-E\|.
\end{equation} 
By these simple considerations, bound (\ref{lipg_02})
holds in all possible cases. This fact along with (\ref{bdg}), (\ref{bdg'})
yield
\begin{equation}
\label{lipg_03}
|h(X_1,\dots, X_n)-h(X_1',\dots, X_n')|\leq 
(3^{s-1}+3^s)L_{g,s} \delta^{s-1} (\|E'-E\| \wedge \delta). 
\end{equation} 

We now obtain an upper bound on $\|E'-E\|.$ 
We have 
$$
\|E'-E\|=
\biggl\|n^{-1}\sum_{j=1}^n X_j\otimes X_j- n^{-1}\sum_{j=1}^n 
X_j'\otimes X_j'\biggr\|
$$
$$
\leq 
\biggl\|n^{-1}\sum_{j=1}^n (X_j-X_j')\otimes X_j\biggr\|+
\biggl\|n^{-1}\sum_{j=1}^n X_j'\otimes (X_j-X_j')\biggr\|
$$
$$
= 
\sup_{\|u\|,\|v\|\leq 1}\biggl|n^{-1}\sum_{j=1}^n \langle X_j-X_j',u\rangle
\langle X_j,v\rangle\biggr|
+
\sup_{\|u\|,\|v\|\leq 1}\biggl|n^{-1}\sum_{j=1}^n \langle X_j',u\rangle
\langle X_j-X_j',v\rangle\biggr|
$$
$$
\leq 
\sup_{\|u\|\leq 1}\biggl(n^{-1}\sum_{j=1}^n \langle X_j-X_j',u\rangle^2\biggr)^{1/2}
\sup_{\|v\|\leq 1}\biggl(n^{-1}\sum_{j=1}^n \langle X_j,v\rangle^2\biggr)^{1/2}
$$
$$
+
\sup_{\|u\|\leq 1}\biggl(n^{-1}\sum_{j=1}^n \langle X_j',u\rangle^2\biggr)^{1/2}
\sup_{\|v\|\leq 1}\biggl(n^{-1}\sum_{j=1}^n \langle X_j-X_j',v\rangle^2\biggr)^{1/2}
$$ 
$$
\leq 
\frac{\|\hat \Sigma\|^{1/2}+\|\hat \Sigma'\|^{1/2}}{\sqrt{n}} 
\biggl(\sum_{j=1}^n \|X_j-X_j'\|^2\biggr)^{1/2}
\leq (2\|\Sigma\|^{1/2}+\|E\|^{1/2}+\|E'\|^{1/2})\Delta,
$$
where 
$
\Delta:=\frac{1}{\sqrt{n}}\biggl(\sum_{j=1}^n \|X_j-X_j'\|^2\biggr)^{1/2}.
$
Without loss of generality, assume that $\|E\|\leq 2\delta$ (again, if both $\|E\|>2\delta$
and $\|E'\|>2\delta,$ then $h(X_1,\dots, X_n)=h(X_1',\dots, X_n')=0$ and inequality \eqref{lipg_04} trivially holds). Then we have 
$$
\|E'-E\|\leq (2\|\Sigma\|^{1/2}+2\sqrt{2\delta}+\|E'-E\|^{1/2})\Delta
$$
If $\|E'-E\|\leq \delta,$ the last bound implies that 
$$
\|E'-E\| \leq (2\|\Sigma\|^{1/2} + (2\sqrt{2}+1)\sqrt{\delta})\Delta
\leq 
4\|\Sigma\|^{1/2}\Delta \bigvee (4\sqrt{2}+2)\sqrt{\delta}\Delta.
$$
Otherwise, if $\|E'-E\|>\delta,$ we get 
$$
\|E'-E\|\leq 4\|\Sigma\|^{1/2}\Delta \bigvee (4\sqrt{2}+2)\Delta\|E'-E\|^{1/2},
$$
which yields 
$$
\|E'-E\|\leq 4\|\Sigma\|^{1/2}\Delta \bigvee (4\sqrt{2}+2)^2\Delta^2.
$$
Thus, either $\|E'-E\|\leq 4\|\Sigma\|^{1/2}\Delta,$ or 
$\|E'-E\|\leq (4\sqrt{2}+2)^2\Delta^2.$  In the last case,
we also have (since $\delta<\|E'-E\|$) 
$$
\delta < \sqrt{\delta}\|E'-E\|^{1/2}\leq  (4\sqrt{2}+2)\sqrt{\delta}\Delta.
$$
This shows that 
\begin{equation}
\label{bdEE'}
\|E'-E\|\wedge \delta \leq 4\|\Sigma\|^{1/2}\Delta \bigvee (4\sqrt{2}+2)\sqrt{\delta}\Delta
\end{equation}
both when $\|E'-E\|\leq \delta$ and when $\|E'-E\|>\delta.$

Substituting bound (\ref{bdEE'}) in (\ref{lipg_03}) yields
\begin{align*}
&
|h(X_1,\dots, X_n)-h(X_1',\dots, X_n')|
\\
&
\leq 
\frac{(3^{s-1}+3^s) L_{g,s} (4\|\Sigma\|^{1/2} + (4\sqrt{2}+2)\sqrt{\delta})\delta^{s-1}}
{\sqrt{n}}\biggl(\sum_{j=1}^n \|X_j-X_j'\|^2\biggr)^{1/2}, 
\end{align*}
which implies (\ref{lipg_04}).

\end{proof}

In what follows, we set, for a given $t>0,$ 
$$
\delta=\delta_n(t):= {\mathbb E}\|\hat\Sigma-\Sigma\|+ C\|\Sigma\|
\biggl[\biggl(\sqrt{\frac{{\bf r}(\Sigma)}{n}} \bigvee 1\biggr)\sqrt{\frac{t}{n}}\bigvee 
\frac{t}{n}\biggr].
$$
It follows from \eqref{operator_hatSigma_exp} that there exists a choice of absolute 
constant $C>0$ such that
\begin{equation}
\label{KL-1}
{\mathbb P}\{\|\hat \Sigma-\Sigma\|\geq \delta_n(t)\}\leq e^{-t}, t\geq 1.
\end{equation}
Assuming that $t\geq \log (4),$ we get that ${\mathbb P}\{\|E\|\geq \delta\}\leq 1/4.$
Let $M:={\rm Med}(S_g(\Sigma;E))$ be a median of random variable $S_g(\Sigma;E).$ 
Then 
$$
{\mathbb P}\{h(X_1,\dots, X_n)\geq M\}\geq 
{\mathbb P}\{h(X_1,\dots, X_n)\geq M, \|E\|<\delta\}
$$
$$
\geq {\mathbb P}\{S_g(\Sigma;E)\geq M, \|E\|<\delta\}\geq 1/2- {\mathbb P}\{\|E\|\geq \delta\}\geq 1/4.
$$
Similarly, ${\mathbb P}\{h(X_1,\dots, X_n)\leq M\}\geq 1/4.$
In view of Lipschitz property of $h$ (Lemma \ref{lemma_Lipschitz}),
we now use a relatively standard argument (see Lemma 2 in \cite{Koltchinskii_Lounici_arxiv} and its 
applications later in Section 3 of that paper) based on Gaussian isoperimetric inequality (see Ledoux \cite{Ledoux}, Theorem 2.5 and inequality (2.9)) to conclude that with 
probability at least $1-e^{-t}$
$$
|h(X_1,\dots, X_n)-M|\lesssim_{s} L_{g,s} \delta^{s-1} (\|\Sigma\|^{1/2}+\delta^{1/2})
\|\Sigma\|^{1/2}\sqrt{\frac{t}{n}}.
$$
Moreover, since $S_g(\Sigma;E)=h(X_1,\dots, X_n)$ on the event 
$\{\|E\|<\delta\}$ of probability at least $1-e^{-t},$ we get that with 
probability $1-2e^{-t}$ 
\begin{equation}
\label{conc_med}
|S_g(\Sigma;E)-M|\lesssim_{s} L_{g,s} \delta^{s-1} (\|\Sigma\|^{1/2}+\delta^{1/2})
\|\Sigma\|^{1/2}\sqrt{\frac{t}{n}}.
\end{equation}
It follows from \eqref{operator norm_hatSigma} that 
\begin{equation}
\label{bd_delta}
\delta= \delta_n(t)\lesssim \|\Sigma\|\biggl(\sqrt{\frac{{\bf r}(\Sigma)}{n}}
\bigvee \frac{{\bf r}(\Sigma)}{n}\bigvee \sqrt{\frac{t}{n}}\bigvee \frac{t}{n}\biggr).
\end{equation}
Substituting (\ref{bd_delta}) into (\ref{conc_med}) easily yields that with probability 
at least $1-2 e^{-t}$
\begin{align}
\label{conc_med_A}
&
|S_g(\Sigma;E)-M|
\\
&
\nonumber
\lesssim_{s} L_{g,s} \|\Sigma\|^s \biggl(
\Bigl(\frac{{\bf r}(\Sigma)}{n}\Bigr)^{(s-1)/2} \bigvee 
\Bigl(\frac{{\bf r}(\Sigma)}{n}\Bigr)^{s-1/2}
\bigvee 
\Bigl(\frac{t}{n}\Bigr)^{(s-1)/2}\bigvee \Bigl(\frac{t}{n}\Bigr)^{s-1/2}\biggr)
\sqrt{\frac{t}{n}},
\end{align}
and, moreover, by adjusting the value of the constant in inequality 
(\ref{conc_med_A}) the probability bound can be replaced by $1-e^{-t}.$ 
By integrating out the tails of probability bound (\ref{conc_med_A}) one can get 
that 
\begin{align}
\label{mean_med}
&
|{\mathbb E} S_g(\Sigma;E) -M|
\leq {\mathbb E}|S_g(\Sigma;E) -M|
\\
&
\nonumber
\lesssim_{s} L_{g,s} \|\Sigma\|^s \biggl(
\Bigl(\frac{{\bf r}(\Sigma)}{n}\Bigr)^{(s-1)/2} \bigvee 
\Bigl(\frac{{\bf r}(\Sigma)}{n}\Bigr)^{s-1/2}\bigvee 
\Bigl(\frac{1}{n}\Bigr)^{(s-1)/2}\biggr)
\sqrt{\frac{1}{n}}.
\end{align}
Combining (\ref{conc_med_A}) and (\ref{mean_med}) implies that, for all $t\geq 1,$ with 
probability at least $1-e^{-t}$
\begin{align}
\label{conc_med_B}
&
|S_g(\Sigma;E)-{\mathbb E} S_g(\Sigma;E)|
\\
&
\nonumber
\lesssim_{s} L_{g,s} \|\Sigma\|^s \biggl(
\Bigl(\frac{{\bf r}(\Sigma)}{n}\Bigr)^{(s-1)/2} \bigvee 
\Bigl(\frac{{\bf r}(\Sigma)}{n}\Bigr)^{s-1/2}
\bigvee 
\Bigl(\frac{t}{n}\Bigr)^{(s-1)/2}\bigvee \Bigl(\frac{t}{n}\Bigr)^{s-1/2}\biggr)
\sqrt{\frac{t}{n}},
\end{align}
which completes the proof.

\end{proof}

\begin{example}
Consider the following functional 
$$
g(A):= \langle f(A),B\rangle = {\rm tr}(f(A)B^{\ast}), A\in {\mathcal B}_{sa}({\mathbb H}),
$$
where $f$ is a given smooth function and $B\in {\mathcal S}_1$ is a given nuclear operator.

\begin{corollary}
\label{corr_f_B}
If $f\in B_{\infty,1}^s({\mathbb R})$ for some $s\in (1,2],$ then with probability at least $1-e^{-t}$
the following concentration inequality holds for the functional $g:$
\begin{align}
\label{conc_med_B_nuclear}
&
|S_g(\Sigma;\hat \Sigma-\Sigma)-{\mathbb E} S_g(\Sigma;\hat \Sigma-\Sigma)|
\\
&
\nonumber
\lesssim_{s} \|f\|_{B_{\infty,1}^s} \|B\|_1\|\Sigma\|^s \biggl(
\Bigl(\frac{{\bf r}(\Sigma)}{n}\Bigr)^{(s-1)/2} \bigvee 
\Bigl(\frac{{\bf r}(\Sigma)}{n}\Bigr)^{s-1/2}
\bigvee 
\Bigl(\frac{t}{n}\Bigr)^{(s-1)/2}\bigvee \Bigl(\frac{t}{n}\Bigr)^{s-1/2}\biggr)
\sqrt{\frac{t}{n}}.
\end{align}
\end{corollary}

\begin{proof}
It easily follows from Lemma \ref{lemma_Lipsch_Bes} that Assumption \ref{assume_Lipschitz} is satisfied 
for $s\in [1,2]$ with $L_{g,s}= 2^{s+1}\|f\|_{B_{\infty,1}^s} \|B\|_1.$ Therefore, 
Theorem \ref{th-conc-med} implies bound (\ref{conc_med_B_nuclear}). 
 
\end{proof}

\end{example}

In what follows, we need a more general version of the bound of Theorem
\ref{th-conc-med}
(under somewhat more general conditions than Assumption \ref{assume_Lipschitz}). 

\begin{assumption}
\label{assume_Lipschitz_ABC}
Assume that, for all $\Sigma \in {\mathcal C}_+({\mathbb H}), H,H'\in {\mathcal B}_{sa}({\mathbb H}),$
$$
|S_g(\Sigma;H')-S_g(\Sigma;H)|\leq \eta(\Sigma; \|H\|\vee \|H'\|) \|H'-H\|,
$$
where $0<\delta\mapsto \eta (\Sigma;\delta)$ is a nondecreasing function 
of the following form:
$$
\eta(\Sigma;\delta) := \eta_1(\Sigma) \delta^{\alpha_1}\bigvee \dots \bigvee \eta_m(\Sigma)
\delta^{\alpha_m},
$$
for given nonnegative functions $\eta_1,\dots, \eta_m$ on ${\mathcal C}_+({\mathbb H})$ and given positive numbers $\alpha_1,\dots \alpha_m.$
\end{assumption}

The proof of the following result is a simple modification of the 
proof of Theorem \ref{th-conc-med}.

\begin{theorem}
\label{conc-med-gen}
Suppose Assumption \ref{assume_Lipschitz_ABC} holds.
Then, for all $t\geq 1$ with probability 
at least $1-e^{-t},$
\begin{align}
\label{conc_med_state_general}
&
|S_g(\Sigma;\hat \Sigma-\Sigma)-{\mathbb E} S_g(\Sigma;\hat \Sigma-\Sigma)|
\lesssim_{\eta}  \eta(\Sigma;\delta_n(\Sigma;t))\Bigl(\sqrt{\|\Sigma\|}+\sqrt{\delta_n(\Sigma;t)})
\sqrt{\|\Sigma\|}\sqrt{\frac{t}{n}},
\end{align}
where 
\begin{equation}
\label{def_delta_n}
\delta_n(\Sigma;t):= \|\Sigma\|\biggl(\sqrt{\frac{{\bf r}(\Sigma)}{n}}\bigvee \frac{{\bf r}(\Sigma)}{n} \bigvee \sqrt{\frac{t}{n}}\bigvee \frac{t}{n}\biggr).
\end{equation}
\end{theorem}

\section{Normal approximation bounds for plug-in estimators}
\label{Sec:Normal approximation bounds}

Let $g:{\mathcal B}_{sa}({\mathbb H})\mapsto {\mathbb R}$ be a Fr\'echet 
differentiable function with respect to the operator norm with derivative 
$Dg(A;H), H\in {\mathcal B}_{sa}({\mathbb H}).$
Recall that $Dg(A;H)= \langle Dg(A),H\rangle, H\in {\mathcal C}_{sa}({\mathbb H}),$
where $Dg(A)\in {\mathcal S}_1.$
Denote 
$$
{\mathcal D}g(\Sigma):= \Sigma^{1/2} Dg(\Sigma) \Sigma^{1/2}.
$$

The following theorem is the main result of this section.

\begin{theorem}
\label{th_norm_approx}
Suppose Assumption \ref{assume_Lipschitz} holds for some $s\in (1,2]$
and also that ${\bf r}(\Sigma)\leq n.$
Define 
$
\gamma_s(g;\Sigma):= \log \biggl(\frac{L_{g,s}\|\Sigma\|^s}{\|{\mathcal D}g(\Sigma)\|_2}\biggr)
$
and 
$$
t_{n,s}(g;\Sigma):= \biggl[-\gamma_s(g;\Sigma)+\frac{s-1}{2}\log \biggl(\frac{n}{{\bf r}(\Sigma)}\biggr)\biggr] \bigvee 1.
$$
Then 
\begin{align}
\label{Ber-Ess_main}
&
\sup_{x\in {\mathbb R}}\biggl|{\mathbb P}
\biggl\{\frac{n^{1/2}(g(\hat \Sigma)-{\mathbb E}g(\hat \Sigma))}
{\sqrt{2}\|{\mathcal D}g(\Sigma)\|_2}\leq x\biggr\}-\Phi(x)\biggr|
\lesssim_s \biggl(\frac{\|{\mathcal D} g(\Sigma)\|_3}{\|\mathcal D g(\Sigma)\|_2}\biggr)^3\frac{1}{\sqrt{n}}
\\
&
\nonumber
+
\frac{L_{g,s}\|\Sigma\|^s}{\|{\mathcal D}g(\Sigma)\|_2}
\biggl(
\Bigl(\frac{{\bf r}(\Sigma)}{n}\Bigr)^{(s-1)/2} 
\bigvee 
\Bigl(\frac{t_{n,s}(g;\Sigma)}{n}\Bigr)^{(s-1)/2}\bigvee \Bigl(\frac{t_{n,s}(g;\Sigma)}{n}\Bigr)^{s-1/2}\biggr)\sqrt{t_{n,s}(g;\Sigma)}.
\end{align}
\end{theorem}

\begin{proof}
Note that 
$$
g(\hat \Sigma)-g(\Sigma)= \langle Dg(\Sigma), \hat \Sigma-\Sigma\rangle 
+ S_g (\Sigma; \hat \Sigma-\Sigma)
$$
and, since ${\mathbb E}\langle Dg(\Sigma), \hat \Sigma-\Sigma\rangle =0,$
$$
{\mathbb E}g(\hat \Sigma)-g(\Sigma)= 
S_g (\Sigma; \hat \Sigma-\Sigma)-{\mathbb E}S_g (\Sigma; \hat \Sigma-\Sigma),
$$
implying that 
\begin{equation}
\label{ghatsigma}
g(\hat \Sigma)-{\mathbb E}g(\hat \Sigma)=
\langle Dg(\Sigma), \hat \Sigma-\Sigma\rangle 
+ S_g (\Sigma; \hat \Sigma-\Sigma)-{\mathbb E}S_g (\Sigma; \hat \Sigma-\Sigma).
\end{equation}
The linear term 
\begin{equation}
\label{linear_term}
\langle Dg(\Sigma), \hat \Sigma-\Sigma\rangle = 
n^{-1} \sum_{j=1}^n \langle Dg(\Sigma)X_j, X_j \rangle - {\mathbb  E}\langle Dg(\Sigma)X, X \rangle
\end{equation}
is the sum of i.i.d. random variables and it will be approximated by a normal 
distribution. 

We will need the following simple lemma.

\begin{lemma}
\label{repra}
Let $A\in {\mathcal S}_1$ be a self-adjoint trace class operator. 
Denote by $\lambda_j, j\geq 1$ the eigenvalues of the operator 
$\Sigma^{1/2}A\Sigma^{1/2}$ (repeated with their multiplicities and, to be specific, such that their absolute values are arranged in a non-increasing order). Then 
$$
\langle A X,X\rangle
\stackrel{d}{=} \sum_{k\geq 1} \lambda_k Z_k^2,
$$
where $Z_1,Z_2, \dots$ are i.i.d. standard normal random variables.
\end{lemma}

\begin{proof}
First assume that $\Sigma$ is a finite rank operator, or, equivalently, 
that $X$ takes values in a finite dimensional subspace $L$ of ${\mathbb H}.$
In this case, $X=\Sigma^{1/2}Z,$ where $Z$ is a standard normal vector 
in $L.$ Therefore,
$$
\langle A X,X\rangle = \langle A \Sigma^{1/2}Z, \Sigma^{1/2}Z\rangle
=\langle \Sigma^{1/2}A\Sigma^{1/2}Z,Z\rangle 
=\sum_{k\geq 1} \lambda_k Z_k^2,
$$
where $\{Z_k\}$ are the coordinates of $Z$ in the basis of eigenvectors 
of $\Sigma^{1/2}A\Sigma^{1/2}.$  

In the general infinite dimensional case the result follows by a standard 
finite dimensional approximation.

\end{proof}

Note that 
$
{\mathbb E}\langle A X,X\rangle = \sum_{k\geq 1}\lambda_k = 
{\rm tr}(\Sigma^{1/2}A\Sigma^{1/2})
$
and 
$$
{\rm Var}(\langle A X,X\rangle) = \sum_{k\geq 1}\lambda_k^2 {\mathbb E}(Z_k^2-1)^2
= 2 \sum_{k\geq 1}\lambda_k^2 = 2\|\Sigma^{1/2}A\Sigma^{1/2}\|_2^2.
$$

The following result immediately follows from Berry-Esseen bound
(see \cite{Petrov}, Chapter 5, Theorem 3; an extension of the inequality 
to infinite sums of independent r.v. is based on a straightforward approximation argument).

\begin{lemma}
\label{Berry-Esseen}
The following bound holds:
$$
\sup_{x\in {\mathbb R}}\biggl|{\mathbb P}
\biggl\{\frac{n^{1/2}\langle Dg(\Sigma), \hat \Sigma-\Sigma\rangle}
{\sqrt{2}\|{\mathcal D}g(\Sigma)\|_2}\leq x\biggr\}-\Phi(x)\biggr|
\lesssim \biggl(\frac{\|{\mathcal D} g(\Sigma)\|_3}{\|\mathcal D g(\Sigma)\|_2}\biggr)^3\frac{1}{\sqrt{n}}.
$$
\end{lemma}

\begin{proof}
Indeed, by (\ref{linear_term}) and Lemma \ref{repra} with $A=Dg(\Sigma),$
\begin{equation}
\label{=d}
\frac{n^{1/2}\langle Dg(\Sigma), \hat \Sigma-\Sigma\rangle}
{\sqrt{2}\|{\mathcal D}g(\Sigma)\|_2}
\stackrel{d}{=}
\frac{\sum_{j=1}^n \sum_{k\geq 1} \lambda_k (Z_{k,j}^2-1)}{{\rm Var}^{1/2}\biggl(\sum_{j=1}^n \sum_{k\geq 1} \lambda_k (Z_{k,j}^2-1)\biggr)},
\end{equation}
where $\{Z_{k,j}\}$ are i.i.d. standard normal random variables.
By Berry-Esseen bound,
$$
\sup_{x\in {\mathbb R}}\biggl|{\mathbb P}\biggl\{
\frac{\sum_{j=1}^n \sum_{k\geq 1} \lambda_k (Z_{k,j}^2-1)}{{\rm Var}^{1/2}\biggl(\sum_{j=1}^n \sum_{k\geq 1} \lambda_k (Z_{k,j}^2-1)\biggr)}
\biggr\}
-\Phi(x)\biggr|
$$
$$
\lesssim \frac{\sum_{j=1}^n \sum_{k\geq 1} |\lambda_k|^3 {\mathbb E}|Z_{k,j}^2-1|^3}
{\biggl(\sum_{j=1}^n\sum_{k\geq 1} \lambda_k^2 {\mathbb E}(Z_{k,j}^2-1)^2\biggr)^{3/2}}
\lesssim \frac{\sum_{k\geq 1}|\lambda_k|^3}{\biggl(\sum_{k\geq 1}\lambda_k^2\biggr)^{3/2}}\frac{1}{\sqrt{n}}
\lesssim
\biggl(\frac{\|{\mathcal D} g(\Sigma)\|_3}{\|\mathcal D g(\Sigma)\|_2}\biggr)^3\frac{1}{\sqrt{n}}.
$$

\end{proof}

Finally, the following lemma will be also used.

\begin{lemma}
\label{xi_eta}
For random variables $\xi, \eta,$ 
denote 
$$
\Delta (\xi, \eta):=\sup_{x\in {\mathbb R}}|{\mathbb P}\{\xi \leq x\}-{\mathbb P}\{\eta\leq x\}| 
\ \
{\rm and}\ \  
\delta (\xi, \eta):= \inf_{\delta>0}\Bigl[{\mathbb P}\{|\xi-\eta|\geq \delta\}+\delta\Bigr].
$$
Then, for a standard normal random variable $Z,$
$
\Delta (\xi,Z) \leq \Delta (\eta,Z)+ \delta(\xi,\eta).
$
\end{lemma}

\begin{proof}
For all $x\in {\mathbb R},$ $\delta>0,$
\begin{align*}
&
{\mathbb P}\{\xi \leq x\}
\leq 
{\mathbb P}\{\xi \leq x, |\xi-\eta|<\delta\}+ {\mathbb P}\{|\xi-\eta|\geq \delta\}
\\
&
\leq {\mathbb P}\{\eta \leq x+\delta\}+ {\mathbb P}\{|\xi-\eta|\geq \delta\}
\\
&
\leq {\mathbb P}\{Z\leq x+\delta\}+ \Delta (\eta,Z)+{\mathbb P}\{|\xi-\eta|\geq \delta\}
\\
&
\leq {\mathbb P}\{Z\leq x\}+ \delta+ \Delta (\eta,Z)+{\mathbb P}\{|\xi-\eta|\geq \delta\},
\end{align*}
where we used a trivial bound 
${\mathbb P}\{Z\leq x+\delta\}-{\mathbb P}\{Z\leq x\}\leq \delta.$ 
Thus,
$$
{\mathbb P}\{\xi \leq x\}- {\mathbb P}\{Z\leq x\}
\leq \Delta (\eta,Z) + {\mathbb P}\{|\xi-\eta|\geq \delta\}+\delta.
$$
Similarly, 
$$
{\mathbb P}\{\xi \leq x\}- {\mathbb P}\{Z\leq x\}
\geq -\Delta (\eta,Z) - {\mathbb P}\{|\xi-\eta|\geq \delta\}-\delta,
$$
implying that for all $\delta>0$
$
\Delta (\xi, Z)\leq  \Delta (\eta,Z) + {\mathbb P}\{|\xi-\eta|\geq \delta\}+\delta.
$
Taking the infimum over $\delta>0$ yields the claim of the lemma.

\end{proof}

We apply the last lemma to random variables 
$$
\xi := \frac{n^{1/2}(g(\hat \Sigma)-{\mathbb E}g(\hat \Sigma))}{\sqrt{2}\|{\mathcal D}
g(\Sigma)\|_2}\ \ {\rm and}\ \  \eta:=\frac{n^{1/2}\langle Dg(\Sigma), \hat \Sigma-\Sigma\rangle}
{\sqrt{2}\|{\mathcal D}g(\Sigma)\|_2}.
$$
By (\ref{ghatsigma}),
$$
\xi-\eta = \frac{n^{1/2}(S_g (\Sigma; \hat \Sigma-\Sigma)-{\mathbb E}S_g (\Sigma; \hat \Sigma-\Sigma))}{\sqrt{2}\|{\mathcal D}g(\Sigma)\|_2}.
$$
Recall that Assumption \ref{assume_Lipschitz} holds and ${\bf r}(\Sigma)\leq n,$ and denote 
$$
\delta_{n,s} (g;\Sigma;t):= 
K_s L_{g,s} \frac{\|\Sigma\|^s}{\sqrt{2}\|{\mathcal D}g(\Sigma)\|_2} \biggl(
\Bigl(\frac{{\bf r}(\Sigma)}{n}\Bigr)^{(s-1)/2} 
\bigvee 
\Bigl(\frac{t}{n}\Bigr)^{(s-1)/2}\bigvee \Bigl(\frac{t}{n}\Bigr)^{s-1/2}\biggr)\sqrt{t}.
$$
It immediately follows from Theorem \ref{th-conc-med} that 
$
{\mathbb P} \{|\xi-\eta|\geq \delta_{n,s} (g;\Sigma;t)\}\leq e^{-t}, t\geq 1
$
and, as a consequence,
$$
\delta (\xi,\eta)\leq \inf_{t\geq 1}\Bigl[ \delta_{n,s} (g;\Sigma;t)+ e^{-t}\Bigr].
$$
It follows from lemmas \ref{Berry-Esseen} and \ref{xi_eta} that, for some $C>0,$
\begin{align}
\label{Ber-Ess}
&
\nonumber
\sup_{x\in {\mathbb R}}\biggl|{\mathbb P}
\biggl\{\frac{n^{1/2}(g(\hat \Sigma)-{\mathbb E}g(\hat \Sigma))}
{\sqrt{2}\|{\mathcal D}g(\Sigma)\|_2}\leq x\biggr\}-\Phi(x)\biggr|
\\
&
\leq C \biggl(\frac{\|{\mathcal D} g(\Sigma)\|_3}{\|\mathcal D g(\Sigma)\|_2}\biggr)^3\frac{1}{\sqrt{n}}
+
\inf_{t\geq 1}\Bigl[ \delta_{n,s} (g;\Sigma;t)\}+ e^{-t}\Bigr].
\end{align}
Recall that 
$
\gamma_s(g;\Sigma)= \log \biggl(\frac{L_{g,s}\|\Sigma\|^s}{\|{\mathcal D}g(\Sigma)\|_2}\biggr)
$
and 
$$
t_{n,s}(g;\Sigma)= \biggl[-\gamma_s(g;\Sigma)+\frac{s-1}{2}\log \biggl(\frac{n}{{\bf r}(\Sigma)}\biggr)\biggr] \bigvee 1.
$$
Let $\bar t:= t_{n,s}(g;\Sigma).$
Then 
$$
e^{-\bar t} \leq 
L_{g,s} \frac{\|\Sigma\|^s\sqrt{\bar t}}{\sqrt{2}\|{\mathcal D}g(\Sigma)\|_2}\Bigl(\frac{{\bf r}(\Sigma)}{n}\Bigr)^{(s-1)/2} \lesssim_s \delta_{n,s} (g;\Sigma;\bar t).
$$
Therefore, 
$$
\inf_{t\geq 1}\Bigl[ \delta_{n,s} (g;\Sigma;t)\}+ e^{-t}\Bigr]
\lesssim_s \delta_{n,s} (g;\Sigma;\bar t)
$$
$$
\lesssim_s 
\frac{L_{g,s}\|\Sigma\|^s}{\|{\mathcal D}g(\Sigma)\|_2}
\biggl(
\Bigl(\frac{{\bf r}(\Sigma)}{n}\Bigr)^{(s-1)/2} 
\bigvee 
\Bigl(\frac{t_{n,s}(g;\Sigma)}{n}\Bigr)^{(s-1)/2}\bigvee \Bigl(\frac{t_{n,s}(g;\Sigma)}{n}\Bigr)^{s-1/2}\biggr)\sqrt{t_{n,s}(g;\Sigma)}.
$$
Substituting this into bound (\ref{Ber-Ess}) completes the proof of Theorem \ref{th_norm_approx}.

\end{proof}

Our main example of interest is the functional $g(A):= \langle f(A), B\rangle, A\in {\mathcal B}_{sa}({\mathbb H}),$ where $f$ is a given smooth function and $B\in {\mathcal S}_1({\mathbb H})$ is a given nuclear operator. If $f\in B^1_{\infty, 1}({\mathbb R}),$
then the function $A\mapsto f(A)$ is operator differentiable implying the 
differentiability of the functional $A\mapsto g(A)$ with derivative 
$
Dg(A;H)= \langle Df(A;H), B\rangle, A, H\in {\mathcal B}_{sa}({\mathbb H}).
$
Moreover, for $A=\Sigma$ with spectral decomposition
$
\Sigma = \sum_{\lambda\in \sigma(\Sigma)}\lambda P_{\lambda},
$
formula (\ref{Loewner}) holds implying that ${\mathcal C}_{sa}({\mathbb H})\ni H\mapsto Df(\Sigma;H)=Df(\Sigma)H\in {\mathcal B}_{sa}({\mathbb H})$ is a symmetric operator:
$
\langle Df(\Sigma) H_1,H_2\rangle = \langle H_1, Df(\Sigma)H_2\rangle,
H_1\in {\mathcal C}_{sa}({\mathbb H}), H_2\in {\mathcal S}_1({\mathbb H}).
$
Therefore, 
$$
Dg(\Sigma;H)= \langle Df(\Sigma;B), H\rangle, H\in {\mathcal C}_{sa}({\mathbb H}),
$$
or, in other words, $Dg(\Sigma)= Df(\Sigma;B).$
Denote 
$$
\sigma_f (\Sigma;B):= \sqrt{2}\|\Sigma^{1/2}Df(\Sigma;B)\Sigma^{1/2}\|_2
\ \
{\rm and} 
\ \
\mu^{(3)}_f (\Sigma;B)= \|\Sigma^{1/2}Df(\Sigma;B)\Sigma^{1/2}\|_3.
$$
The following result is a simple consequence of Theorem \ref{th_norm_approx}
and Corollary \ref{corr_f_B}.

\begin{corollary}
\label{corr_f_B_norm}
Let $f\in B_{\infty,1}^s({\mathbb R})$ for some $s\in (1,2].$
Define 
$$
\gamma_s(f;\Sigma):= \log \biggl(\frac{2^{s+3/2}\|f\|_{B^{s}_{\infty,1}}\|B\|_1\|\Sigma\|^s}
{\sigma_f(\Sigma;B)}\biggr)
$$
and 
$$
t_{n,s}(f;\Sigma):= \biggl[-\gamma_s(f;\Sigma)+\frac{s-1}{2}\log \biggl(\frac{n}{{\bf r}(\Sigma)}\biggr)\biggr] \bigvee 1.
$$
Then 
\begin{align}
\label{Ber-Ess_f_B}
&
\nonumber
\sup_{x\in {\mathbb R}}\biggl|{\mathbb P}
\biggl\{\frac{n^{1/2}\Bigl\langle f(\hat \Sigma)-{\mathbb E}f(\hat \Sigma)),B\Bigr\rangle}
{\sigma_f(\Sigma,B)}\leq x\biggr\}-\Phi(x)\biggr|
\lesssim_s 
\Delta_{n}^{(s)}(f;\Sigma;B):=
\biggl(\frac{\mu^{(3)}_f(\Sigma;B)}
{\sigma_f(\Sigma;B)}\biggr)^3\frac{1}{\sqrt{n}}
\\
&
+
\frac{\|f\|_{B^{s}_{\infty,1}}\|B\|_1\|\Sigma\|^s}
{\sigma_f(\Sigma;B)}
\biggl(
\Bigl(\frac{{\bf r}(\Sigma)}{n}\Bigr)^{(s-1)/2} 
\bigvee 
\Bigl(\frac{t_{n,s}(f;\Sigma)}{n}\Bigr)^{(s-1)/2}\bigvee \Bigl(\frac{t_{n,s}(f;\Sigma)}{n}\Bigr)^{s-1/2}\biggr)\sqrt{t_{n,s}(f;\Sigma)}.
\end{align}
\end{corollary}

We will now prove Theorem \ref{Th:normal_approximation_uniform} 
and Corollary \ref{r_small} from Section \ref{Sec:Intro}. 

\begin{proof} The proof of \eqref{normal_approximation_uniform} immediately
follows from bound \eqref{Ber-Ess_f_B}. It is also easy to prove  \eqref{normal_approximation_uniform_r_small} using \eqref{normal_approximation_uniform}, the bound on the bias
\begin{align}
\label{bias_bou_bou_bou}
&
\|{\mathbb E}_{\Sigma} f(\hat \Sigma)- f(\Sigma)\|
=\|{\mathbb E}_{\Sigma}S_f(\Sigma;\hat \Sigma-\Sigma)\|
\lesssim \|f\|_{B^s_{\infty,1}} {\mathbb E}\|\hat \Sigma-\Sigma\|^s
\lesssim \|f\|_{B^s_{\infty,1}} \|\Sigma\|^s\biggl(\frac{{\bf r}(\Sigma)}{n}\biggr)^{s/2}
\end{align}
and Lemma \ref{xi_eta}.	

The proof of  \eqref{normal_approximation_loss_f} is a bit more involved and requires a couple of more lemmas. 
The following fact is well known 
(it follows, e.g., from \cite{Vershynin}, Proposition 5.16). 

\begin{lemma}
\label{cite_Vershynin}
Let $\{\xi_i\}$ be i.i.d. standard normal random variables 
and let $\{\gamma_i\}$ be real numbers. Then, for all $t\geq 0$ 
with probability at least $1-e^{-t}$ 
$$
\biggl|\sum_{i\geq 1}\gamma_i (\xi_i^2-1)\biggr|
\lesssim \biggl(\sum_{i\geq 1} \gamma_i^2\biggr)^{1/2} \sqrt{t}\bigvee 
\sup_{i\geq 1}|\gamma_i| t.
$$
\end{lemma}

\begin{lemma}
\label{bd_exp_fhat}
If, for some $s\in (1,2],$ $f\in B_{\infty,1}^s({\mathbb R})$
and  ${\bf r}(\Sigma)\leq n,$ then, for all $t\geq 1$ with probability 
at least $1-e^{-t}$ 
\begin{align}
&
\nonumber
\label{bd_exp_fhat_bd}
\biggl|\frac{n^{1/2}\langle f(\hat \Sigma)-f(\Sigma), B\rangle}
{\sigma_f(\Sigma;B)}\biggr|
\\
&
\lesssim_s \biggl(\frac{\|f\|_{B_{\infty,1}^s} \|B\|_1\|\Sigma\|^s}
{\sigma_f(\Sigma;B)}\bigvee \frac{\|f\|_{L_\infty} \|B\|_1}{\sigma_f(\Sigma;B)}\bigvee 1\biggr)\biggl(\sqrt{t}\bigvee \frac{({\bf r}(\Sigma))^{s/2}}{n^{(s-1)/2}}\biggr).
\end{align}
\end{lemma}

\begin{proof}
Recall that 
$$
\langle f(\hat \Sigma)-{\mathbb E}f(\hat \Sigma),B\rangle 
= \langle Df(\Sigma;\hat \Sigma-\Sigma), B\rangle 
+ \langle S_f(\Sigma;\hat \Sigma-\Sigma)- {\mathbb E}S_f(\Sigma;\hat \Sigma-\Sigma),B\rangle.
$$
It follows from \eqref{=d} that 
\begin{equation}
\label{=d''}
\frac{n^{1/2}\langle Df(\Sigma;\hat \Sigma-\Sigma), B\rangle}
{\sigma_f(\Sigma;B)}
\stackrel{d}{=}
\frac{\sum_{j=1}^n \sum_{k\geq 1} \lambda_k (Z_{k,j}^2-1)}{{\rm Var}^{1/2}\biggl(\sum_{j=1}^n \sum_{k\geq 1} \lambda_k (Z_{k,j}^2-1)\biggr)},
\end{equation}
where $Z_{k,j}$ are i.i.d. standard normal r.v. and $\lambda_k$
are the eigenvalues (repeated with their multiplicities) of $\Sigma^{1/2}Df(\Sigma;B)\Sigma^{1/2}.$
Using Lemma \ref{cite_Vershynin}, we easily get that 
for all $t\geq 1$ with probability at least $1-e^{-t},$
\begin{equation}
\label{lin_exp}
\biggl|\frac{n^{1/2}\langle Df(\Sigma;\hat \Sigma-\Sigma), B\rangle}
{\sigma_f(\Sigma;B)}\biggr|\lesssim \sqrt{t}\vee \frac{t}{\sqrt{n}}.
\end{equation}
To control $\langle S_f(\Sigma;\hat \Sigma-\Sigma)- {\mathbb E}S_f(\Sigma;\hat \Sigma-\Sigma),B\rangle,$ we use bound \eqref{conc_med_B_nuclear}
to get that for all $t\geq 1$ with probability at least $1-e^{-t}$
\begin{align}
\label{conc_med_B_nuclear''}
&
|\langle S_f(\Sigma;\hat \Sigma-\Sigma)-{\mathbb E} S_f(\Sigma;\hat \Sigma-\Sigma), B\rangle|
\\
&
\nonumber
\lesssim_{s} \|f\|_{B_{\infty,1}^s} \|B\|_1\|\Sigma\|^s \biggl(
\Bigl(\frac{{\bf r}(\Sigma)}{n}\Bigr)^{(s-1)/2} \bigvee 
\Bigl(\frac{{\bf r}(\Sigma)}{n}\Bigr)^{s-1/2}
\bigvee 
\Bigl(\frac{t}{n}\Bigr)^{(s-1)/2}\bigvee \Bigl(\frac{t}{n}\Bigr)^{s-1/2}\biggr)
\sqrt{\frac{t}{n}}.
\end{align}
If ${\bf r}(\Sigma)\leq n$ and $t\leq n,$ bounds \eqref{lin_exp},\eqref{conc_med_B_nuclear''} and \eqref{bias_bou_bou_bou} easily imply that 
with probability at least $1-e^{-t}$
\begin{equation}
\label{bd<n}
\biggl|\frac{n^{1/2}\langle f(\hat \Sigma)-f(\Sigma), B\rangle}
{\sigma_f(\Sigma;B)}\biggr|
\lesssim_s \biggl(\frac{\|f\|_{B_{\infty,1}^s} \|B\|_1\|\Sigma\|^s}
{\sigma_f(\Sigma;B)}\bigvee 1\biggr)\biggl(\sqrt{t}\bigvee \frac{({\bf r}(\Sigma))^{s/2}}{n^{(s-1)/2}}\biggr).
\end{equation}
Note also that, for all $t>n,$
\begin{equation}
\label{bd>n}
\biggl|\frac{n^{1/2}\langle f(\hat \Sigma)-f(\Sigma), B\rangle}
{\sigma_f(\Sigma;B)}\biggr| \leq \frac{2\|f\|_{L_\infty} \|B\|_1}{\sigma_f(\Sigma;B)}\sqrt{t}.
\end{equation}
The result immediately follows from bounds \eqref{bd<n} and \eqref{bd>n}. 

\end{proof}

\begin{lemma}
\label{loss-bd}
Let $\ell$ be a loss function satisfying Assumption \ref{assump_loss}.
For any random variables $\xi, \eta$ and for all $A>0$
$$
|{\mathbb E}\ell(\xi)-{\mathbb E}\ell(\eta)|\leq 4\ell(A)\Delta(\xi;\eta)+ {\mathbb E}\ell(\xi)I(|\xi|\geq A)+ {\mathbb E}\ell(\eta)I(|\eta|\geq A).
$$
\end{lemma}

\begin{proof}
Clearly,
\begin{equation}
\label{bd_ell}
|{\mathbb E}\ell(\xi)-{\mathbb E}\ell(\eta)|\leq |{\mathbb E}\ell(\xi)I(|\xi|<A)-{\mathbb E}\ell(\eta)I(|\eta|<A)|+ {\mathbb E}\ell(\xi)I(|\xi|\geq A)+ {\mathbb E}\ell(\eta)I(|\eta|\geq A).
\end{equation}
Denoting by $F_{\xi}, F_{\eta}$ the distribution functions of $\xi, \eta,$ 
assuming that $A$ is a continuity point of both $F_{\xi}$ and $F_{\eta}$
and using integration by parts, we get 
$$
|{\mathbb E}\ell(\xi)I(|\xi|<A)-{\mathbb E}\ell(\eta)I(|\eta|<A)|= \Bigl|\int_{-A}^A \ell(x)d(F_{\xi}-F_{\eta})(x)\Bigr|
$$
$$
=\Bigl|\ell(A) (F_{\xi}-F_{\eta})(A)-\ell(-A) (F_{\xi}-F_{\eta})(-A)-\int_{-A}^A (F_{\xi}-F_{\eta})(x) \ell^{\prime}(x)dx\Bigr|.
$$
Using the properties of $\ell$ (in particular, that $\ell$ is an even function and $\ell^{\prime}$ is nonnegative and nondecreasing on ${\mathbb R}_+$),
we get 
$$
|{\mathbb E}\ell(\xi)I(|\xi|<A)-{\mathbb E}\ell(\eta)I(|\eta|<A)|\leq 2\ell(A)\Delta(\xi;\eta)+ 2\int_{0}^A \ell^{\prime}(u)du \Delta(\xi,\eta)=4\ell(A)\Delta(\xi,\eta),
$$
which together with (\ref{bd_ell}) imply the claim.
If $A$ is not a continuity point of $F_{\xi}$ or $F_{\eta},$ one can easily obtain the result 
by a limiting argument.

\end{proof}

The following lemma is elementary. 

\begin{lemma}
\label{Lemma:Bernstein-loss}
Let $\xi$ be a random variable such that for some 
$\tau>0$ 
and for all $t\geq 1$ with probability at least $1-e^{-t}$
\begin{equation}
\label{Bernst_bd}
|\xi| \leq \tau \sqrt{t} 
\end{equation}
Let $\ell$ be a loss function satisfying Assumption \ref{assump_loss}.
Then 
\begin{equation}
\label{Bernstein-loss}
{\mathbb E} \ell^2 (\xi)\leq 
2e\sqrt{2\pi}  c_1^2 e^{2c_2^2 \tau^2}.
\end{equation}
\end{lemma}

We now apply lemmas \ref{loss-bd} and \ref{Lemma:Bernstein-loss} to r.v. 
$$
\xi:= \xi(\Sigma):=\frac{\sqrt{n}\Bigl(\langle f(\hat \Sigma),B\rangle - \langle f(\Sigma), B\rangle\Bigr)}
{\sigma_{f}(\Sigma;B)}
$$ 
and $\eta:=Z.$
Bound \eqref{bd_exp_fhat_bd} and 
Lemma \ref{Lemma:Bernstein-loss} along with the fact that under conditions 
of the theorem $\frac{({\bf r}(\Sigma))^{s/2}}{n^{(s-1)/2}}\leq \frac{r_n^{s/2}}{n^{(s-1)/2}}\leq 1$
(for large enough $n$)  
imply that bounds \eqref{Bernst_bd} and \eqref{Bernstein-loss}
hold with 
$$
\tau:= \frac{\|f\|_{B_{\infty,1}^s} \|B\|_1\|\Sigma\|^s}
{\sigma_f(\Sigma;B)}\bigvee \frac{2\|f\|_{L_\infty} \|B\|_1}{\sigma_f(\Sigma;B)}\bigvee 1.
$$ 
It follows from the bound of Lemma \ref{loss-bd} that 
\begin{equation}
\label{ellxietaell}
|{\mathbb E}\ell(\xi)-{\mathbb E}\ell(Z)|\leq 4\ell(A)\Delta(\xi;Z)+ 
{\mathbb E}^{1/2}\ell^2(\xi){\mathbb P}^{1/2}\{|\xi|\geq A\}
+ {\mathbb E}^{1/2}\ell^2(Z)
{\mathbb P}^{1/2}\{|Z|\geq A\}.
\end{equation}
Using bounds \eqref{Bernstein-loss}, 
standard bounds on
${\mathbb E}\ell^2(Z),$ ${\mathbb P}\{|Z|\geq A\}$ 
and the bound of
Corollary  \ref{corr_f_B_norm}, we get 
\begin{align*}
&
|{\mathbb E}\ell(\xi)-{\mathbb E}\ell(Z)|
\lesssim_s 4 c_1^2 e^{2c_2 A^2}
\Delta_{n}^{(s)}(f;\Sigma;B)
+\sqrt{2e}(2\pi)^{1/4}c_1  e^{c_2^2 \tau^2} 
e^{-A^2/(2\tau^2)} + c_1 e^{c_2^2} e^{-A^2/4}.
\end{align*}
To complete the proof of \eqref{normal_approximation_loss_f},
it remains to take the supremum over the class of covariances 
${\mathcal G}(r_n;a)\cap \{\Sigma: \sigma_f(\Sigma;B)\geq \sigma_0\}$
and over all the operators $B$ with $\|B\|_1\leq 1,$ 
and to pass to the limit first as $n\to\infty$ and then as $A\to\infty.$

\end{proof}

\section{Wishart operators, bootstrap chains, invariant functions and bias reduction}
\label{Sec:Wishart}

In what follows, we assume that ${\mathbb H}$ is a finite-dimensional 
inner product space of dimension ${\rm dim}({\mathbb H})=d.$
Recall that ${\mathcal C}_+({\mathbb H})\subset {\mathcal B}_{sa}({\mathbb H})$ 
denotes the cone of covariance operators in ${\mathbb  H}$ 
and let $L_{\infty}({\mathcal C}_+({\mathbb H}))$ be the space 
of uniformly bounded Borel measurable functions on ${\mathcal C}_+({\mathbb H})$
equipped with the uniform norm. 
Define the following operator ${\mathcal T}:L_{\infty}({\mathcal C}_+({\mathbb H}))\mapsto 
L_{\infty}({\mathcal C}_+({\mathbb H})):$
\begin{equation}
\label{define_calT}
{\mathcal T}g(\Sigma)= {\mathbb E}_{\Sigma}g(\hat \Sigma), 
\Sigma\in {\mathcal C}_+({\mathbb H}),
\end{equation}
where $\hat \Sigma = \hat \Sigma_n := n^{-1}\sum_{j=1}^nX_j\otimes X_j$ is the sample covariance operator based on i.i.d. observations $X_1,\dots, X_n$ sampled from 
$N(0;\Sigma).$ Let $P(\Sigma;\cdot)$ denote the probability distribution of 
$\hat \Sigma$ in the space ${\mathcal C}_+({\mathbb H})$ (equipped with its Borel 
$\sigma-$ algebra ${\frak B}({\mathcal C}_+({\mathbb H}))$). Note that 
$P(\Sigma; n^{-1}A), A\in {\frak B}({\mathcal C}_+({\mathbb H}))$ is a Wishart distribution ${\mathcal W}_d(\Sigma;n).$ Clearly, $P$ is a Markov kernel, 
$$
{\mathcal T}g(\Sigma)= \int_{{\mathcal C}_+({\mathbb H})} g(V) P(\Sigma;dV), g\in L_{\infty}({\mathcal C}_+({\mathbb H}))
$$
and operator ${\mathcal T}$ is a contraction: 
$\|{\mathcal T}g\|_{L_{\infty}}\leq 
\|g\|_{L_{\infty}}.$

Let $\hat \Sigma^{0}:=\Sigma,$ $\hat \Sigma^{(1)}:=\hat \Sigma$ 
and, more generally, given $\hat \Sigma^{(k)},$ 
define $\hat \Sigma^{(k+1)}$ as the sample covariance
based on $n$ i.i.d. observations $X_1^{(k)},\dots, X_n^{(k)}$
sampled from $N(0;\hat \Sigma^{(k)}).$ Then $\hat \Sigma^{(k)}, k\geq 0$
is a homogeneous Markov chain with values in ${\mathcal C}_+({\mathbb H}),$ with $\hat \Sigma^{(0)}=\Sigma$ and with transition probability kernel $P.$ The operator 
${\mathcal T}^k$ can be represented as 
$$
{\mathcal T}^k g(\Sigma)= {\mathbb E}_{\Sigma}g(\hat \Sigma^{(k)})
$$
$$
= \int_{{\mathcal C}_+({\mathbb H})}\dots \int_{{\mathcal C}_+({\mathbb H})}
g(V_k) P(V_{k-1};dV_k) P(V_{k-2};dV_{k-1})\dots P(V_1;dV_2)P(\Sigma;dV_1),
\Sigma \in {\mathcal C}_+({\mathbb H}).
$$
In what follows, we will be interested in operator ${\mathcal B}={\mathcal T}-{\mathcal I}$ that can be called 
\it a bias operator \rm since ${\mathcal B}g(\Sigma)$ represents the bias of the plug-in 
estimator $g(\hat \Sigma)$ of $g(\Sigma):$
$$
{\mathcal B}g(\Sigma)= {\mathbb E}_{\Sigma}g(\hat \Sigma)- g(\Sigma), \Sigma \in 
{\mathcal C}_+({\mathbb H}).
$$
Note that, by Newton's binomial formula, ${\mathcal B}^{k} g(\Sigma)$ can be represented as follows 
\begin{align}
\label{Bkrepr_1}
&
{\mathcal B}^{k} g(\Sigma) = ({\mathcal T}-{\mathcal I})^{k} g(\Sigma)=
\sum_{j=0}^k (-1)^{k-j}{k\choose j} {\mathcal T}^{j} g(\Sigma)
\\
&
\nonumber
= {\mathbb E}_{\Sigma}\sum_{j=0}^k (-1)^{k-j}{k\choose j} g(\hat \Sigma^{(j)}),
\end{align}
which could be viewed as the expectation of the $k$-th order difference 
of function $g$ along the sample path of Markov chain $\hat \Sigma^{(t)}, t=0,1, \dots.$

Denote 
\begin{equation}
\label{def_g_k}
g_k(\Sigma):= \sum_{j=0}^k (-1)^{j} {\mathcal B}^{j} g(\Sigma), \Sigma\in 
{\mathcal C}_+({\mathbb H}).
\end{equation}

\begin{proposition}
\label{prop_g_k}
The bias of estimator $g_k(\hat \Sigma)$ of $g(\Sigma)$
is given by the following formula:
\begin{equation}
\label{bias_g_k}
\nonumber
{\mathbb E}_{\Sigma}g_k(\hat \Sigma)-g(\Sigma) = 
(-1)^k {\mathcal B}^{k+1} g(\Sigma).
\end{equation}
\end{proposition}

\begin{proof}
Indeed,
$$
{\mathbb E}_{\Sigma}g_k(\hat \Sigma)-g(\Sigma) = {\mathcal T}g_k (\Sigma)-g(\Sigma)
=({\mathcal I}+{\mathcal B})g_k(\Sigma)-g(\Sigma)
$$
$$
=
\sum_{j=0}^k (-1)^{j} {\mathcal B}^{j} g(\Sigma)-
\sum_{j=1}^{k+1} (-1)^{j} {\mathcal B}^{j} g(\Sigma) -g(\Sigma)
=(-1)^k {\mathcal B}^{k+1} g(\Sigma).
$$

\end{proof}

Let now $L_{\infty}({\mathcal C}_+({\mathbb H}); {\mathcal B}_{sa}(\mathbb H))$
be the space of uniformly bounded Borel measurable functions 
$g: {\mathcal C}_+({\mathbb H})\mapsto {\mathcal B}_{sa}({\mathbb H}).$
We will need a version of the linear operator defined by formula \eqref{define_calT}
acting from the space $L_{\infty}({\mathcal C}_+({\mathbb H}); {\mathcal B}_{sa}(\mathbb H))$
into itself. With a little abuse of notation, 
we still denote it by ${\mathcal T}$ and also set ${\mathcal B}:={\mathcal T}-{\mathcal I}.$ These operators satisfy all the properties 
stated above. This allows one to define operator valued function $g_k$ by (\ref{def_g_k})
for which Proposition \ref{prop_g_k} still holds. In what follows, it should be clear from 
the context whether ${\mathcal T}$ and ${\mathcal B}$ act on real valued, or on 
operator valued functions.  

Given a smooth function $f$ in real line, we would like to find an estimator of $f(\Sigma)$
with a small bias. To this end, we consider an estimator $f_k(\hat \Sigma)$ and, in view of 
Proposition \ref{prop_g_k}, we need to show that, for a proper choice 
of $k$ (depending on $\alpha$ such that $d={\rm dim}({\mathbb H})\leq n^{\alpha}$), 
$$\|{\mathbb E}_{\Sigma}f_k(\hat \Sigma)-f(\Sigma)\|=\|{\mathcal B}^{k+1} f(\Sigma)\|=o(n^{-1/2}).$$  
At the same time, we need to show that function $f_k$
satisfies certain smoothness properties such as Assumption \ref{assume_Lipschitz_ABC}.
As a consequence, (properly normalized) random variables 
$n^{1/2}\Bigl(\langle f_k(\hat \Sigma),B\rangle-{\mathbb E}_{\Sigma}\langle f_k(\hat \Sigma),B\rangle\Bigr)$ would be close in 
distribution to a standard normal r.v.. Since, in addition, the bias 
${\mathbb E}_{\Sigma}\langle f_k(\hat \Sigma),B\rangle-\langle f(\Sigma),B\rangle$ is of the order $o(n^{-1/2}),$
we would be able to conclude that $\langle f_k(\hat \Sigma),B\rangle$ is an asymptotically normal 
estimator of $\langle f(\Sigma),B\rangle$ with the classical convergence rate 
$n^{-1/2}.$ 

Our approach to this problem is based on representing operator valued function $f_k(\Sigma)$
as $f_k(\Sigma)={\mathcal D}g_k(\Sigma),$ where $g:{\mathcal C}_+({\mathbb H})\mapsto {\mathbb R}$
is a real valued orthogonally invariant function and ${\mathcal D}$ is a differential operator 
defined below and called the {\it lifting operator}. This approach allows us to derive certain 
integral representations for functions ${\mathcal B}^k f(\Sigma)= {\mathcal D}{\mathcal B}^k g(\Sigma)$
that are then used to obtain proper bounds on ${\mathcal B}^k f(\Sigma)$ and to study smoothness 
properties of functions ${\mathcal B}^k f(\Sigma)$ and $f_k(\Sigma).$

A function $g\in L_{\infty}({\mathcal C}_+({\mathbb H}))$
is {\it orthogonally invariant}
iff, for all orthogonal transformations $U$ of ${\mathbb H},$ 
$
g(U\Sigma U^{-1}) = g(\Sigma), \Sigma \in {\mathcal C}_+({\mathbb H}). 
$
Note that any such function $g$ could be represented as 
$g(\Sigma)=\varphi (\lambda_1(\Sigma), \dots, \lambda_d(\Sigma)),$
where $\lambda_1(\Sigma)\geq \dots \lambda_d(\Sigma)$ are 
the eigenvalues of $\Sigma$ and $\varphi$ is a symmetric function of 
$d$ variables. A typical example of orthogonally invariant 
function is $g(\Sigma)={\rm tr}(\psi (\Sigma))$ for a function of real variable $\psi.$
Let $L_{\infty}^{O}({\mathcal C}_+({\mathbb H}))$ be the space of all
orthogonally invariant functions from  $L_{\infty}({\mathcal C}_+({\mathbb H})).$
Clearly, orthogonally invariant functions form an algebra. 
We will need several facts concerning the properties of operators ${\mathcal T}, {\mathcal B}$
as well as the lifting operator operator ${\mathcal D}$ on the space of orthogonally invariant functions.
In the case of orthogonally invariant polynomials, similar properties could be found in the 
literature on Wishart distribution (see, e.g., \cite{Faraut_Koranyi, Letac-Massam}).

\begin{proposition}
\label{orth}
If $g\in L_{\infty}^{O}({\mathcal C}_+({\mathbb H})),$
then ${\mathcal T}g\in L_{\infty}^{O}({\mathcal C}_+({\mathbb H}))$ and
${\mathcal B}g\in L_{\infty}^{O}({\mathcal C}_+({\mathbb H})).$ 
\end{proposition} 
 
\begin{proof} 
Indeed, 
the transformation $\Sigma\mapsto U\Sigma U^{-1}$ is a bijection of 
${\mathcal C}_+({\mathbb H}),$   
$$
{\mathcal T}g(U\Sigma U^{-1})= {\mathbb E}_{U\Sigma U^{-1}}g(\hat \Sigma)=
{\mathbb E}_{\Sigma}g(U\hat \Sigma U^{-1})= {\mathbb E}_{\Sigma}g(\hat \Sigma)=
{\mathcal T}g(\Sigma)
$$
and the function $Tg$ is uniformly bounded.

\end{proof}

An operator valued function $g: {\mathcal C}_+({\mathbb H})\mapsto {\mathcal B}_{sa}({\mathbb H})$ is called {\it orthogonally equivariant} iff for all orthogonal transformations 
$U$
$
g(U\Sigma U^{-1}) = Ug(\Sigma)U^{-1}, \Sigma \in {\mathcal C}_+({\mathbb H}). 
$

We say that $g:{\mathcal C}_+({\mathbb H})\mapsto {\mathcal B}_{sa}({\mathbb H})$ is differentiable (resp., continuously differentiable, $k$ times continuously differentiable, etc) in ${\mathcal C}_+({\mathbb H})$ iff there exists a uniformly bounded,  Lipschitz with respect to the operator norm and differentiable (resp., continuously differentiable, $k$ times continuously differentiable, etc) extension of $g$ to an open set $G,$ ${\mathcal C}_+({\mathbb H})\subset G\subset {\mathcal B}_{sa}({\mathbb H}).$ 
Note that $g$ could be further extended from $G$ to a uniformly bounded Lipschitz with respect to the operator norm function on ${\mathcal B}_{sa}({\mathbb H}),$ which will be still denoted by $g.$

\begin{proposition}
\label{equiv}
If $g:{\mathcal C}_+({\mathbb H})\mapsto {\mathbb R}$ is orthogonally invariant 
and continuously differentiable in ${\mathcal C}_+({\mathbb H})$ with derivative $Dg,$ then $Dg$ is orthogonally equivariant. 
\end{proposition}

\begin{proof}
First suppose that $\Sigma$ is positively definite. 
Then, given $H\in {\mathcal B}_{sa}({\mathbb H}),$ 
$\Sigma+tH$ is a covariance operator for all small enough $t.$
Thus, for all $H\in {\mathcal B}_{sa}({\mathbb H}),$
$$
\langle Dg(U\Sigma U^{-1}), H\rangle  
= \lim_{t\to 0} \frac{g(U\Sigma U^{-1}+tH)-g(U\Sigma U^{-1})}{t}
$$
$$
= 
\lim_{t\to 0} \frac{g(U(\Sigma +t U^{-1}HU)U^{-1})-g(U\Sigma U^{-1})}{t}
=\lim_{t\to 0} \frac{g(\Sigma +t U^{-1}HU))-g(\Sigma)}{t}
$$
$$
=
\langle Dg(\Sigma), U^{-1}HU\rangle = \langle UDg(\Sigma)U^{-1},H\rangle
$$
implying 
\begin{equation}
\label{orth_equiv_AAA}
Dg(U\Sigma U^{-1})=UDg(\Sigma)U^{-1}.
\end{equation}
It remains to observe that positively definite covariance 
operators are dense in ${\mathcal C}_+({\mathbb H})$ 
and to extend \eqref{orth_equiv_AAA} to ${\mathcal C}_+({\mathbb H})$ by continuity.

\end{proof}

We now define the following differential operator
$$
{\mathcal D}g (\Sigma):= \Sigma^{1/2} Dg(\Sigma)\Sigma^{1/2}
$$ 
acting on continuously differentiable functions in ${\mathcal C}_+({\mathbb H}).$
It will be called the {\it lifting operator}. 
We will show that operators ${\mathcal T}$ and ${\mathcal D}$
commute (and, as a consequence, ${\mathcal B}$ and ${\mathcal D}$
also commute).

\begin{proposition}
\label{commute}
Suppose $d\lessim n.$
For all functions $g\in L_{\infty}^O({\mathcal C}_+({\mathbb H}))$ that are 
continuously differentiable in ${\mathcal C}_+({\mathbb H})$
with a uniformly bounded derivative $Dg$
and for all $\Sigma\in {\mathcal C}_+({\mathbb H})$
$$
{\mathcal D}{\mathcal T}g(\Sigma)= {\mathcal T}{\mathcal D}g(\Sigma)\ {\rm and}\ 
 {\mathcal D}{\mathcal B}g(\Sigma)= {\mathcal B}{\mathcal D}g(\Sigma).
$$
\end{proposition}

\begin{proof}
Note that $\hat \Sigma \stackrel{d}{=} \Sigma^{1/2}W \Sigma^{1/2},$
where $W$ is the sample covariance based on i.i.d. standard normal 
random variables $Z_1,\dots, Z_n$ in ${\mathbb H}$ (which is a rescaled Wishart matrix).
Let $\Sigma^{1/2} W^{1/2}= RU$ be the polar decomposition of $\Sigma^{1/2} W^{1/2}$ with positively semidefinite $R$ and orthogonal $U.$ 
Then, we have 
$$
\hat \Sigma=\Sigma^{1/2}W\Sigma^{1/2} =
\Sigma^{1/2}W^{1/2}W^{1/2}\Sigma^{1/2}
=RU U^{-1}R =R^2 
$$
and 
$$
W^{1/2}\Sigma W^{1/2}=
W^{1/2}\Sigma^{1/2}
\Sigma^{1/2}W^{1/2}
=
U^{-1}R R U =U^{-1}R^2U 
= U^{-1} \Sigma^{1/2}W\Sigma^{1/2} U = U^{-1}\hat \Sigma U.
$$
Since $g$ is orthogonally invariant, we have 
\begin{equation}
\label{Tg_Tg}
{\mathcal T}g(\Sigma)={\mathbb E}_{\Sigma}g(\hat \Sigma)= 
{\mathbb E} g(\Sigma^{1/2}W\Sigma^{1/2}) = {\mathbb E}g(W^{1/2}\Sigma W^{1/2}),
\Sigma \in {\mathcal C}_+({\mathbb H}).
\end{equation}
Since we extended $g$ to a uniformly bounded function on ${\mathcal B}_{sa}({\mathbb H}),$
the right hand side of \eqref{Tg_Tg} is well defined for all $\Sigma\in {\mathcal B}_{sa}({\mathbb H}),$
and it will be used to extend ${\mathcal T}g(\Sigma)$ to ${\mathcal B}_{sa}({\mathbb H}).$
Moreover, since $g$ is Lipschitz with respect to the operator norm and, for $d\lesssim n,$ 
${\mathbb E}\|W\|\leq 1+ {\mathbb E}\|W-I\|\leq 1+C\sqrt{\frac{d}{n}}\lesssim 1$ (see \eqref{operator_hatSigma_exp_dimension}),
it is easy to check that ${\mathcal T}g(\Sigma)$ is Lipschitz with respect  to the operator norm on ${\mathcal B}_{sa}({\mathbb H}).$
 
Let $H\in {\mathcal B}_{sa}({\mathbb H})$ and $\Sigma_t:=\Sigma+tH, t> 0.$ Note that 
\begin{align}
\label{Tg_odin}
&
\nonumber
\frac{{\mathcal T}g(\Sigma_t)-{\mathcal T}g(\Sigma)}{t}
\\
&
\nonumber
=
\frac{{\mathbb E} g(W^{1/2}\Sigma_tW^{1/2})-
{\mathbb E} g(W^{1/2}\Sigma W^{1/2})}{t}
\\
&
\nonumber
= {\mathbb E}\frac{g(W^{1/2}\Sigma_tW^{1/2})-g(W^{1/2}\Sigma W^{1/2})}{t}I(\|W\|\leq 1/\sqrt{t})
\\
&
+
{\mathbb E}\frac{g(W^{1/2}\Sigma_tW^{1/2})-g(W^{1/2}\Sigma W^{1/2})}{t}I(\|W\|>1/\sqrt{t}).
\end{align}
Recall that $g$ is continuously differentiable in the open set $G\supset {\mathcal C}_+({\mathbb H}).$
Also, $W^{1/2}\Sigma W^{1/2}\in {\mathcal C}_+({\mathbb H})\subset G$ and 
$W^{1/2}\Sigma_t W^{1/2}\in G$ for all small enough $t>0.$ The last fact follows from the bound 
$
\|W^{1/2}(\Sigma_t-\Sigma)W^{1/2}\|\leq \|W\| t\|H\|\leq \sqrt{t}\|H\|
$
that holds for all $t\leq \frac{1}{\|W\|^2}$ (or $\|W\|\leq 1/\sqrt{t}$)
Therefore, we easily get that  
\begin{align*}
&
\lim_{t\to 0}\frac{g(W^{1/2}\Sigma_tW^{1/2})-
g(W^{1/2}\Sigma W^{1/2})}{t}I(\|W\|\leq 1/\sqrt{t})
\\
&
= \langle Dg(W^{1/2}\Sigma W^{1/2}), W^{1/2}HW^{1/2}\rangle
=\langle W^{1/2}Dg(W^{1/2}\Sigma W^{1/2})W^{1/2}, H\rangle.
\end{align*}
Also, since $g$ is Lipschitz with respect to the operator norm, 
\begin{align*}
&
\biggl|\frac{g(W^{1/2}\Sigma_tW^{1/2})-
g(W^{1/2}\Sigma W^{1/2})}{t}I(\|W\|\leq 1/\sqrt{t})\biggr| 
\\
&
\lesssim_g \frac{\|W^{1/2}(\Sigma_t-\Sigma)W^{1/2}\|}{t}
\leq \frac{\|W\|\|\Sigma_t-\Sigma\|}{t}\leq \|W\|\|H\|.
\end{align*}
Since ${\mathbb E}\|W\|\lesssim 1,$ we can use Lebesgue 
dominated convergence theorem to prove that 
\begin{align}
\label{Tg_dva}
&
\nonumber
\lim_{t\to 0}{\mathbb E}\frac{g(W^{1/2}\Sigma_tW^{1/2})-g(W^{1/2}\Sigma W^{1/2})}{t}I(\|W\|\leq 1/\sqrt{t})
\\
&
={\mathbb E}\langle W^{1/2}Dg(W^{1/2}\Sigma W^{1/2})W^{1/2}, H\rangle
=\langle {\mathbb E} W^{1/2}Dg(W^{1/2}\Sigma W^{1/2})W^{1/2}, H\rangle.
\end{align}
On the other hand, since $g$ is uniformly bounded, we can use bound \eqref{operator_hatSigma_exp_dimension}
to prove that for some constant $C>0$ and for all $t\leq 1/C^2$
\begin{align}
\label{Tg_tri}
&
\nonumber
{\mathbb E}\biggl|\frac{g(W^{1/2}\Sigma_tW^{1/2})-g(W^{1/2}\Sigma W^{1/2})}{t}I(\|W\|>1/\sqrt{t})\biggr|
\\
&
\lesssim_g \frac{1}{t} {\mathbb P}\Bigl\{\|W\|\geq \frac{1}{\sqrt{t}}\Bigr\} \leq \frac{1}{t}\exp\biggl\{-\frac{n}{C\sqrt{t}}\biggr\}\to 0
\ {\rm as}\ t\to 0.
\end{align}
It follows from \eqref{Tg_odin}, \eqref{Tg_dva} and \eqref{Tg_tri} that 
$$
\langle D{\mathcal T}g(\Sigma), H \rangle 
= 
\langle{\mathbb E}W^{1/2}Dg(W^{1/2}\Sigma W^{1/2})W^{1/2}, H\rangle.
$$
It is also easy to check that ${\mathbb E}W^{1/2}Dg(W^{1/2}\Sigma W^{1/2})W^{1/2}$
is a continuous function in $G$ implying that ${\mathcal T}g$ is 
continuously differentiable in $G$ with Fr\'echet derivative
$$
D{\mathcal T}g(\Sigma)= {\mathbb E} W^{1/2}Dg(W^{1/2}\Sigma W^{1/2})W^{1/2}.
$$
Since 
$
W^{1/2} \Sigma W^{1/2} =U^{-1} \hat \Sigma U
$
and $Dg$ is an orthogonally 
equivariant function (see Proposition \ref{equiv}), we get 
$
Dg(W^{1/2}\Sigma W^{1/2})= U^{-1}Dg(\hat \Sigma) U. 
$
Therefore,
\begin{align*}
&
{\mathcal D} {\mathcal T}g(\Sigma)
\\
&
= \Sigma^{1/2} D{\mathcal T}g(\Sigma)\Sigma^{1/2} 
= \Sigma^{1/2}{\mathbb E} (W^{1/2}Dg(W^{1/2}\Sigma W^{1/2})W^{1/2})\Sigma^{1/2}
\\
&
=
{\mathbb E} (\Sigma^{1/2}W^{1/2}Dg(W^{1/2}\Sigma W^{1/2})W^{1/2}\Sigma^{1/2})
={\mathbb E} (\Sigma^{1/2} W^{1/2}U^{-1}Dg(\hat \Sigma) UW^{1/2}\Sigma^{1/2})
\\
&
= {\mathbb E} (R U U^{-1}Dg(\hat \Sigma) UU^{-1} R)=
{\mathbb E} (RDg(\hat \Sigma)R)
= 
{\mathbb E}_{\Sigma} (\hat \Sigma^{1/2}Dg(\hat \Sigma)\hat \Sigma^{1/2})
={\mathbb E}_{\Sigma} {\mathcal D}g(\hat \Sigma)
\\
&
= {\mathcal T} {\mathcal D}g(\Sigma).
\end{align*}
Similar relationship for operators ${\mathcal B}$ and ${\mathcal D}$
easily follows.

\end{proof}

We will now derive useful representations of operators ${\mathcal T}^k$ and 
${\mathcal B}^k$ and prove that they also commute with the differential operator ${\mathcal D}.$

\begin{proposition}
\label{T^kB^k}
Suppose $d\lesssim n.$ Let $W_1,\dots, W_k, \dots$ be i.i.d. copies of $W.$\footnote{Recall that 
$W$ is the sample covariance based on i.i.d. standard normal 
random variables $Z_1,\dots, Z_n$ in ${\mathbb H}.$}
Then, for all $g\in L_{\infty}^{O}({\mathcal C}_+({\mathbb H}))$ and for all $k\geq 1,$
\begin{equation}
\label{T^k}
{\mathcal T}^k g(\Sigma)= {\mathbb E} g(W_k^{1/2}\dots W_1^{1/2}\Sigma W_1^{1/2}\dots 
W_k^{1/2})
\end{equation}
and
\begin{equation}
\label{B^k}
{\mathcal B}^{k} g(\Sigma) = {\mathbb E}\sum_{I\subset \{1,\dots, k\}}
(-1)^{k-|I|} g(A_I^{\ast} \Sigma A_I),
\end{equation}
where
$
A_I:=\prod_{i\in I} W_i^{1/2}.
$
Suppose, in addition, that  
$g$ is continuously differentiable in ${\mathcal C}_+({\mathbb H})$
with a uniformly bounded derivative $Dg.$
Then
\begin{equation}
\label{B^k_diff}
D{\mathcal B}^{k} g(\Sigma) = {\mathbb E}\sum_{I\subset \{1,\dots, k\}}
(-1)^{k-|I|} A_{I}Dg(A_I^{\ast} \Sigma A_I)A_I^{\ast},
\end{equation}
and,
for all $\Sigma \in {\mathcal C}_+({\mathbb H})$
\begin{equation}
\label{commutative}
{\mathcal D}{\mathcal T}^kg(\Sigma)= {\mathcal T}^k{\mathcal D}g(\Sigma)\ {\rm and}\ 
 {\mathcal D}{\mathcal B}^kg(\Sigma)= {\mathcal B}^k{\mathcal D}g(\Sigma).
\end{equation}
Finally, 
\begin{align}
\label{B_kDB_k}
&
\nonumber
{\mathcal B}^k{\mathcal D}g(\Sigma)= 
{\mathcal D} {\mathcal B}^k g(\Sigma)
\\
&
={\mathbb E}\Bigl(\sum_{I\subset \{1,\dots, k\}}
(-1)^{k-|I|} \Sigma^{1/2}A_{I}Dg(A_I^{\ast} \Sigma A_I)A_I^{\ast}\Sigma^{1/2}\Bigr).
\end{align}
\end{proposition}

\begin{proof}
Since $\hat \Sigma \stackrel{d}{=} \Sigma^{1/2}W \Sigma^{1/2},$
$W^{1/2}\Sigma W^{1/2}=U^{-1}\Sigma^{1/2}W\Sigma^{1/2}U,$
where $U$ is an orthogonal operator, and $g$ is orthogonally invariant,
we have 
\begin{equation}
\label{k=1}
{\mathcal T}g(\Sigma)= {\mathbb E}_{\Sigma}g(\hat \Sigma)=
{\mathbb E}g(W^{1/2}\Sigma W^{1/2})
\end{equation}
(which has been already used in the proof of Proposition \ref{commute}). 

By Proposition \ref{orth}, orthogonal invariance of $g$ implies 
the same property of ${\mathcal T}g$ and, by induction, of ${\mathcal T}^k g$ for all $k\geq 1.$
Then, also by induction, it follows from (\ref{k=1}) that 
$$
{\mathcal T}^k g(\Sigma)= {\mathbb E} g(W_k^{1/2}\dots W_1^{1/2}\Sigma W_1^{1/2}\dots 
W_k^{1/2}).
$$
If $I\subset \{1,\dots, k\}$
with $|I|={\rm card}(I)=j$ and
$
A_I=\prod_{i\in I} W_i^{1/2},
$
it clearly implies that  
$$
{\mathcal T}^{j}g(\Sigma)= {\mathbb E}g(A_I^{\ast} \Sigma A_I).
$$
In view of (\ref{Bkrepr_1}), we easily get that (\ref{B^k}) holds. 
If $g$ is continuously differentiable in ${\mathcal C}_+({\mathbb H})$
with a uniformly bounded derivative $Dg,$ it follows from 
(\ref{B^k}) that ${\mathcal B}^k g(\Sigma)$ is continuously differentiable 
in ${\mathcal C}_+({\mathbb H})$ with Fr\'echet derivative given by (\ref{B^k_diff}). 
To prove this, it is enough to justify differentiation 
under the expectation sign which is done exactly as in the proof 
of Proposition \ref{commute}. Finally, it follows from (\ref{B^k_diff})
that the derivatives $D{\mathcal B}^k g, k\geq 1$ are uniformly 
bounded in ${\mathcal C}_+({\mathbb H}).$ Similarly, 
as a consequence of (\ref{T^k}) and the properties of $g,$ 
${\mathcal T}^k g(\Sigma)$ is continuously differentiable 
in ${\mathcal C}_+({\mathbb H})$ with uniformly bounded derivative  
$D{\mathcal T}^k g$ for all $k\geq 1.$ 
Therefore, ({\ref{commutative}})
follows from Proposition \ref{commute} by induction.
Formula (\ref{B_kDB_k}) follows from (\ref{commutative}) and (\ref{B^k_diff}).

\end{proof}

Define the following functions providing the linear interpolation between the identity 
operator $I$ and operators $W_1^{1/2},\dots, W_k^{1/2}:$
$$
V_j(t_j):= I + t_j (W_j^{1/2}-I), t_j\in [0,1],  1\leq j\leq k.
$$
Clearly, for all $j=1,\dots, k, t_j\in [0,1],$ $V_j(t_j)\in {\mathcal C}_+({\mathbb H}).$ 
Let 
$$
R= R(t_1,\dots, t_k) = V_1(t_1) \dots V_k(t_k)
\ \
{\rm and}
\ \ 
L= L(t_1,\dots, t_k)= V_k(t_k)\dots V_1(t_1) = R^{\ast}.
$$
Define 
$$
S=S(t_1,\dots, t_k)= L(t_1,\dots, t_k) \Sigma R(t_1,\dots, t_k), (t_1,\dots, t_k)\in [0,1]^k. 
$$
Finally, let 
$$
\varphi (t_1,\dots, t_k):= \Sigma^{1/2}R(t_1,\dots, t_k)
Dg(S(t_1,\dots, t_k))L(t_1,\dots, t_k) \Sigma^{1/2}, 
(t_1,\dots, t_k)\in [0,1]^k.
$$

The following representation will play a crucial role in our further analysis. 

\begin{proposition}
Suppose $g\in L_{\infty}^O({\mathcal C}_+({\mathbb H}))$
is $k+1$ times continuously differentiable function with uniformly bounded 
derivatives $D^{j}g, j=1,\dots, k+1.$ Then the function $\varphi$ is $k$ times continuously differentiable in $[0,1]^k$ and
\begin{equation}
\label{B_k_repres}
{\mathcal B}^k{\mathcal D}g(\Sigma)=  {\mathbb E}\int_0^1 \dots \int_0^1 
\frac{\partial^k \varphi (t_1,\dots, t_k)}{\partial t_1\dots \partial t_k}dt_1\dots dt_k, \Sigma\in {\mathcal C}_+({\mathbb H}).
\end{equation}
\end{proposition}

\begin{proof}
Given a function $\phi: [0,1]^k \mapsto {\mathbb R},$ define for 
$1\leq i\leq k$ finite difference 
operators 
$$
{\frak D}_i \phi (t_1,\dots, t_k):= \phi (t_1,\dots, t_{i-1},1, t_{i+1},\dots , t_k)-
\phi (t_1,\dots, t_{i-1}, 0, t_{i+1}, \dots, t_k), 
$$
(with obvious modifications for $i=1,k$). Then 
${\frak D}_1 \dots {\frak D}_k \phi$ does not depend on $t_1,\dots, t_k$
and is given by the formula
\begin{equation}
\label{finite_diff_1}
{\frak D}_1 \dots {\frak D}_k \phi = \sum_{(t_1,\dots, t_k)\in \{0,1\}^k}
(-1)^{k-(t_1+\dots +t_k)} \phi (t_1,\dots, t_k).
\end{equation}
It is well known and easy to check 
that if $\phi$ is $k$ times continuously  differentiable in $[0,1]^k,$
then 
\begin{equation}
\label{finite_diff_2}
{\frak D}_1 \dots {\frak D}_k \phi = \int_0^1 \dots \int_0^1 
\frac{\partial^k \phi (t_1,\dots, t_k)}{\partial t_1\dots \partial t_k}dt_1\dots dt_k.
\end{equation}
Similar definitions and formula (\ref{finite_diff_2}) also hold for vector- and operator-valued functions $\phi.$

It immediately follows from (\ref{B_kDB_k}) and (\ref{finite_diff_1}) that 
\begin{equation}
\label{Bkrepr_10}
{\mathcal B}^k {\mathcal D}g(\Sigma)= {\mathbb E} {\frak D}_1\dots {\frak D}_k \varphi.
\end{equation}
Since $Dg$ is $k$ times continuously differentiable and the functions 
$S(t_1,\dots, t_k)$, $R(t_1,\dots, t_k)$ are polynomials with respect 
to $t_1,\dots, t_k,$ the function $\varphi$ is $k$ times continuously differentiable 
in $[0,1]^k.$ Representation (\ref{B_k_repres}) follows from (\ref{Bkrepr_10})
and (\ref{finite_diff_2}).

\end{proof}

\section{Bounds on iterated bias operator}
\label{Sec:bias-iter}

Our goal in this section is to prove the following bound on iterated 
bias operator ${\mathcal B}^k{\mathcal D}g(\Sigma).$

\begin{theorem}
\label{bias-iter}
Suppose $g\in L_{\infty}^O({\mathcal C}_+({\mathbb H}))$
is $k+1$ times continuously differentiable function with uniformly bounded 
derivatives $D^{j}g, j=1,\dots, k+1.$ Suppose also that $d\leq n$ and $k\leq n.$ 
Then the following bound holds for some constant $C>0$ and 
for all $\Sigma\in {\mathcal C}_+({\mathbb H}):$
\begin{equation}
\label{Bk_konec_01}
\|{\mathcal B}^{k}{\mathcal D}g(\Sigma)\|
\leq 
C^{k^2}
\max_{1\leq j\leq k+1} 
\|D^{j}g\|_{L_{\infty}}
(\|\Sigma\|^{k+1}\vee \|\Sigma\|)
\biggl(\frac{d}{n}\bigvee \frac{k}{n}\biggr)^{k/2}.
\end{equation}
\end{theorem}

It follows from commutativity relationships (\ref{commutative}) that 
$$
{\mathcal D}g_k(\Sigma)= ({\mathcal D}g)_k(\Sigma), \Sigma \in 
{\mathcal C}_+({\mathbb H}),
$$
where $g_k$ is defined by formula (\ref{def_g_k}) and 
$$
({\mathcal D}g)_k (\Sigma):=\sum_{j=0}^{k}(-1)^j {\mathcal B}^j {\mathcal D}g(\Sigma), \Sigma\in {\mathcal C}_+({\mathbb H}).
$$
Clearly, we have (see Proposition \ref{prop_g_k}) that
\begin{equation}
\label{bias_Dgk}
{\mathbb E}_{\Sigma}{\mathcal D}g_k(\hat \Sigma)- {\mathcal D}g(\Sigma)=(-1)^k {\mathcal B}^{k+1}{\mathcal D}g(\Sigma).
\end{equation}  

Bound (\ref{Bk_konec_01}) is needed, in particular, to control the bias of estimator 
${\mathcal D} g_k(\hat \Sigma)$ of ${\mathcal D} g(\Sigma).$
Namely, we have the following corollary.

\begin{corollary}
\label{Bias_bound}
Suppose that $g\in L_{\infty}^O({\mathcal C}_+({\mathbb H}))$
is $k+2$ times continuously differentiable function with uniformly bounded 
derivatives $D^{j}g, j=1,\dots, k+2$ and also that $d\leq n, k+1\leq n.$ 
Then 
\begin{equation}
\label{Bk_konec_01'''}
\|{\mathbb E}_{\Sigma}{\mathcal D}g_k(\hat \Sigma)-{\mathcal D}g(\Sigma)\|
\leq 
C^{(k+1)^2}
\max_{1\leq j\leq k+2} 
\|D^{j}g\|_{L_{\infty}}
(\|\Sigma\|^{k+2}\vee \|\Sigma\|)
\biggl(\frac{d}{n}\bigvee \frac{k+1}{n}\biggr)^{(k+1)/2}.
\end{equation}
If, in addition, $k+1\leq d\leq n$ and, for some 
$\delta>0,$ 
\begin{equation}
\label{cond_k}
k \geq \frac{\log d}{\log(n/d)}+\delta \biggl(1+ \frac{\log d}{\log (n/d)}\biggr).
\end{equation}
Then 
\begin{align}
\label{nakonecto}
&
\|{\mathbb E}_{\Sigma}{\mathcal D}g_k(\hat \Sigma)-{\mathcal D}g(\Sigma)\|
\leq 
C^{(k+1)^2}\max_{1\leq j\leq k+2}\|D^{j}g\|_{L_{\infty}}
(\|\Sigma\|^{k+2}\vee \|\Sigma\|) n^{-\frac{1+\delta}{2}}.
\end{align}
\end{corollary}

The proof of this corollary immediately follows from formula (\ref{bias_Dgk}) and bound (\ref{Bk_konec_01}).
If $d=n^{\alpha}$ for some $\alpha\in (0,1),$ condition (\ref{cond_k}) becomes 
$k\geq \frac{\alpha+\delta}{1-\alpha}.$ Thus, if 
$$
k(\alpha,\delta):= \min \biggl\{k\geq \frac{\alpha+\delta}{1-\alpha}\biggr\},
$$
then bound (\ref{nakonecto}) holds with $k=k(\alpha,\delta).$ 

\begin{remark}
In Section \ref{Sec:smooth}, we will obtain a sharper bound on the bias of 
estimator ${\mathcal D}g_k(\hat \Sigma)$ (under stronger smoothness assumptions, see Corollary \ref{bias_better}).
\end{remark}

The first step towards the proof of Theorem \ref{bias-iter} is to compute the partial derivative 
$\frac{\partial^k \varphi}{\partial t_1\dots \partial t_k}$ 
of function $\varphi$ which would allow us to use representation (\ref{B_k_repres}). To this end, we first derive formulas for partial 
derivatives of operator-valued function $h(S(t_1,\dots, t_k)),$ where $h=Dg.$ 
To simplify the notations, given $T=\{t_{i_1}, \dots, t_{i_m}\}\subset \{t_1,\dots, t_k\},$
we will write $\partial_T S$ instead of 
$\frac{\partial^mS(t_1,\dots, t_k)}{\partial t_{i_1}\dots \partial t_{i_m}}$
(similarly, we use the notation $\partial_T h(S)$ for a partial derivative 
of a function $h(S)$). 

Let ${\mathcal D}_{j,T}$ be the set of all partitions $(\Delta_1,\dots, \Delta_j)$
of $T\subset \{t_1,\dots, t_k\}$ with non-empty sets $\Delta_i, i=1,\dots, j$
(partitions with different order of $\Delta_1,\dots, \Delta_j$ being identical). 
For $\Delta = (\Delta_1,\dots, \Delta_j)\in {\mathcal D}_{j,T},$ set 
$
\partial_{\Delta} S= (\partial_{\Delta_1}S, \dots, \partial_{\Delta_j}S).
$
Denote ${\mathcal D}_T:= \bigcup_{j=1}^{|T|} {\mathcal D}_{j,T}.$
For $\Delta=(\Delta_1,\dots, \Delta_j)\in {\mathcal D}_T,$
set $j_{\Delta}:=j.$

\begin{lemma}
\label{chain}
Suppose, for some $m\leq k,$ $h=Dg\in L_{\infty}({\mathcal C}_+({\mathbb H}); {\mathcal B}_{sa}({\mathbb H}))$ is $m$ times continuously differentiable with derivatives $D^j h, j\leq m.$\footnote{Recall that $D^j h$ is an operator valued symmetric $j$-linear form on the 
space ${\mathcal B}_{sa}({\mathbb H}).$}
Then 
the function $[0,1]^k\ni (t_1,\dots, t_k)\mapsto h(S(t_1,\dots, t_k))$ is $m$ times continuously differentiable and for any $T\subset \{t_1,\dots, t_k\}$
with $|T|=m$ 
\begin{equation}
\label{deriv}
\partial_T h(S)
= 
\sum_{\Delta \in {\mathcal D}_T}D^{j_{\Delta}}h(S)(\partial_{\Delta}S)
=\sum_{j=1}^m 
\sum_{\Delta\in {\mathcal D}_{j,T}}D^{j}h(S)(\partial_{\Delta}S).
\end{equation}
\end{lemma}

\begin{proof}
Since $[0,1]^k\ni (t_1,\dots, t_k)\mapsto S(t_1,\dots, t_k)$ is an operator 
valued polynomial and $h$ is $m$ times continuously differentiable, 
the function $[0,1]^k\ni (t_1,\dots, t_k)\mapsto h(S(t_1,\dots, t_k))$ is also $m$ times continuously differentiable.
We will now prove formula (\ref{deriv}) by induction with respect to $m.$ For $m=1,$ it reduces 
to 
$$
\partial_{\{t_i\}}h(S)= \frac{\partial h(S)}{\partial t_i}=
Dh(S)\Bigl(\frac{\partial S}{\partial t_i}\Bigr),
$$
which is true by the chain rule. Assume that (\ref{deriv}) holds for some $m<k$ 
and for any $T\subset \{t_1,\dots, t_k\},$ $|T|=m.$
Let $T'= T\cup \{t_l\}$ for some $t_l\not\in T.$  
Then
\begin{equation}
\label{deriv_a}
\partial_{T'} h(S)=\partial_{\{t_{l}\}}
\partial_T h(S)= \sum_{j=1}^m 
\sum_{\Delta\in {\mathcal D}_{j,T}}
\partial_{\{t_{l}\}} D^{j}h(S)(\partial_{\Delta}S).
\end{equation}
Given $\Delta=(\Delta_1,\dots, \Delta_j)\in {\mathcal D}_{j,T},$ define 
partitions $\Delta^{(i)}\in {\mathcal D}_{j,T'}, i=1,\dots, j$ as follows:
$$
\Delta^{(1)}:= (\Delta_1\cup \{t_{l}\}, \Delta_2, \dots, \Delta_j),
\Delta^{(2)}:= (\Delta_1, \Delta_2\cup \{t_{l}\}, \dots, \Delta_j), \dots , 
$$
$$
\Delta^{(j)} := (\Delta_1, \dots,  \Delta_{j-1}, \Delta_j \cup \{t_{l}\}).
$$
Also define a partition $\tilde \Delta\in {\mathcal D}_{j+1,T'}$
as follows:
$
\tilde \Delta:=  (\Delta_1,\dots, \Delta_j, \{t_{l}\}).
$
It is easy to see that any partition $\Delta'\in {\mathcal D}_{T'}$
is the image of a unique partition $\Delta\in {\mathcal D}_T$ under
one of the transformations $\Delta \mapsto \Delta^{(i)}, i=1,\dots, j_{\Delta}$
and $\Delta \mapsto \tilde \Delta.$ 
This implies that 
$$
{\mathcal D}_{T'}= \bigcup_{\Delta\in {\mathcal D}_T} 
\{\Delta^{(1)},\dots, \Delta^{(j_{\Delta})}, \tilde \Delta\}.
$$
It easily follows from the chain rule and the product rule that 
$$
\partial_{\{t_{l}\}} D^{j}h(S)(\partial_{\Delta}S)
= \sum_{i=1}^j D^{j}h(S)(\partial_{\Delta^{(i)}}S)
+ D^{j+1}h(S)(\partial_{\tilde \Delta}S).
$$ 
Substituting this in (\ref{deriv_a}) easily yields 
$$
\partial_{T'} h(S)= \sum_{j=1}^{m+1} 
\sum_{\Delta\in {\mathcal D}_{j,T'}}D^{j}h(S)(\partial_{\Delta}S).
$$

\end{proof}

Next we derive upper bounds on $\|\partial_{T} S\|,$ $\|\partial_T R\|$ and 
$\|\partial_T L\|$ 
for $T\subset \{t_1,\dots, t_k\}.$
Denote 
$
\delta_i:= \|W_i-I\|, i=1,\dots, k.
$

\begin{lemma}
\label{bd_S}
For all 
$T\subset \{t_1,\dots, t_k\},$ 
\begin{equation}
\label{T_R}
\|\partial_T R\|\leq \prod_{t_i\in T}\frac{\delta_i}{1+\delta_i} \prod_{i=1}^k (1+\delta_i),
\end{equation}
\begin{equation}
\label{T_L}
\|\partial_T L\|\leq \prod_{t_i\in T}\frac{\delta_i}{1+\delta_i} \prod_{i=1}^k (1+\delta_i)\end{equation}
and 
\begin{equation}
\label{T_S}
\|\partial_{T} S\|\leq 2^{k} \|\Sigma\| 
\prod_{t_i\in T}\frac{\delta_i}{1+\delta_i} \prod_{i=1}^k (1+\delta_i)^2.
\end{equation}
\end{lemma}

\begin{remark}
The bounds of the lemma hold for $T=\emptyset$ with an obvious convention
that in this case $\prod_{t_i\in T} a_i=1.$
\end{remark}

\begin{proof}
Observe that
$
\frac{\partial }{\partial t_i} V_i(t_i)= W_i^{1/2}-I.
$
Let $B_i^{0}:= V_i(t_i)$ and $B_i^{1} := W_i^{1/2}-I.$
For $R=V_1(t_1)\dots V_k(t_k),$ we have 
$
\partial_T R = \prod_{i=1}^k B_i^{I_T(t_i)}
$
and 
$$
\|\partial_T R\| \leq \prod_{t_i\in T} \|W_i^{1/2}-I\|
\prod_{t_i\not\in T} \|V_i(t_i)\|. 
$$
Note that, due to an elementary inequality $|\sqrt{x}-1|\leq |x-1|, x\geq 0,$ we have 
$
\|W_i^{1/2}-I\|\leq \|W_i-I\|=\delta_i
$
and $\|V_i(t_i)\|\leq 1 + \|W_i^{1/2}-I\|\leq 1+\|W_i-I\|=1+\delta_i.$
Therefore, 
$$
\|\partial_T R\| \leq \prod_{t_i\in T} \delta_i
\prod_{t_i\not\in T} (1+\delta_i)= 
\prod_{t_i\in T} \frac{\delta_i}{1+\delta_i} 
\prod_{i=1}^k (1+\delta_i),
$$
which proves (\ref{T_R}). Similarly, we have (\ref{T_L}).

Note that, by the product rule,
$$
\partial_T S = \partial_T (L\Sigma R) = \sum_{T'\subset T} (\partial_{T'} L)\Sigma (\partial_{T\setminus T'}R).
$$
Therefore,
\begin{align*}
&
\|\partial_T S\| \leq \|\Sigma\|\sum_{T'\subset T} \|\partial_{T'} L\|
\|\partial_{T\setminus T'}R\|
\\
&
\leq \|\Sigma\|\sum_{T'\subset T} \prod_{t_i\in T'} \frac{\delta_i}{1+\delta_i} 
\prod_{t_i\in T\setminus T'} \frac{\delta_i}{1+\delta_i} 
\prod_{i=1}^k (1+\delta_i)^2
=2^k \|\Sigma\| 
\prod_{t_i\in T} \frac{\delta_i}{1+\delta_i}\prod_{i=1}^k (1+\delta_i)^2,
\end{align*}
proving (\ref{T_S}).

\end{proof}

\begin{lemma}
\label{chain+A}
Suppose that, for some $0\leq m\leq k,$ $h=Dg\in L_{\infty}({\mathcal C}_+({\mathbb H}); {\mathcal B}_{sa}({\mathbb H}))$ is $m$ times differentiable with uniformly bounded continuous derivatives $D^j h, j=1,\dots, m.$
Then for all $T\subset \{t_1,\dots, t_k\}$ with $|T|=m$
\begin{align}
\label{partial_bd}
 &
\|\partial_T h(S)\|
\leq  2^{m(k+m+1)} 
\max_{0\leq j\leq m}\|D^{j}h\|_{L_{\infty}}
(\|\Sigma\|^m \vee 1)
\prod_{i=1}^k (1+\delta_i)^{2m}\prod_{t_i\in T}\frac{\delta_i}{1+\delta_i}.
\end{align}
\end{lemma}

\begin{proof}
Assume that $m\geq 1$ (for $m=0,$ the bound of the lemma is trivial). 
Let $\Delta = (\Delta_1,\dots, \Delta_j)\in {\mathcal D}_{j,T}, j\leq m.$ 
Note that
\begin{align*}
&
\|D^{j} h(S)(\partial_{\Delta_1} S,\dots, \partial_{\Delta_j}S)\|\leq 
\|D^{j} h(S)\| \|\partial_{\Delta_1} S\|\dots \|\partial_{\Delta_j} S\|
\\
&
\leq \|D^{j} h(S)\| 
2^{kj} \|\Sigma\|^j 
\prod_{l=1}^j\prod_{t_i\in \Delta_l} \frac{\delta_i}{1+\delta_i}
\prod_{i=1}^k (1+\delta_i)^{2j}
\\
&
=
\|D^{j} h(S)\| 
2^{kj} \|\Sigma\|^j 
\prod_{t_i\in T} \frac{\delta_i}{1+\delta_i}
\prod_{i=1}^k (1+\delta_i)^{2j}.
\end{align*}
Using Lemma \ref{chain}, we get 
\begin{align*}
&
\|\partial_T h(S)\|
\leq 
\sum_{j=1}^m 
\sum_{\Delta\in {\mathcal D}_{j,T}}\|D^{j}h(S) (\partial_{\Delta}S)\|
\\
&
\leq 
\sum_{j=1}^m {\rm card}({\mathcal D}_{j,T})
\|D^{j} h(S)\| 
2^{kj} \|\Sigma\|^j \prod_{i=1}^k (1+\delta_i)^{2j}
\prod_{t_i\in T} \frac{\delta_i}{1+\delta_i}.
\end{align*}
Note that the number of all functions on $T$ with values in $\{1,\dots,j\}$
is equal to $j^m$ and, clearly, ${\rm card}({\mathcal D}_{j,T})\leq j^m.$ 
Therefore, 
 \begin{align}
 \label{partial_bd'}
&
\nonumber
\|\partial_T h(S)\|
\leq 
\sum_{j=1}^m j^m
\|D^{j} h(S)\| 
2^{kj} \|\Sigma\|^j \prod_{i=1}^k (1+\delta_i)^{2j}
\prod_{t_i\in T} \frac{\delta_i}{1+\delta_i}
\\
&
\nonumber
\leq 
m^{m+1} 2^{km} 
\max_{1\leq j\leq m}
\|D^{j}h\|_{L_{\infty}}
(\|\Sigma\|\vee \|\Sigma\|^m)
\prod_{i=1}^k (1+\delta_i)^{2m}
\prod_{t_i\in T} \frac{\delta_i}{1+\delta_i}
\\
&
\leq 
2^{m(k+m+1)} 
\max_{1\leq j\leq m}
\|D^{j}h\|_{L_{\infty}}
(\|\Sigma\|\vee \|\Sigma\|^m)
\prod_{i=1}^k (1+\delta_i)^{2m}
\prod_{t_i\in T} \frac{\delta_i}{1+\delta_i},
\end{align}
which easily implies bound (\ref{partial_bd}).

\end{proof}

Next we bound partial derivatives of the function $\Sigma^{1/2}Lh(S)R\Sigma^{1/2}$
(with $S=S(t_1,\dots, t_k),$ $L=L(t_1,\dots, t_k),$ $R=R(t_1,\dots, t_k)$ and $h=Dg$). 

\begin{lemma}
\label{bd_partial}
Assume that $d\leq n$ and $k\leq n.$
Suppose $h=Dg\in L_{\infty}({\mathcal C}_+({\mathbb H}); {\mathcal B}_{sa}({\mathbb H}))$ 
is $k$ times differentiable with uniformly bounded continuous derivatives $D^j h, j=1,\dots, k.$
Then 
\begin{align}
\label{bd_LhR}
&
\|\partial_{\{t_1,\dots, t_k\}}\Sigma^{1/2}Rh(S)L\Sigma^{1/2}\|
\leq 3^k 2^{k(2k+1)}\max_{0\leq j\leq k} 
\|D^{j}h\|_{L_{\infty}}
(\|\Sigma\|^{k+1}\vee \|\Sigma\|)\prod_{i=1}^k (1+\delta_i)^{2k+1} \delta_i.  
\end{align}
\end{lemma}

\begin{proof} Note that 
\begin{equation}
\label{represent_partial}
\partial_{\{t_1,\dots, t_k\}}\Sigma^{1/2}R h(S) L\Sigma^{1/2}
=  \sum_{T_1,T_2,T_3} \Bigl(\Sigma^{1/2}(\partial_{T_1}R)(\partial_{T_2}h(S)) (\partial_{T_3}L)\Sigma^{1/2}\Bigr),
\end{equation}
where the sum is over all the partitions of the set $\{t_1,\dots, t_k\}$
into disjoint subsets $T_1,T_2,T_3.$ The number of such partitions is 
equal to $3^k.$
We have 
\begin{equation}
\label{T_1T_2T_3}
\|\Sigma^{1/2}(\partial_{T_1}R)(\partial_{T_2}h(S)) 
(\partial_{T_3}L)\Sigma^{1/2}\|\leq 
\|\Sigma\| \|\partial_{T_1}L\|\|\partial_{T_2}h(S)\| \|\partial_{T_3}R\|.
\end{equation}
Assume $|T_1|=m_1, |T_2|=m_2, |T_3|=m_3.$
It follows from Lemma \ref{chain+A} that 
$$
\|\partial_{T_2}h(S)\|\leq 
2^{m_2(k+m_2+1)} 
\max_{0\leq j\leq m_2}\|D^{j}h\|_{L_{\infty}}
(\|\Sigma\|^{m_2} \vee 1)
\prod_{i=1}^k (1+\delta_i)^{2m_2}\prod_{t_i\in T_2}\frac{\delta_i}{1+\delta_i}.
$$
On the other hand, by (\ref{T_R}) and (\ref{T_L}), we have 
$$
\|\partial_{T_3}R\|\leq 
\prod_{t_i\in T_3}\frac{\delta_i}{1+\delta_i} \prod_{i=1}^k (1+\delta_i)
\ \ 
{\rm and} 
\ \
\|\partial_{T_1}L\|\leq 
\prod_{t_i\in T_1}\frac{\delta_i}{1+\delta_i} \prod_{i=1}^k (1+\delta_i).
$$
It follows from these bounds and (\ref{T_1T_2T_3}) that 
\begin{align*}
&
\|\Sigma^{1/2}(\partial_{T_1}R)(\partial_{T_2}h(S))(\partial_{T_3}L)\Sigma^{1/2}\|
\\
&
\leq 
\|\Sigma\|
2^{m_2(k+m_2+1)} 
\max_{0\leq j\leq m_2}\|D^{j}h\|_{L_{\infty}}
(\|\Sigma\|^{m_2} \vee 1)
\prod_{i=1}^k (1+\delta_i)^{2m_2+2}
\prod_{t_i\in T_1\cup T_2\cup T_3}\frac{\delta_i}{1+\delta_i}
\\
&
\leq 
2^{k(2k+1)} 
\max_{0\leq j\leq k}\|D^{j}h\|_{L_{\infty}}
(\|\Sigma\|^{k+1} \vee \|\Sigma\|)
\prod_{i=1}^k (1+\delta_i)^{2k+2}
\prod_{i=1}^k\frac{\delta_i}{1+\delta_i}
\\
&
=
2^{k(2k+1)} 
\max_{0\leq j\leq k}\|D^{j}h\|_{L_{\infty}}
(\|\Sigma\|^{k+1} \vee \|\Sigma\|)
\prod_{i=1}^k (1+\delta_i)^{2k+1} \delta_i.
\end{align*}
Since the number of terms in the sum in the right hand side 
of (\ref{represent_partial}) is equal to $3^k,$ we easily get that
bound (\ref{bd_LhR}) holds.

\end{proof}

We are now in a position to prove Theorem \ref{bias-iter}.

\begin{proof}
We use representation (\ref{B_k_repres}) 
to get 
\begin{equation}
\label{Bk_pochti_konec}
\|{\mathcal B}^k {\mathcal D}g(\Sigma)\| 
\leq \int_{0}^1 \dots \int_{0}^{1}{\mathbb E}\|\partial_{\{t_1,\dots, t_k\}}
\Sigma^{1/2}Rh(S)L\Sigma^{1/2}\|
dt_1\dots dt_k 
\end{equation}
Using bounds (\ref{bd_LhR}), this yields 
\begin{equation}
\label{Bk_konecAA}
\|{\mathcal B}^{k}{\mathcal D}g(\Sigma)\|
\leq 
3^k 2^{k(2k+1)}\max_{0\leq j\leq k} 
\|D^{j}h\|_{L_{\infty}}
(\|\Sigma\|^{k+1}\vee \|\Sigma\|){\mathbb E}\prod_{i=1}^k (1+\delta_i)^{2k+1} \delta_i.
\end{equation}
Note that 
$$
{\mathbb E}\prod_{i=1}^k (1+\delta_i)^{2k+1} \delta_i=
\prod_{i=1}^k {\mathbb E}(1+\delta_i)^{2k+1} \delta_i
= \Bigl({\mathbb E}(1+\|W-I\|)^{2k+1}\|W-I\|\Bigr)^{k}
$$
and 
\begin{align*}
&
{\mathbb E}(1+\|W-I\|)^{2k+1}\|W-I\|= 
2^{2k+1} {\mathbb E}\biggl(\frac{1+\|W-I\|}{2}\biggr)^{2k+1} \|W-I\|
\\
&
\leq 
2^{2k+1} {\mathbb E}\frac{1+ \|W-I\|^{2k+1}}{2} \|W-I\|
= 2^{2k}\Bigl({\mathbb E} \|W-I\|+ {\mathbb E}\|W-I\|^{2k+2}\Bigr). 
\end{align*}
Using bound \eqref{bound_p_r_Sigma}, we get that with some constant $C_1\geq 1$
$$
{\mathbb E}\|W-I\|\leq {\mathbb E}^{1/(2k+2)}\|W-I\|^{2k+2} \leq 
C_1\biggl(\sqrt{\frac{d}{n}}\bigvee \sqrt{\frac{k}{n}}\biggr),
$$
which implies that 
\begin{align*}
&
{\mathbb E}(1+\|W-I\|)^{2k+1}\|W-I\|
\leq 2^{2k}
\biggl[C_1\biggl(\sqrt{\frac{d}{n}}\bigvee \sqrt{\frac{k}{n}}\biggr)
+ C_1^{2k+2} \biggl(\frac{d}{n} \bigvee \frac{k}{n}\biggr)^{k+1}
\biggr]
\\
&
\leq 2^{2k}C_1^{2k+2} \biggl(\sqrt{\frac{d}{n}}\bigvee \sqrt{\frac{k}{n}}\biggr)
\end{align*}
and 
\begin{equation}
\label{bd_prod_delta}
{\mathbb E}\prod_{i=1}^k (1+\delta_i)^{2k+1} \delta_i
\leq 
2^{2k^2}C_1^{2k^2+2k}\biggl(\frac{d}{n}\bigvee \frac{k}{n}\biggr)^{k/2}.
\end{equation}
We substitute this bound in (\ref{Bk_konecAA}) to get 
\begin{equation}
\label{Bk_konec}
\|{\mathcal B}^{k}{\mathcal D}g(\Sigma)\|
\leq 
3^k 2^{4k^2+k} C_1^{2k^2+2k}
\max_{0\leq j\leq k} 
\|D^{j}h\|_{L_{\infty}}
(\|\Sigma\|^{k+1}\vee \|\Sigma\|)
\biggl(\frac{d}{n}\bigvee \frac{k}{n}\biggr)^{k/2},
\end{equation}
which implies the result.

\end{proof}

\section{Smoothness properties of ${\mathcal D}g_k(\Sigma).$}
\label{Sec:smooth}

Our goal in this section is to show that, for a smooth orthogonally invariant function $g,$ the function ${\mathcal D}g_k(\Sigma)$ satisfies Assumption \ref{assume_Lipschitz_ABC}.
This result will be used in the next section to prove normal approximation 
bounds for ${\mathcal D}g_k(\hat \Sigma).$ 
We will assume in what follows that $g$ is defined and properly smooth {\it on the whole space}  
${\mathcal B}_{sa}({\mathbb H})$ of self-adjoint operators in ${\mathbb H}.$

Recall that, by formula (\ref{B_k_repres}),
$$
{\mathcal B}^k {\mathcal D}g(\Sigma)= 
{\mathbb E} \int_0^1 \dots \int_0^1 \frac{\partial^k \varphi (t_1,\dots, t_k)}{\partial t_1\dots \partial t_k}
dt_1\dots dt_k, 
$$
where 
$$
\varphi (t_1,\dots, t_k):= \Sigma^{1/2}R(t_1,\dots, t_k)
Dg(S(t_1,\dots, t_k))L(t_1,\dots, t_k) \Sigma^{1/2}, 
(t_1,\dots, t_k)\in [0,1]^k.
$$

Let $\delta \in (0,1/2)$ and let $\gamma:{\mathbb R}\mapsto {\mathbb R}$ be a $C^{\infty}$ function such that: 
\begin{align*}
&
0\leq \gamma (u)\leq \sqrt{u}, u\geq 0,
\gamma (u)=\sqrt{u}, u\in \Bigl[\delta, \frac{1}{\delta}\Bigr],
\\
&   
{\rm supp}(\gamma)\subset \Bigl[\frac{\delta}{2}, \frac{2}{\delta}\Bigr]\ \ 
{\rm and}\  \|\gamma\|_{B^1_{\infty,1}}\lesssim \frac{\log (2/\delta)}{\sqrt{\delta}}.
\end{align*}
For instance, one can take 
$\gamma(u):= \lambda(u/\delta)\sqrt{u}(1-\lambda(\delta u/2)),$
where $\lambda$ is a $C^{\infty}$ non-decreasing function with values in $[0,1],$ 
$\lambda(u)=0, u\leq 1/2$ and $\lambda (u)=1, u\geq 1.$ 
The bound on the norm $\|\gamma\|_{B^1_{\infty,1}}$ could be proved 
using equivalent definition of Besov norms in terms of difference 
operators (see \cite{Triebel}, Section 2.5.12). 
Clearly, for all $\Sigma\in {\mathcal C}_+({\mathbb H}),$ 
$\|\gamma (\Sigma)\|\leq \|\Sigma\|^{1/2}$ 
and for all $\Sigma\in {\mathcal C}_+({\mathbb H})$ with $\sigma(\Sigma)\subset \Bigl[\delta, \frac{1}{\delta}\Bigr],$
we have $\gamma (\Sigma)=\Sigma^{1/2}.$

Since we need further differentiation of ${\mathcal B}^k {\mathcal D}g(\Sigma)$ with respect to $\Sigma,$ it will 
be convenient to introduce (for given $H,H'\in {\mathcal B}_{sa}({\mathbb H})$) the following function:
\begin{align*}
&
\phi (t_1,\dots, t_k;s_1,s_2):=
\\
&
\gamma(\bar \Sigma(s_1,s_2)) R(t_1,\dots, t_k)
Dg\Bigl(L(t_1,\dots, t_k)\bar \Sigma(s_1,s_2)R(t_1,\dots, t_k)\Bigr)L(t_1,\dots, t_k) 
\gamma(\bar \Sigma(s_1,s_2)), 
\end{align*}
where $\bar \Sigma(s_1,s_2)= \Sigma + s_1 H + s_2(H'-H), s_1,s_2\in {\mathbb R}.$
Note that 
$
\varphi(t_1,\dots, t_k)=\phi(t_1,\dots, t_k,0,0).
$ 
By the argument already used at the beginning of the proof of Lemma \ref{chain}, 
if $h:=Dg$ is $k$ times continuously differentiable, then the function $\phi$ 
is also $k$ times continuously differentiable.

For simplicity, we write in what follows $B_k(\Sigma):={\mathcal B}^k {\mathcal D}g(\Sigma)$ and $D_k(\Sigma):={\mathcal D}g_k(\Sigma).$
Clearly, 
$
D_k(\Sigma):= \sum_{j=0}^k (-1)^j B_j(\Sigma)
$ 
and the following representation holds: 
$$
B_k (\Sigma):= {\mathbb E} \int_{0}^1\dots \int_0^1\frac{\partial^k \phi (t_1,\dots, t_k,0,0)}{\partial t_1\dots \partial t_k}dt_1\dots dt_k, k\geq 1.
$$ 
For $k=0,$ we have $B_0(\Sigma):= {\mathcal D}g(\Sigma).$

Denote 
$$
\gamma_{\beta,k}(\Sigma;u):= 
(\|\Sigma\|\vee u \vee 1)^{k+1/2} (u \vee u^{\beta}), u>0, \beta \in [0,1], k\geq 1. 
$$
Recall definition \eqref{define_C^s} of $C^s$-norms
of smooth operator valued functions defined in an open set $G\subset {\mathcal B}_{sa}({\mathbb H}).$
It is assumed in this section that $G={\mathcal B}_{sa}({\mathbb H})$ and we will write 
$\|\cdot\|_{C^s}$ instead of $\|\cdot \|_{C^s({\mathcal B}_{sa}({\mathbb H}))}.$

\begin{theorem}
\label{th:remainder_B_k}
Suppose that, for some $k\leq d,$ 
$g$ is $k+2$ times continuously differentiable in ${\mathcal B}_{sa}({\mathbb H})$
and, for some $\beta\in (0,1],$  $\|Dg\|_{C^{k+1+\beta}}<\infty.$
In addition, suppose that $g\in L_{\infty}^{O}({\mathcal C}_+({\mathbb H}))$ and $\sigma(\Sigma)\subset \Bigl[\delta, \frac{1}{\delta}\Bigr].$
Then, for some constant $C\geq 1$ and for all $H,H'\in {\mathcal B}_{sa}({\mathbb H})$
\begin{align}
\label{remainder_B_k}
&
\|S_{B_k}(\Sigma;H')-S_{B_k}(\Sigma;H)\|
\leq C^{k^2}\frac{\log^2 (2/\delta)}{\delta}
\|Dg\|_{C^{k+1+\beta}}
\biggl(\frac{d}{n}\biggr)^{k/2}\gamma_{\beta,k}(\Sigma;\|H\|\vee \|H'\|)\|H'-H\|.
\end{align}
\end{theorem}

\begin{corollary}
\label{cor:rem_D_k}
Suppose that, for some $k\leq d,$
$g$ is $k+2$ times continuously differentiable and, for some $\beta\in (0,1],$
$\|Dg\|_{C^{k+1+\beta}}<\infty.$ Suppose also that $g\in L_{\infty}^{O}({\mathcal C}_+({\mathbb H})),$ $d\leq n/2$ and that 
$\sigma(\Sigma)\subset \Bigl[\delta, \frac{1}{\delta}\Bigr].$
Then, for some constant $C\geq 1$ and for all $H,H'\in {\mathcal B}_{sa}({\mathbb H})$
\begin{align}
\label{remainder_D_k_beta}
&
\|S_{D_k}(\Sigma;H')-S_{D_k}(\Sigma;H)\|
\leq
C^{k^2}\frac{\log^2 (2/\delta)}{\delta}
\|Dg\|_{C^{k+1+\beta}}
\gamma_{\beta,k}(\Sigma;\|H\|\vee \|H'\|)\|H'-H\|.
\end{align}
\end{corollary}

\begin{proof}
Indeed,
$$
\|S_{D_k}(\Sigma;H')-S_{D_k}(\Sigma;H)\|\leq 
\sum_{j=0}^k \|S_{B_j}(\Sigma;H')-S_{B_j}(\Sigma;H)\|
$$
$$
\leq
C^{k^2}\frac{\log^2 (2/\delta)}{\delta}
\|Dg\|_{C^{k+1+\beta}}
\sum_{j=0}^k\biggl(\frac{d}{n}\biggr)^{j/2}
\gamma_{\beta,k}(\Sigma;\|H\|\vee \|H'\|)\|H'-H\|
$$
$$
\leq
2C^{k^2}\frac{\log^2 (2/\delta)}{\delta}
\|Dg\|_{C^{k+1+\beta}}
\gamma_{\beta,k}(\Sigma;\|H\|\vee \|H'\|)\|H'-H\|,
$$
implying the bound of the corollary (after proper adjustment of the value of $C$).

\end{proof}

We now give the proof of Theorem \ref{th:remainder_B_k}.

\begin{proof}
Note that 
$$
S_{B_k}(\Sigma;H')-S_{B_k}(\Sigma;H)=
DB_k (\Sigma+H;H'-H)-DB_k(\Sigma;H'-H)
+S_{B_k}(\Sigma+H; H'-H),
$$
so, we need to bound 
$$
\|DB_k (\Sigma+H;H'-H)-DB_k(\Sigma;H'-H)\|\ \ 
{\rm and}\ \ 
\|S_{B_k}(\Sigma+H; H'-H)\|
 $$ 
separately. 
To this end, note that 
$$
B_k(\Sigma+s_1H) = {\mathbb E} \int_{0}^1\dots \int_0^1\frac{\partial^k \phi (t_1,\dots, t_k,s_1,0)}{\partial t_1\dots \partial t_k}dt_1\dots dt_k,
$$
$$
DB_k(\Sigma ;H) = {\mathbb E} \int_{0}^1\dots \int_0^1\frac{\partial^{k+1} 
\phi (t_1,\dots, t_k,0,0)}{\partial t_1\dots \partial t_k \partial s_1}dt_1\dots dt_k,
$$
and 
$$
DB_k(\Sigma+s_1H; H'-H)= {\mathbb E} \int_{0}^1\dots \int_0^1\frac{\partial^{k+1} \phi (t_1,\dots, t_k,s_1,0)}{\partial t_1\dots \partial t_k \partial s_2}dt_1\dots dt_k.
$$
The last two formulas hold provided that $g$ is $k+2$ times continuously differentiable 
with uniformly bounded derivatives $D^j g, j=0,\dots, k+2$ and, as a consequence,
the function $\phi(t_1,\dots, t_k,s_1,s_2)$ is $k+1$ times continuously differentiable 
(the proof of this fact is similar to the proof of differentiability of function $\phi(t_1,\dots, t_k),$ see the proofs of 
Proposition \ref{commute} and Lemma \ref{chain}).

As a consequence,
\begin{align}
\label{D_B_k}
&
\nonumber
DB_k (\Sigma+H;H'-H)-DB_k(\Sigma;H'-H)
\\
&
 ={\mathbb E} \int_{0}^1\dots \int_0^1
 \biggl[
 \frac{\partial^{k+1} \phi (t_1,\dots, t_k,1,0)}{\partial t_1\dots \partial t_k \partial s_2}
 -
 \frac{\partial^{k+1} \phi (t_1,\dots, t_k,0,0)}{\partial t_1\dots \partial t_k \partial s_2}\biggr]
 dt_1\dots dt_k
\end{align}
 and 
\begin{align}
\label{S_B_k}
&
\nonumber
S_{B_k}(\Sigma+H; H'-H)
\\
&
\nonumber
={\mathbb E} \int_{0}^1\dots \int_0^1
\biggl[
\frac{\partial^{k} \phi (t_1,\dots, t_k,1,1)}{\partial t_1\dots \partial t_k }
- \frac{\partial^{k} \phi (t_1,\dots, t_k,1,0)}{\partial t_1\dots \partial t_k }
- \frac{\partial^{k+1} \phi (t_1,\dots, t_k,1,0)}{\partial t_1\dots \partial t_k \partial s_2}
\biggr]dt_1\dots dt_k 
\\
&
=
{\mathbb E} \int_{0}^1\dots \int_0^1 
\int_0^1\biggl[
\frac{\partial^{k+1} \phi (t_1,\dots, t_k,1,s_2)}{\partial t_1\dots \partial t_k \partial s_2}
- \frac{\partial^{k+1} \phi (t_1,\dots, t_k,1,0)}{\partial t_1\dots \partial t_k \partial s_2}
\biggr] ds_2 dt_1\dots dt_k.
\end{align}

The next two lemmas provide upper bounds on 
$$
\|DB_k (\Sigma+H;H'-H)-DB_k(\Sigma;H'-H)\|\ \ 
{\rm and}\ \ 
\|S_{B_k}(\Sigma+H; H'-H)\|
 $$ 

\begin{lemma}
\label{lem_rem_A}
Suppose that, for some $k\leq d,$
$g\in L_{\infty}^{O}({\mathcal C}_+({\mathbb H}))$ is $k+2$ times continuously differentiable and, for some $\beta\in (0,1],$
$\|Dg\|_{C^{k+1+\beta}}<\infty.$
In addition, suppose that 
$\sigma(\Sigma)\subset \Bigl[\delta, \frac{1}{\delta}\Bigr].$
Then, for some constant $C>0$ and for all $H,H'\in {\mathcal B}_{sa}({\mathbb H})$
\begin{align}
\label{remainder_A}
&
\|DB_k (\Sigma+H;H'-H)-DB_k(\Sigma;H'-H)\|
\\
&
\nonumber
\leq
C^{k^2}\frac{\log^2 (2/\delta)}{\delta}
\|Dg\|_{C^{k+1+\beta}}
((\|\Sigma\|+\|H\|)^{k+1/2} \vee 1)
\biggl(\frac{d}{n}\biggr)^{k/2}
(\|H\|\vee \|H\|^{\beta})
\|H'-H\|.
\end{align}
\end{lemma}

\begin{lemma}
\label{lem_rem_B}
Suppose that, for some $k\leq d,$ $g\in L_{\infty}^{O}({\mathcal C}_+({\mathbb H}))$ is
$k+2$ times continuously differentiable and, for some $\beta\in (0,1],$ $\|Dg\|_{C^{k+1+\beta}}<\infty.$
In addition, suppose that 
$\sigma(\Sigma)\subset \Bigl[\delta, \frac{1}{\delta}\Bigr].$
Then, for some constant $C>0$ and for all $H,H'\in {\mathcal B}_{sa}({\mathbb H})$
\begin{align}
\label{remainder_B}
&
\|S_{B_k}(\Sigma+H; H'-H)\|
\\
&
\nonumber
\leq
C^{k^2}\frac{\log^2 (2/\delta)}{\delta}
\|Dg\|_{C^{k+1+\beta}}
((\|\Sigma\|+\|H\|+\|H'\|)^{k+1/2} \vee 1)\biggl(\frac{d}{n}\biggr)^{k/2}(\|H'-H\|^{1+\beta}\vee \|H'-H\|^2).
\end{align}
\end{lemma}

In the next section, we will also need the following lemma.

\begin{lemma}
\label{lem_rem_C}
Suppose that, for some $k\leq d,$ $g\in L_{\infty}^{O}({\mathcal C}_+({\mathbb H}))$ is $k+2$ times differentiable with uniformly bounded  continuous derivatives $D^j g, j=0,\dots, k+2.$ In addition, suppose that 
$\sigma(\Sigma)\subset \Bigl[\delta, \frac{1}{\delta}\Bigr].$
Then, for some constant $C>0$ and for all $H\in {\mathcal B}_{sa}({\mathbb H}),$
\begin{align}
\label{remainder_C''}
&
\|DB_k (\Sigma;H)\|
\leq
C^{k^2}\frac{\log^2 (2/\delta)}{\delta}
\|Dg\|_{C^{k+1}}
(\|\Sigma\|^{k+1/2} \vee 1)\biggl(\frac{d}{n}\biggr)^{k/2}\|H\|.
\end{align}
\end{lemma}

We give below a proof of Lemma \ref{lem_rem_A}.
The proofs of lemmas \ref{lem_rem_B}, \ref{lem_rem_C} are based on a similar 
approach.

\begin{proof}
First, we derive an upper bound on the difference 
$$
 \frac{\partial^{k+1} \phi (t_1,\dots, t_k,1,0)}{\partial t_1\dots \partial t_k \partial s_2}
 -
 \frac{\partial^{k+1} \phi (t_1,\dots, t_k,0,0)}{\partial t_1\dots \partial t_k \partial s_2}
$$
in the right hand side of (\ref{D_B_k}). To this end, note that by the product rule
$$
\frac{\partial^{k+1} \phi (t_1,\dots, t_k,s_1,s_2)}{\partial t_1\dots \partial t_k \partial s_2}
= 
\frac{\partial}{\partial s_2} 
\sum_{T_1,T_2,T_3} \gamma(\bar \Sigma) (\partial_{T_1}R) 
(\partial_{T_2} h(L\bar \Sigma R))(\partial_{T_3} L) \gamma(\bar \Sigma),
$$
where $h=Dg$ and the sum extends to all partitions $T_1,T_2,T_3$ 
of the set of variables $\{t_1,\dots, t_k\}.$ We can further write 
\begin{align}
\label{partial_k+1}
&
\nonumber
\frac{\partial^{k+1} \phi (t_1,\dots, t_k,s_1,s_2)}{\partial t_1\dots \partial t_k \partial s_2}
\\
&
\nonumber
= 
\sum_{T_1,T_2,T_3} 
\Bigl[(\partial_{\{s_2\}}\gamma(\bar \Sigma)) (\partial_{T_1}R) 
(\partial_{T_2} h(L\bar \Sigma R))(\partial_{T_3} L) \gamma(\bar \Sigma)
+
\gamma(\bar \Sigma) (\partial_{T_1}R) 
(\partial_{T_2\cup \{s_2\}} h(L\bar \Sigma R))(\partial_{T_3} L) \gamma(\bar \Sigma)
\\
&
+
\gamma(\bar \Sigma) (\partial_{T_1}R) 
(\partial_{T_2} h(L\bar \Sigma R))(\partial_{T_3} L) (\partial_{\{s_2\}}\gamma(\bar \Sigma))
\Bigr].
\end{align}
Observe that 
$
\partial_{\{s_2\}}\gamma(\bar \Sigma)=D\gamma(\bar \Sigma;H'-H)
$
and deduce from (\ref{partial_k+1}) that 
\begin{equation}
\label{partial_diff_1}
\frac{\partial^{k+1} \phi (t_1,\dots, t_k,1,0)}{\partial t_1\dots \partial t_k \partial s_2}
 -
 \frac{\partial^{k+1} \phi (t_1,\dots, t_k,0,0)}{\partial t_1\dots \partial t_k \partial s_2}
= \sum_{T_1,T_2,T_3}[A_1+A_2+\dots +A_9], 
\end{equation}
where 
\begin{align*}
&
A_1:=[D\gamma(\Sigma+H;H'-H)-D\gamma(\Sigma;H'-H)] (\partial_{T_1}R) 
(\partial_{T_2} h(L\bar \Sigma_{1,0} R))(\partial_{T_3} L) \gamma(\bar \Sigma_{1,0}), 
\\
&
A_2:=D\gamma(\Sigma;H'-H) (\partial_{T_1}R) 
(\partial_{T_2} h(L(\Sigma+H)R)-\partial_{T_2}h(L\Sigma R))(\partial_{T_3} L) \gamma(\bar \Sigma_{1,0}),
\\
&
A_3:= 
D\gamma(\Sigma;H'-H) (\partial_{T_1}R) 
(\partial_{T_2} h(L \Sigma R))(\partial_{T_3} L) (\gamma(\Sigma+H)-\gamma(\Sigma)),
\\
&
A_4:= 
(\gamma(\Sigma+H)-\gamma(\Sigma)) (\partial_{T_1}R) 
(\partial_{T_2\cup \{s_2\}} h(L\bar \Sigma_{1,0} R))(\partial_{T_3} L) \gamma(\bar \Sigma_{1,0})
\\
&
A_5:= 
\gamma(\Sigma) (\partial_{T_1}R) 
(\partial_{T_2\cup \{s_2\}} h(L(\Sigma+H) R) - \partial_{T_2\cup \{s_2\}} h(L\Sigma R))(\partial_{T_3} L) \gamma(\bar \Sigma_{1,0}),
\\
&
A_6:= \gamma(\Sigma) (\partial_{T_1}R) 
(\partial_{T_2\cup \{s_2\}} h(L\Sigma R))(\partial_{T_3} L) (\gamma(\Sigma+H)-\gamma(\Sigma)), 
\\
&
A_7:= 
(\gamma(\Sigma+H)-\gamma(\Sigma)) (\partial_{T_1}R) 
(\partial_{T_2} h(L\bar \Sigma_{1,0} R))(\partial_{T_3} L) 
D\gamma(\bar \Sigma_{1,0};H'-H)
\\
&
A_8:=
\gamma(\Sigma) (\partial_{T_1}R) 
(\partial_{T_2} h(L(\Sigma+H) R)-\partial_{T_2}h(L\Sigma R))(\partial_{T_3} L) 
D\gamma(\bar \Sigma_{1,0};H'-H)
\end{align*}
and
$$
A_9:= 
\gamma(\Sigma) (\partial_{T_1}R) 
(\partial_{T_2} h(L\Sigma R))(\partial_{T_3} L) (D\gamma(\Sigma+H;H'-H)-D\gamma(\Sigma;H'-H)).
$$

To bound the norms of operators $A_1,A_2,\dots, A_9,$ we need several lemmas.
We introduce here some notation used in their proofs. Recall that for 
a partition $\Delta= (\Delta_1,\dots, \Delta_j)$ of the set $\{t_1,\dots, t_k\}$
$$ 
\partial_{\Delta} (L\Sigma R)= 
(\partial_{\Delta_1} (L\Sigma R), \dots , \partial_{\Delta_j}(L\Sigma R)).
$$
We will need some transformations of $\partial_{\Delta}(L\Sigma R).$
In particular, for $i=1,\dots, j$ and $H\in {\mathcal B}_{sa}({\mathbb H}),$
denote 
$$ 
\partial_{\Delta} (L\Sigma R)[i: \Sigma \rightarrow H]= 
(\partial_{\Delta_1} (L\Sigma R), \dots ,\partial_{\Delta_{i-1}}(L\Sigma R),
\partial_{\Delta_i}(LHR), \partial_{\Delta_{i+1}}(L\Sigma R), \dots, \partial_{\Delta_j}(L\Sigma R)).
$$
We will also write 
$$ 
\partial_{\Delta} (L\Sigma R)[i: \Sigma \rightarrow H; i+1,\dots, j: \Sigma \rightarrow \Sigma+H]
$$
$$
= 
(\partial_{\Delta_1} (L\Sigma R), \dots ,\partial_{\Delta_{i-1}}(L\Sigma R),
\partial_{\Delta_i}(LHR), \partial_{\Delta_{i+1}}(L(\Sigma+H) R), \dots, \partial_{\Delta_j}(L(\Sigma+H) R)).
$$
In addition, the following notation will be used:
$$
\partial_{\Delta}(L\Sigma R)\sqcup B = (\partial_{\Delta_1} (L\Sigma R), \dots , \partial_{\Delta_j}(L\Sigma R), B)
$$
The meaning of other similar notation should be clear from the context.
Finally, recall that $\delta_i =\|W_i-I\|, i\geq 1.$

\begin{lemma}
\label{chain+lip}
Suppose that, for some $0\leq m\leq k,$ function $h\in L_{\infty}({\mathcal C}_+({\mathbb H}); {\mathcal B}_{sa}({\mathbb H}))$ is $m+1$ times differentiable with uniformly bounded  continuous derivatives $D^j h, j=1,\dots, m+1.$
For all $T\subset \{t_1,\dots, t_k\}$ with $|T|=m,$
\begin{align}
 \label{partial_bd_lip}
 &
\|\partial_T h(L(\Sigma+H)R)-\partial_T h(L\Sigma R)\|
\\
&
\nonumber
\leq  2^{m(k+m+2)+1}
\max_{1\leq j\leq m+1}\|D^{j}h\|_{L_{\infty}}
((\|\Sigma\|+\|H\|)^m \vee 1)\prod_{t_i\in T}\frac{\delta_i}{1+\delta_i} \prod_{i=1}^k (1+\delta_i)^{2m+2}\|H\|.
\end{align}
\end{lemma}

\begin{proof}
By Lemma \ref{chain},
\begin{align}
\label{follow_chain}
&
\partial_T h(L(\Sigma+H )R)-\partial_T h(L\Sigma R)= 
\\
&
\nonumber
\sum_{j=1}^m \sum_{\Delta\in {\mathcal D}_{j,T}} 
\Bigl[D^j h(L(\Sigma+H)R)(\partial_{\Delta}(L(\Sigma+H) R))-
D^j h(L\Sigma R)(\partial_{\Delta}(L\Sigma R))\Bigr].
\end{align}
Obviously,
\begin{align*}
&
D^j h(L(\Sigma+H)R)(\partial_{\Delta}(L(\Sigma+H) R))-
D^j h(L\Sigma R)(\partial_{\Delta}(L\Sigma R))
\\
&
=\sum_{i=1}^j D^j h(L(\Sigma+H)R)
(\partial_{\Delta}(L\Sigma R)[i: \Sigma\rightarrow H;
i+1,\dots, j: \Sigma \rightarrow \Sigma+H])
\\
&
+ (D^j h(L(\Sigma+H)R)- D^j h(L\Sigma R))(\partial_{\Delta}(L\Sigma R)).
\end{align*}
The following bounds hold for all $1\leq i\leq j:$
\begin{align*}
&
\Bigl\|D^j h(L(\Sigma+H)R)
(\partial_{\Delta}(L\Sigma R)[i: \Sigma\rightarrow H;
i+1,\dots, j: \Sigma \rightarrow \Sigma+H])\Bigr\|
\\
&
\leq 
\|D^j h(L(\Sigma+H) R)\|\prod_{1\leq l<i} \|\partial_{\Delta_l} (L\Sigma R)\|
\prod_{i<l\leq j} \|\partial_{\Delta_l} (L(\Sigma+H)R)\| \|\partial_{\Delta_i}(LHR)\|.
\end{align*}
As in the proof of Lemma \ref{chain+A}, we get 
\begin{align*}
&
\Bigl\|D^j h(L(\Sigma+H)R)
(\partial_{\Delta}(L\Sigma R)[i: \Sigma\rightarrow H;
i+1,\dots, j: \Sigma \rightarrow \Sigma+H])\Bigr\|
\\
&
\leq \|D^j h\|_{L_{\infty}} 2^{kj} \|\Sigma\|^{i-1}\|\Sigma+H\|^{j-i} 
\prod_{t_i\in T}\frac{\delta_i}{1+\delta_i} \prod_{i=1}^k (1+\delta_i)^{2j} \|H\|
\\
&
\leq 2^{kj} \|D^j h\|_{L_{\infty}}(\|\Sigma\|+\|H\|)^{j-1} 
\prod_{t_i\in T}\frac{\delta_i}{1+\delta_i} \prod_{i=1}^k (1+\delta_i)^{2j} \|H\|.
\end{align*}
In addition, 
\begin{align*}
&
\|(D^j h(L(\Sigma+H)R)- D^j h(L\Sigma R))(\partial_{\Delta}(L\Sigma R))\|
\\
&
\leq \|D^{j+1} h\|_{L_{\infty}} \|L\|\|R\|\|H\| \prod_{i=1}^j \|\partial_{\Delta_i} (L\Sigma R)\|
\\
&
\leq 2^{kj}\|D^{j+1} h\|_{L_{\infty}} \|\Sigma\|^j 
\prod_{t_i\in T}\frac{\delta_i}{1+\delta_i} \prod_{i=1}^k (1+\delta_i)^{2j+2} \|H\|.
\end{align*}
Therefore,
\begin{align*}
&
\Bigl\|D^j h(L(\Sigma+H)R)(\partial_{\Delta}(L(\Sigma+H) R))-
D^j h(L\Sigma R)(\partial_{\Delta}(L\Sigma R))\Bigr\|
\\
&
\leq 
2^{kj}\Bigl(j \|D^j h\|_{L_{\infty}}(\|\Sigma\|+\|H\|)^{j-1}+ \|D^{j+1} h\|_{L_{\infty}} 
\|\Sigma\|^j\Bigr)\prod_{t_i\in T}\frac{\delta_i}{1+\delta_i} \prod_{i=1}^k (1+\delta_i)^{2j+2} \|H\|.
\end{align*}
Substituting the last bound to (\ref{follow_chain}) and recalling that 
${\rm card}({\mathcal D}_{j,T})\leq j^m,$ it is easy to conclude 
the proof of bound (\ref{partial_bd_lip}).

\end{proof}

\begin{lemma}
\label{bd_s_2}
Suppose that, for some $0\leq m\leq k,$ $h\in L_{\infty}({\mathcal C}_+({\mathbb H}); {\mathcal B}_{sa}({\mathbb H}))$ is $m+1$ times differentiable with uniformly bounded  continuous derivatives $D^j h, j=1,\dots, m+1.$
Then for some constant $C>0$ and for all $T\subset \{t_1,\dots, t_k\}$ with $|T|=m$ and all $s_1\in [0,1],$
\begin{align}
 \label{partial_bd_s_2}
 &
\|\partial_{T\cup \{s_2\}} h(L \bar \Sigma_{s_1,0}R)\|
\\
&
\nonumber
\leq  2^{m(k+m+2)+1}
\max_{1\leq j\leq m+1}\|D^{j}h\|_{L_{\infty}}
((\|\Sigma\|+\|H\|)^m \vee 1)
\prod_{t_i\in T}\frac{\delta_i}{1+\delta_i} \prod_{i=1}^k (1+\delta_i)^{2m+2}
 \|H'-H\|.
\end{align}
\end{lemma}

\begin{proof}
By Lemma \ref{chain},
\begin{equation}
\label{part_s_2}
\partial_{T\cup \{s_2\}} h(L\bar \Sigma R)=
\sum_{j=1}^m \sum_{\Delta\in {\mathcal D}_{j,T}}
\partial_{\{s_2\}}
D^j h(L\bar \Sigma R)(\partial_{\Delta}(L\bar \Sigma R)).
\end{equation}
Next, we have 
\begin{align*}
&
\partial_{\{s_2\}} 
D^j h(L \bar \Sigma R)(\partial_{\Delta}(L\bar \Sigma R))
=D^{j+1} h(L \bar \Sigma R) (\partial_{\Delta} (L\bar \Sigma R)\sqcup \partial_{\{s_2\}}(L\bar \Sigma R))
\\
&
+\sum_{i=1}^j 
D^{j}h(L\bar \Sigma R)
(\partial_{\Delta}((L\bar \Sigma R)[i: \partial_{\Delta_i}(L\bar \Sigma R)
\rightarrow \partial_{\Delta_i \cup \{s_2\}} (L\bar \Sigma R)]).
\end{align*}
Note that 
$
\partial_{\{s_2\}}(L\bar \Sigma R)= L(H'-H)R
$
and 
$
\partial_{\Delta_i \cup \{s_2\}} (L\bar \Sigma R)= \partial_{\Delta_i}(L(H'-H)R),$
implying 
\begin{align}
\label{part_s_2_s_2}
&
\nonumber
\partial_{\{s_2\}} 
D^j h(L \bar \Sigma R)(\partial_{\Delta}(L\bar \Sigma R))
=D^{j+1} h(L \bar \Sigma R) (\partial_{\Delta} (L\bar \Sigma R)\sqcup L(H'-H)R)
\\
&
+\sum_{i=1}^j 
D^{j}h(L\bar \Sigma R)
(\partial_{\Delta}(L\bar \Sigma R)[i: \bar \Sigma \rightarrow H'-H]).
\end{align}
The following bounds hold:
\begin{align*}
&
\Bigl\|
D^{j+1} h(L \bar \Sigma R) (\partial_{\Delta} (L\bar \Sigma R)\sqcup L(H'-H)R)
\Bigr\|
\\
&
\leq \|D^{j+1}h\|_{L_{\infty}} \prod_{i=1}^j \|\partial_{\Delta_i}(L\bar \Sigma R)\|
\|L\|\|R\|\|H'-H\|
\end{align*}
and 
\begin{align*}
&
\Bigl\|
D^{j}h(L\bar \Sigma R)
(\partial_{\Delta}(L\bar \Sigma R)[i: \bar \Sigma \rightarrow H'-H])
\Bigr\|
\\
&
\leq 
\|D^{j}h\|_{L_{\infty}} \prod_{l\neq i} \|\partial_{\Delta_l}(L\bar \Sigma R)\|
\|\partial_{\Delta_i}(L(H'-H)R)\|.
\end{align*}
The rest of the proof is based on the bounds almost identical 
to the ones in the proof of Lemma \ref{chain+lip}.

\end{proof}

\begin{lemma}
\label{chain+lip_s_2}
Suppose that, for some $0\leq m\leq k,$ $h\in L_{\infty}({\mathcal C}_+({\mathbb H}); {\mathcal B}_{sa}({\mathbb H}))$ is $m+2$ times differentiable with uniformly bounded  continuous derivatives $D^j h, j=1,\dots, m+2.$
For some constant $C>0$ and for all $T\subset \{t_1,\dots, t_k\}$ with $|T|=m,$
\begin{align}
 \label{partial_bd_lip_s_2}
 &
\Bigl\|\partial_{T\cup {\{s_2\}}}h(L(\Sigma+H)R)- 
\partial_{T\cup {\{s_2\}}}h(L\Sigma R)\Bigr\|
\\
&
\nonumber
\leq  
C^{k(m+1)}
\max_{1\leq j\leq m+2}\|D^{j}h\|_{L_{\infty}}
((\|\Sigma\|+\|H\|)^m \vee 1)
\prod_{t_i\in T}\frac{\delta_i}{1+\delta_i} \prod_{i=1}^k (1+\delta_i)^{2m+4}
\|H\|\|H'-H\|.
\end{align}
Moreover, if for some $0\leq m\leq k,$ $h\in L_{\infty}({\mathcal C}_+({\mathbb H}); {\mathcal B}_{sa}({\mathbb H}))$ is $m+1$ times continuously differentiable and, for some $\beta\in (0,1],$
$\|h\|_{C^{m+1+\beta}}<\infty,$ then 
\begin{align}
 \label{partial_bd_lip_s_2_beta}
 &
\Bigl\|\partial_{T\cup {\{s_2\}}}h(L(\Sigma+H)R)- 
\partial_{T\cup {\{s_2\}}}h(L\Sigma R)\Bigr\|
\\
&
\nonumber
\leq  
C^{k(m+1)}
\|h\|_{C^{m+1+\beta}}
((\|\Sigma\|+\|H\|)^m \vee 1)
\prod_{t_i\in T}\frac{\delta_i}{1+\delta_i} \prod_{i=1}^k (1+\delta_i)^{2m+4}
(\|H\|\vee \|H\|^{\beta})\|H'-H\|.
\end{align}
\end{lemma}

\begin{proof}
By (\ref{part_s_2}), 
\begin{align}
\label{bd_T_s'}
&
\partial_{T\cup {\{s_2\}}}h(L(\Sigma+H)R)- 
\partial_{T\cup {\{s_2\}}}h(L\Sigma R)
\\
&
\nonumber
=\sum_{j=1}^m \sum_{\Delta\in {\mathcal D}_{j,T}}
\Bigl[
\partial_{\{s_2\}} D^j h(L\bar \Sigma_{1,0}R)(\partial_{\Delta}(L\bar \Sigma_{1,0} R))
-\partial_{\{s_2\}} D^j h(L\bar \Sigma_{0,0} R)(\partial_{\Delta}(L\bar \Sigma_{0,0} R))
\Bigr]
\end{align}
and by (\ref{part_s_2_s_2}),
\begin{align}
\label{bd_T_s''}
&
\nonumber
\partial_{\{s_2\}} D^j h(L\bar \Sigma_{1,0}R)(\partial_{\Delta}(L\bar \Sigma_{1,0} R))
-\partial_{\{s_2\}} D^j h(L\bar \Sigma_{0,0} R)(\partial_{\Delta}(L\bar \Sigma_{0,0} R))
\\
&
\nonumber
=
\sum_{i=1}^j 
D^{j+1} h(L(\Sigma+H)R)
(\partial_{\Delta}(L\Sigma R)[i: \Sigma \rightarrow H; i+1,\dots, j: \Sigma\rightarrow  \Sigma+H]\sqcup L(H'-H)R)
\\
&
\nonumber 
+[D^{j+1}h(L(\Sigma+H) R)- D^{j+1}h(L\Sigma R)](\partial_{\Delta}(L\Sigma R)\sqcup
L(H'-H)R)
\\
&
\nonumber
+\sum_{i=1}^j \sum_{i'\neq i}
D^{j}h(L(\Sigma+H)R)
(\partial_{\Delta}(L \Sigma R)[i: \Sigma \rightarrow H'-H;
i': \Sigma \rightarrow H; l>i', l\neq i: \Sigma \rightarrow \Sigma+H])
\\
&
+\sum_{i=1}^j [D^{j}h(L(\Sigma+H)R)-D^{j}h(L \Sigma R)](\partial_{\Delta}(L\Sigma R)[i: \Sigma\rightarrow H'-H]).
\end{align}
Similarly to the bounds in the proof of Lemma \ref{chain+lip},
we get 
\begin{align*}
&
\Bigl\|D^{j+1} h(L(\Sigma+H)R)
(\partial_{\Delta}(L\Sigma R)[i: \Sigma \rightarrow H; i+1,\dots, j: \Sigma\rightarrow  \Sigma+H]\sqcup L(H'-H)R)\Bigr\|
\\
&
\leq 
2^{kj}\|D^{j+1} h\|_{L_{\infty}} (\|\Sigma\|+\|H\|)^{j-1} 
\prod_{t_i\in T}\frac{\delta_i}{1+\delta_i} \prod_{i=1}^k (1+\delta_i)^{2j+2}
\|H\| \|H'-H\|,
\\
&
\Bigl\|[D^{j+1}h(L(\Sigma+H) R)- D^{j+1}h(L\Sigma R)](\partial_{\Delta}(L\Sigma R)\sqcup
L(H'-H)R)\Bigr\|
\\
&
\leq 
2^{kj}\|D^{j+2} h\|_{L_{\infty}} \|\Sigma\|^j 
\prod_{t_i\in T}\frac{\delta_i}{1+\delta_i} \prod_{i=1}^k (1+\delta_i)^{2j+4}
\|H\|\|H'-H\|,
\\
&
\Bigl\|D^{j}h(L(\Sigma+H)R)
(\partial_{\Delta}(L \Sigma R)[i: \Sigma \rightarrow H'-H;
i': \Sigma \rightarrow H; l>i', l\neq i: \Sigma \rightarrow \Sigma+H])\Bigr\|
\\
&
\leq 
2^{kj}\|D^{j} h\|_{L_{\infty}} (\|\Sigma\|+\|H\|)^{j-2} 
\prod_{t_i\in T}\frac{\delta_i}{1+\delta_i} \prod_{i=1}^k (1+\delta_i)^{2j}
\|H\| \|H'-H\|
\end{align*}
and 
\begin{align*}
&
 \Bigl\|[D^{j}h(L(\Sigma+H)R)-D^{j}h(L \Sigma R)](\partial_{\Delta}(L\Sigma R)[i: \Sigma\rightarrow H'-H])\Bigr\|
\\
&
\leq 2^{kj}\|D^{j+1} h\|_{L_{\infty}} \|\Sigma\|^{j-1} 
\prod_{t_i\in T}\frac{\delta_i}{1+\delta_i} \prod_{i=1}^k (1+\delta_i)^{2j}
\|H\|\|H'-H\|.
\end{align*}
These bounds along with formulas (\ref{bd_T_s'}), (\ref{bd_T_s''}) imply 
that bound (\ref{partial_bd_lip_s_2}) holds.
The proof of bound \eqref{partial_bd_lip_s_2_beta}
is similar. 

\end{proof}

We now get back to bounding operators $A_1,\dots, A_9$ in the 
right hand side of (\ref{partial_diff_1}). It easily follows 
from lemmas \ref{bd_S}, \ref{chain+A}, \ref{chain+lip}, \ref{bd_s_2} and 
\ref{chain+lip_s_2} as well as from the bounds 
\begin{align*}
&
\|\gamma(\Sigma)\|\leq \|\Sigma\|^{1/2},
\\
&
\|\gamma (\Sigma+H)-\gamma(\Sigma)\| \leq 2\|\gamma\|_{B^{1}_{\infty,1}} \|H\|
\lesssim \frac{\log (2/\delta)}{\sqrt{\delta}}\|H\|,
\\
&
\|D\gamma (\Sigma;H)\|\leq 2\|\gamma\|_{B^{1}_{\infty,1}} \|H\| \lesssim \frac{\log(2/\delta)}{\sqrt{\delta}} \|H\|
\end{align*}
and 
\begin{align*}
\|D\gamma(\Sigma+H;H'-H)-D\gamma(\Sigma;H'-H)\|\lesssim 
\|\gamma\|_{B^{2}_{\infty,1}}\|H\|\|H'-H\|\lesssim \frac{\log^2(2/\delta)}{\delta}\|H\|\|H'-H\| 
\end{align*}
that for some constant $C_1>0$ and for all $l=1,\dots, 9$
\begin{align*}
|A_l|\leq
C_1^{k^2}\frac{\log^2 (2/\delta)}{\delta}
\|Dg\|_{C^{k+1+\beta}}
((\|\Sigma\|+\|H\|)^{k+1/2} \vee 1)
\prod_{i=1}^k \delta_i(1+\delta_i)^{2k+5}
(\|H\|\vee \|H\|^{\beta})\|H'-H\|.
\end{align*}
It then follows from representation (\ref{partial_diff_1}) that 
\begin{align}
\label{bd_phi}
&
\biggl\|\frac{\partial^{k+1} \phi (t_1,\dots, t_k,1,0)}{\partial t_1\dots \partial t_k \partial s_2}
 -
 \frac{\partial^{k+1} \phi (t_1,\dots, t_k,0,0)}{\partial t_1\dots \partial t_k \partial s_2}\biggr\|
\\
&
\nonumber
\leq C^{k^2}\frac{\log^2 (2/\delta)}{\delta}
\|Dg\|_{C^{k+1+\beta}}
((\|\Sigma\|+\|H\|)^{k+1/2} \vee 1)\prod_{i=1}^k \delta_i(1+\delta_i)^{2k+5}
(\|H\|\vee \|H\|^{\beta})\|H'-H\|
\end{align}
with some constant $C>0.$ Similarly to (\ref{bd_prod_delta}), we 
have that for $k\leq d$ with some constant $C_2\geq 1$
\begin{align*}
{\mathbb E}\prod_{i=1}^k \delta_i(1+\delta_i)^{2k+5} 
\leq 
C_2^{k^2}\biggl(\frac{d}{n}\biggr)^{k/2}.
\end{align*}
Using this together with (\ref{bd_phi}) to bound on the expectation in (\ref{D_B_k})
yields bound (\ref{remainder_A}).

\end{proof}

Theorem \ref{th:remainder_B_k} immediately follows from 
lemmas \ref{lem_rem_A} and \ref{lem_rem_B}.

\end{proof}

We will now derive a bound on the bias of estimator 
${\mathcal D}g_k(\hat \Sigma)$ that improves the bounds 
of Section \ref{Sec:bias-iter} under stronger assumptions 
on smoothness of $g.$

\begin{corollary}
\label{bias_better}
Suppose $g\in L_{\infty}^{O}({\mathcal C}_+({\mathbb H}))$ is $k+2$ times continuously differentiable for some $k\leq d\leq n$ and, for some $\beta\in (0,1],$ $\|Dg\|_{C^{k+1+\beta}}<\infty.$ In addition, 
suppose that for some $\delta>0$ $\sigma(\Sigma)\subset \Bigl[\delta, \frac{1}{\delta}\Bigr].$
Then, for some constant $C>0,$
\begin{align}
\label{bias_better_bound}
&
\|{\mathbb E}_{\Sigma}{\mathcal D}g_k(\hat \Sigma)- {\mathcal D}g(\Sigma)\|
\leq 
C^{k^2}\frac{\log^2 (2/\delta)}{\delta}\|Dg\|_{C^{k+1+\beta}}
(\|\Sigma\|\vee 1)^{k+3/2} \|\Sigma\|\biggl(\frac{d}{n}\biggr)^{(k+1+\beta)/2}.
\end{align}
\end{corollary}

\begin{proof}
First note that 
\begin{align*}
&
{\mathcal B}^{k+1} {\mathcal D}g(\Sigma)=
{\mathbb E}_{\Sigma}B_k(\hat \Sigma)-B_k(\Sigma)
\\
&
= {\mathbb E}_{\Sigma} DB_k(\Sigma;\hat \Sigma-\Sigma)
+{\mathbb E}_{\Sigma} S_{B_k}(\Sigma;\hat \Sigma-\Sigma)
= {\mathbb E}_{\Sigma} S_{B_k}(\Sigma;\hat \Sigma-\Sigma).
\end{align*}
It follows from bound \eqref{remainder_B_k} (with $H'=\hat \Sigma-\Sigma$
and $H=0$) that 
\begin{align}
\label{remainder_B_k_hat}
&
\|S_{B_k}(\Sigma;\hat \Sigma-\Sigma)\|
\leq
C^{k^2}\frac{\log^2 (2/\delta)}{\delta}\|Dg\|_{C^{k+1+\beta}}
\biggl(\frac{d}{n}\biggr)^{k/2}
\gamma_{\beta,k}(\Sigma;\|\hat \Sigma-\Sigma\|)\|\hat \Sigma-\Sigma\|.
\end{align}
Since 
\begin{align*}
&
\gamma_{\beta,k}(\Sigma;\|\hat \Sigma-\Sigma\|)\|\hat \Sigma-\Sigma\|
\\
&
\leq (\|\Sigma\|\vee 1)^{k+1/2}(\|\hat \Sigma-\Sigma\|^{1+\beta}+ \|\hat \Sigma-\Sigma\|^2) + \|\hat \Sigma-\Sigma\|^{k+\beta + 3/2} + \|\hat \Sigma-\Sigma\|^{k+5/2},
\end{align*}
we can use the bound ${\mathbb E}^{1/p}\|\hat \Sigma-\Sigma\|^p\lesssim \|\Sigma\|\Bigl(\sqrt{\frac{d}{n}}\vee \sqrt{\frac{p}{n}}\Bigr)$ to get that for some constant $C_1>0$ and for $k\leq d\leq n$
\begin{align*}
{\mathbb E}
\gamma_{\beta,k}(\Sigma;\|\hat \Sigma-\Sigma\|)\|\hat \Sigma-\Sigma\|
\leq 
C_1^k (\|\Sigma\|\vee 1)^{k+3/2} \|\Sigma\|\biggl(\frac{d}{n}\biggr)^{(1+\beta)/2}.
\end{align*}
Therefore, for some constant $C>0,$
\begin{align*}
&
\|{\mathcal B}^{k+1} {\mathcal D}g(\Sigma)\|
\leq {\mathbb E}\|S_{B_k}(\Sigma;\hat \Sigma-\Sigma)\|
\leq C^{k^2}\frac{\log^2 (2/\delta)}{\delta}\|Dg\|_{C^{k+1+\beta}}
(\|\Sigma\|\vee 1)^{k+3/2} \|\Sigma\|\biggl(\frac{d}{n}\biggr)^{(k+1+\beta)/2}.
\end{align*}
Since ${\mathbb E}_{\Sigma}{\mathcal D}g_k(\hat \Sigma)- {\mathcal D}g(\Sigma)=
(-1)^k{\mathcal B}^{k+1} {\mathcal D}g(\Sigma),$
the result follows.

\end{proof}

\section{Normal approximation bounds  for estimators with reduced bias}
\label{Sec:norm_appr}

In this section, our goal is to prove bounds showing that, for sufficiently smooth 
orthogonally invariant functions $g,$ for large enough $k$ and
for an operator $B$ with nuclear norm bounded 
by a constant,
the distribution 
of random variables 
$$
\frac{\sqrt{n}\Bigl(\langle{\mathcal D}g_k(\hat \Sigma),B\rangle - \langle {\mathcal D}g(\Sigma), B\rangle\Bigr)}
{\sigma_{g}(\Sigma;B)}
$$
is close to the standard normal distribution when $n\to \infty$ and $d=o(n).$
It will be shown that this holds true with  
\begin{equation}
\label{variance_conjecture}
\sigma_g^2(\Sigma;B)= 2\Bigl\|\Sigma^{1/2}(D{\mathcal D}g(\Sigma))^{\ast}B\Sigma^{1/2}\Bigr\|_2^2,
\end{equation}
where $(D{\mathcal D}g(\Sigma))^{\ast}$ is the adjoint operator of $D{\mathcal D}g(\Sigma):$
$
\langle D {\mathcal D}g(\Sigma)H_1, H_2\rangle = \langle H_1, (D{\mathcal D}g(\Sigma))^{\ast}H_2\rangle.
$

We will prove the following result.

\begin{theorem}
\label{th-main-inv}
Suppose that, for some $s>0,$
$g\in C^{s+1} ({\mathcal B}_{sa}({\mathbb H}))\cap L_{\infty}^O({\mathcal C}_+({\mathbb H}))$ is an orthogonally invariant function. 
Suppose that 
$d\geq 3\log n$ and, for some $\alpha \in (0,1),$
$d\leq n^{\alpha}.$
Suppose also that $\Sigma$ is non-singular and, for a small enough constant $c>0,$ 
\begin{equation}
\label{assump_on_d}
d\leq \frac{c n}{(\|\Sigma\|\vee \|\Sigma^{-1}\|)^4}.
\end{equation}
Finally, suppose that $s> \frac{1}{1-\alpha}$
and  let $k$ be an integer number such that $\frac{1}{1-\alpha}<k+1+\beta \leq s$
for some $\beta \in (0,1].$ 
Then, there exists a constant $C$ such that  
\begin{align}
\label{Ber-Ess_finis}
&
\nonumber
\sup_{x\in {\mathbb R}}\biggl|{\mathbb P}\biggl\{
\frac{\sqrt{n}\Bigl(\langle{\mathcal D}g_k(\hat \Sigma),B\rangle - \langle {\mathcal D}g(\Sigma), B\rangle\Bigr)}
{\sigma_{g}(\Sigma;B)}\leq x\biggr\}-\Phi(x)\biggr| 
\\
&
\leq 
C^{k^2} L_g(B;\Sigma)
\Bigl[n^{-\frac{k+\beta-\alpha(k+1+\beta)}{2}}
+ n^{-(1-\alpha)\beta/2}\sqrt{\log n}\Bigr]
+\frac{C}{\sqrt{n}},
\end{align}
where 
$$
L_g(B;\Sigma):= \frac{\|B\|_1\|Dg\|_{C^{s}}}{\sigma_{g}(\Sigma;B)}
(\|\Sigma\|\vee \|\Sigma^{-1}\|)\log^2 (2(\|\Sigma\|\vee \|\Sigma^{-1}\|))\|\Sigma\|(\|\Sigma\|\vee 1)^{k+3/2}.
$$
\end{theorem}

We will also need the following exponential  
upper bound on the r.v. 
$
\frac{\sqrt{n}\Bigl(\langle{\mathcal D}g_k(\hat \Sigma),B\rangle - \langle {\mathcal D}g(\Sigma), B\rangle\Bigr)}
{\sigma_{g}(\Sigma;B)}.
$

\begin{proposition}
\label{main-inv_exp}
Under the assumptions of Theorem \ref{th-main-inv}, 
there exists a constant $C$ such that, for all $t\geq 1$
with probability at least $1-e^{-t},$
\begin{equation}
\label{main-inv_exp_A}
\biggl|\frac{\sqrt{n}\Bigl(\langle{\mathcal D}g_k(\hat \Sigma),B\rangle - \langle {\mathcal D}g(\Sigma), B\rangle\Bigr)}
{\sigma_{g}(\Sigma;B)}\biggr| 
\leq C^{k^2} (L_g(B;\Sigma)\vee 1)\sqrt{t}.
\end{equation}
\end{proposition}

Our main application is to the problem 
of estimation of  the functional $\langle f(\Sigma), B\rangle$
for a given smooth function $f$ and given operator $B.$
We will use $\langle f_k(\hat \Sigma),B\rangle$ as its 
estimator, where  
$
f_k (\Sigma):= \sum_{j=0}^k (-1)^j {\mathcal B}^j f(\Sigma).
$
Denote 
$$
\sigma_f^2(\Sigma;B)= 2\Bigl\|\Sigma^{1/2}Df(\Sigma;B)\Sigma^{1/2}\Bigr\|_2^2.
$$

\begin{theorem}
\label{th-main-function}
Suppose that, for some $s>0,$ 
$
f \in B^{s}_{\infty,1}({\mathbb R}).
$
Suppose that 
$d\geq 3\log n$ and, for some $\alpha \in (0,1),$
$d\leq n^{\alpha}.$ 
Suppose also that $\Sigma$ is non-singular and, for a small enough constant $c=c_s>0,$ 
\begin{equation}
\label{assump_on_d_s}
d\leq \frac{c n}{(\|\Sigma\|\vee \|\Sigma^{-1}\|)^4}.
\end{equation}
Finally, suppose that $s> \frac{1}{1-\alpha}$
and let $k$ be an integer number such that $\frac{1}{1-\alpha}<
k+1+\beta\leq s$ for some $\beta\in (0,1].$
Then, there exists a constant $C$ such that  
\begin{align}
\label{Ber-Ess_finis_a}
&
\nonumber
\sup_{x\in {\mathbb R}}\biggl|{\mathbb P}\biggl\{
\frac{\sqrt{n}\Bigl(\langle f_k(\hat \Sigma),B\rangle - 
\langle f(\Sigma), B\rangle\Bigr)}
{\sigma_{f}(\Sigma;B)}\leq x\biggr\}-\Phi(x)\biggr| 
\\
&
\leq 
C^{k^2} M_f(B;\Sigma)
\Bigl[n^{-\frac{k+\beta-\alpha (k+1+\beta)}{2}}
+ n^{-(1-\alpha)\beta/2}\sqrt{\log n}\Bigr]
+\frac{C}{\sqrt{n}}, 
\end{align}
where 
$$
M_f(B;\Sigma):=
\frac{\|B\|_1 \|f\|_{B_{\infty,1}^{s}}}{\sigma_{f}(\Sigma;B)}
(\|\Sigma\|\vee \|\Sigma^{-1}\|)^{2+s}\log^2 (2(\|\Sigma\|\vee \|\Sigma^{-1}\|))\|\Sigma\|(\|\Sigma\|\vee 1)^{k+3/2}.
$$
\end{theorem}

\begin{proposition}
\label{main-inv_function_exp}
Under the assumptions of Theorem \ref{th-main-function}, 
there exists a constant $C$ such that, for all $t\geq 1$
with probability at least $1-e^{-t},$
\begin{equation}
\label{main-inv_exp_function_A}
\biggl|\frac{\sqrt{n}\Bigl(\langle f_k(\hat \Sigma),B\rangle - \langle f(\Sigma), B\rangle\Bigr)}
{\sigma_{f}(\Sigma;B)}\biggr| 
\leq C^{k^2}(M_f(B;\Sigma)\vee 1)\sqrt{t}.
\end{equation}
\end{proposition}

We now turn to the proof of Theorem \ref{th-main-inv} and 
Proposition \ref{main-inv_exp}.

\begin{proof}
Recall the notation $D_k(\Sigma):={\mathcal D}g_k(\Sigma).$
For a given operator $B,$ define functionals
$
{\frak d}_k(\Sigma):= \langle D_k(\Sigma), B\rangle,
$
recall that 
$$
{\frak d}_k(\hat \Sigma)- {\mathbb E}{\frak d}_k(\hat \Sigma)
= \langle D{\frak d}_k(\Sigma), \hat \Sigma-\Sigma\rangle 
+ S_{{\frak d}_k}(\Sigma; \hat \Sigma-\Sigma)- {\mathbb E}S_{{\frak d}_k}(\Sigma; \hat \Sigma-\Sigma)
$$
and consider the following representation:
\begin{equation}
\label{repres_repres}
\frac{\sqrt{n}\Bigl(\langle{\mathcal D}g_k(\hat \Sigma),B\rangle - \langle {\mathcal D}g(\Sigma), B\rangle\Bigr)}
{\sigma_{g}(\Sigma;B)}=
\frac
{\sqrt{n}\langle D{\frak d}_k(\Sigma), \hat \Sigma-\Sigma\rangle}
{\sqrt{2}\|{\mathcal D}{\frak d}_k(\Sigma)\|_2} + \zeta,
\end{equation}
with the remainder $\zeta:= \zeta_1+\zeta_2+\zeta_3,$ where
\begin{align*}
&
\zeta_1:= 
\frac{ \sqrt{n}(\langle {{\mathbb E}\mathcal D} g_k(\hat \Sigma)- {\mathcal D}g(\Sigma),B\rangle)}
{\sigma_{g}(\Sigma;B)},
\\
&
\zeta_2:=\frac{\sqrt{n}(S_{{\frak d}_k}(\Sigma; \hat \Sigma-\Sigma)- {\mathbb E}S_{{\frak d}_k}(\Sigma; \hat \Sigma-\Sigma))}{\sigma_{g}(\Sigma;B)},
\\
&
\zeta_3:=\frac{\sqrt{n}\langle D{\frak d}_k(\Sigma), \hat \Sigma-\Sigma\rangle}
{\sqrt{2}\|{\mathcal D}{\frak d}_k(\Sigma)\|_2} 
\frac{\sqrt{2}\|{\mathcal D}{\frak d}_k(\Sigma)\|_2-\sigma_g(\Sigma;B)}{\sigma_g(\Sigma;B)}.
\end{align*}

{\it Step 1}. By Lemma \ref{Berry-Esseen}, 
\begin{align}
\label{Ber-Ess_d_k}
&
\sup_{x\in {\mathbb R}}\biggl|{\mathbb P}
\biggl\{\frac{\sqrt{n}\langle D{\frak d}_k(\Sigma), \hat \Sigma-\Sigma\rangle}
{\sqrt{2}\|{\mathcal D}{\frak d}_k(\Sigma)\|_2}\leq x\biggr\}-\Phi(x)\biggr|
\lesssim  \biggl(\frac{\|{\mathcal D} {\frak d}_k(\Sigma)\|_3}
{\|\mathcal D {\frak d}_k(\Sigma)\|_2}\biggr)^3\frac{1}{\sqrt{n}}
\lesssim \frac{\|{\mathcal D} {\frak d}_k(\Sigma)\|}{\|\mathcal D {\frak d}_k(\Sigma)\|_2}
\frac{1}{\sqrt{n}}\lesssim \frac{1}{\sqrt{n}}.
\end{align}
Note also that 
$$
\frac{\sqrt{n}\langle D{\frak d}_k(\Sigma), \hat \Sigma-\Sigma\rangle}
{\sqrt{2}\|{\mathcal D}{\frak d}_k(\Sigma)\|_2} 
\stackrel{d}{=}
\frac{\sum_{j=1}^n\sum_{i\geq 1} \lambda_i (Z_{i,j}^2-1)}
{\sqrt{2n}\biggl(\sum_{i\geq 1}\lambda_i^2\biggr)^{1/2}},
$$
where $\{Z_{i,j}\}$ are i.i.d. standard normal random variables
and $\{\lambda_i\}$ are the eigenvalues of ${\mathcal D}{\frak d}_k(\Sigma)$
(see the proof of Lemma \ref{Berry-Esseen}). To provide an upper bound 
on the right hand side, we use Lemma \ref{cite_Vershynin} to get that with probability at least $1-e^{-t}$
\begin{equation}
\label{expo_X}
\biggl|\frac{\sqrt{n}\langle D{\frak d}_k(\Sigma), \hat \Sigma-\Sigma\rangle}
{\sqrt{2}\|{\mathcal D}{\frak d}_k(\Sigma)\|_2}\biggr| \lesssim \sqrt{t}\vee \frac{t}{\sqrt{n}}.
\end{equation}

We will now control separately each of the random variables $\zeta_{1}, \zeta_{2},
\zeta_{3}.$ 

{\it Step 2}. To bound $\zeta_{1},$ we observe that, for $\delta = \frac{1}{\|\Sigma\|\vee \|\Sigma^{-1}\|},$ $\sigma(\Sigma)\subset \Bigl[\delta,\frac{1}{\delta}\Bigr]$ and 
use inequality \eqref{bias_better_bound} that yields:
\begin{align}
&
\nonumber
|\zeta_{1}| \leq 
\frac{ \sqrt{n}\|{\mathbb E}{\mathcal D} g_k(\hat \Sigma)- {\mathcal D}g(\Sigma)\|\|B\|_1}{\sigma_{g}(\Sigma;B)}
\\
&
\leq 
C^{k^2}\Lambda_{k,\beta}(g;\Sigma;B)
(\|\Sigma\|\vee 1)^{k+3/2} \|\Sigma\|\sqrt{n}\biggl(\frac{d}{n}\biggr)^{(k+1+\beta)/2},
\end{align}
where
$$
\Lambda_{k,\beta}(g;\Sigma;B):=\frac{\|B\|_1\|Dg\|_{C^{k+1+\beta}}}{\sigma_{g}(\Sigma;B)}
(\|\Sigma\|\vee \|\Sigma^{-1}\|)\log^2 (2(\|\Sigma\|\vee \|\Sigma^{-1}\|)).
$$
Under the assumption that, for some $\alpha \in (0,1),$ $d\leq n^{\alpha},$
the last bound implies that 
\begin{align}
\label{zeta1}
&
|\zeta_{1}|
\leq 
C^{k^2}\Lambda_{k,\beta}(g;\Sigma;B)
(\|\Sigma\|\vee 1)^{k+3/2} \|\Sigma\| n^{-\frac{k+\beta-\alpha(k+1+\beta)}{2}},
\end{align}
which tends to $0$ for $k+1+\beta>\frac{1}{1-\alpha}.$

{\it Step 3}. To bound $\zeta_{2},$ recall Theorem \ref{conc-med-gen} and 
Corollary \ref{cor:rem_D_k}. It follows from these statements that, under 
the assumptions $\|g\|_{C^{k+2+\beta}}<\infty$ and $d\leq n/2,$ for all 
$t\geq 1$ with probability at least $1-e^{-t},$
\begin{align}
\label{conc_S_d_k}
&
\nonumber
|\zeta_{2}|\leq 
\frac{\sqrt{n}|S_{{\frak d}_k}(\Sigma;\hat \Sigma-\Sigma)-{\mathbb E} S_{{\frak d}_k}(\Sigma;\hat \Sigma-\Sigma)|}{\sigma_{g}(\Sigma;B)}
\\
&
\leq C^{k^2}\Lambda_{k,\beta}(g;\Sigma;B)
\gamma_{\beta,k}(\Sigma;\delta_n(\Sigma;t))
\Bigl(\sqrt{\|\Sigma\|}+\sqrt{\delta_n(\Sigma;t)}\Bigr)
\sqrt{\|\Sigma\|}\sqrt{t},
\end{align}
where 
\begin{equation}
\nonumber
\delta_n(\Sigma;t):= \|\Sigma\|\biggl(\sqrt{\frac{{\bf r}(\Sigma)}{n}}\bigvee \frac{{\bf r}(\Sigma)}{n} \bigvee \sqrt{\frac{t}{n}}\bigvee \frac{t}{n}\biggr)
\leq 
\|\Sigma\|\biggl(\sqrt{\frac{d}{n}}\bigvee \sqrt{\frac{t}{n}}\bigvee \frac{t}{n}\biggr)
=: \bar \delta_n(\Sigma;t).
\end{equation}
Recall that 
$
\gamma_{\beta,k}(\Sigma;u)= (\|\Sigma\|\vee u\vee 1)^{k+1/2}(u\vee u^{\beta}), u>0.
$
For $d\leq n$ and $t\leq n,$ we have $\delta_n(\Sigma;t)\leq \|\Sigma\|$
and 
$
\gamma_{\beta,k}(\Sigma;\delta_n(\Sigma;t))\leq (\|\Sigma\|\vee 1)^{k+3/2},
$
which implies that, for some $C>1$ and for all 
$t\in [1,n]$ with probability at least $1-e^{-t},$
\begin{align}
\label{conc_S_d_k_A}
&
|\zeta_{2}|
\leq C^{k^2}\Lambda_{k,\beta}(g;\Sigma;B)\|\Sigma\|(\|\Sigma\|\vee 1)^{k+3/2}\sqrt{t}.
\end{align}

Let now $t=3\log n.$ For $d\geq 3\log n, d\leq n,$
we have $\bar \delta_n(\Sigma;t)\leq \|\Sigma\|\sqrt{\frac{d}{n}}\leq \|\Sigma\|$
and 
$$
\gamma_{\beta,k}(\Sigma;\delta_n(\Sigma;t))\leq 
\gamma_{\beta,k}(\Sigma;\bar \delta_n(\Sigma;t))
\leq (\|\Sigma\|\vee 1)^{k+3/2}\biggl(\frac{d}{n}\biggr)^{\beta/2}.
$$
In addition, 
$$
\Bigl(\sqrt{\|\Sigma\|}+\sqrt{\delta_n(\Sigma;t)}\Bigr)
\sqrt{\|\Sigma\|}\sqrt{t}
\lesssim \|\Sigma\|\sqrt{\log n}.
$$
Thus, for $d\geq 3\log n, d\leq n^{\alpha},$ it follows from 
\eqref{conc_S_d_k} that 
with some constant $C\geq 1$ and with probability at least $1-n^{-3},$
\begin{align}
\label{zeta2}
&
\nonumber
|\zeta_{2}|\leq 
C^{k^2}
\Lambda_{k,\beta}(g;\Sigma;B)
\|\Sigma\|(\|\Sigma\|\vee 1)^{k+3/2} \biggl(\frac{d}{n}\biggr)^{\beta/2}
\sqrt{\log n}
\\
&
\leq C^{k^2}
\Lambda_{k,\beta}(g;\Sigma;B)
\|\Sigma\|(\|\Sigma\|\vee 1)^{k+3/2} n^{-(1-\alpha)\beta/2}
\sqrt{\log n}.
\end{align}

{\it Step 4}. Finally, we need to bound $\zeta_{3}.$ To this end, denote 
${\frak b}_k (\Sigma):= \langle B_k(\Sigma), B\rangle.$ Then,
${\frak b}_0(\Sigma)= \langle {\mathcal D}g(\Sigma), B\rangle$
and
$
{\frak d}_k (\Sigma)= \sum_{j=0}^k (-1)^j {\frak b}_j(\Sigma).
$
Observe that 
$$
\langle D {\frak b}_j (\Sigma), H\rangle
=D{\frak b}_j(\Sigma;H) 
= \langle DB_j(\Sigma)H, B\rangle 
=\langle H, (DB_j(\Sigma))^{\ast}B\rangle,
$$
implying $D {\frak b}_j (\Sigma)=(DB_j(\Sigma))^{\ast}B.$
Therefore, we have 
$$
\|D {\frak b}_j (\Sigma)\|_2 = 
\sup_{\|H\|_2\leq 1}
|\langle DB_j(\Sigma)H, B\rangle| \leq  \sup_{\|H\|\leq 1}
|\langle DB_j(\Sigma)H, B\rangle| \leq 
\|B\|_1\sup_{\|H\|\leq 1}\|DB_j(\Sigma)H\|. 
$$
To bound the right hand side we use Lemma \ref{lem_rem_C}
that yields
$$
\sup_{\|H\|\leq 1}\|DB_j(\Sigma)H\|
\leq C^{j^2}
\max_{1\leq j\leq j+2}\|D^{j}g\|_{L_{\infty}}
(\|\Sigma\|^{j+1/2} \vee 1)\biggl(\frac{d}{n}\biggr)^{j/2}.
$$
Therefore, for all $j=1,\dots, k,$
$$
\|D {\frak b}_j (\Sigma)\|_2 \leq 
C^{k^2}\|B\|_1
\max_{1\leq j\leq k+2}\|D^{j}g\|_{L_{\infty}}
(\|\Sigma\|^{k+1/2} \vee 1)\biggl(\frac{d}{n}\biggr)^{j/2}
$$
and 
$$
\|{\mathcal D}{\frak b}_j (\Sigma)\|_2=
\|\Sigma^{1/2}D {\frak b}_j (\Sigma)\Sigma^{1/2}\|_2 
\leq 
\|\Sigma\|\|D {\frak b}_j (\Sigma)\|_2
$$
$$
\leq 
C^{k^2}\|B\|_1
\max_{1\leq j\leq k+2}\|D^{j}g\|_{L_{\infty}}
(\|\Sigma\|^{k+3/2} \vee \|\Sigma\|)\biggl(\frac{d}{n}\biggr)^{j/2}.
$$
Since also 
$$
\sqrt{2}\|{\mathcal D}{\frak b}_0 (\Sigma)\|_2= \sqrt{2}\|\Sigma^{1/2}
(D{\mathcal D}g(\Sigma))^{\ast} B\Sigma^{1/2}\|_2=
\sigma_g(\Sigma;B),
$$
we get 
$$
\Bigl|\sqrt{2}\|{\mathcal D}{\frak d}_k(\Sigma)\|_2-\sigma_g(\Sigma;B)\Bigr|
\leq \sqrt{2}\sum_{j=1}^k \|{\mathcal D}{\frak b}_j (\Sigma)\|_2
$$
$$
\leq \sqrt{2}C^{k^2}\|B\|_1
\max_{1\leq j\leq k+2}\|D^{j}g\|_{L_{\infty}}
(\|\Sigma\|^{k+3/2} \vee \|\Sigma\|)\sum_{j=1}^k\biggl(\frac{d}{n}\biggr)^{j/2},
$$
implying that, under the assumption $d\leq n/4,$
\begin{align}
\label{ratio_X}
&
\biggl|\frac{\sqrt{2}\|{\mathcal D}{\frak d}_k(\Sigma)\|_2-\sigma_g(\Sigma;B)}{\sigma_g(\Sigma;B)}\biggr|
\leq \frac{2\sqrt{2}C^{k^2}\|B\|_1}{\sigma_g(\Sigma;B)}
\max_{1\leq j\leq k+2}\|D^{j}g\|_{L_{\infty}}
(\|\Sigma\|^{k+3/2} \vee \|\Sigma\|)\sqrt{\frac{d}{n}}.
\end{align}
It follows from (\ref{ratio_X}) and (\ref{expo_X}) that 
with some $C>1$ and with probability at least $1-e^{-t}$
\begin{equation}
\label{zeta3_A}
|\zeta_{3}|\leq
\frac{C^{k^2}\|B\|_1}{\sigma_g(\Sigma;B)}
\max_{1\leq j\leq k+2}\|D^{j}g\|_{L_{\infty}}
\|\Sigma\|(\|\Sigma\|\vee 1)^{k+1/2}
\sqrt{\frac{d}{n}}\biggl(\sqrt{t}\vee \frac{t}{\sqrt{n}}\biggr).
\end{equation}
For $d\geq 3\log n, d\leq n^{\alpha}$ and $t=3\log n,$ this yields
\begin{equation}
\label{zeta3}
|\zeta_{3}|\leq 
\frac{C^{k^2}}{\sigma_g(\Sigma;B)}
\max_{1\leq j\leq k+2}\|D^{j}g\|_{L_{\infty}}
\|\Sigma\|(\|\Sigma\|\vee 1)^{k+1/2}
n^{-(1-\alpha)/2}\sqrt{\log n}
\end{equation}
that holds for some $C\geq 1$ with probability at least $1-n^{-3}.$

{\it Step 5}. It follows from bounds (\ref{zeta1}), (\ref{zeta2}) and (\ref{zeta3}) 
that, for some $C\geq 1$ with probability at least $1-2 n^{-3},$
\begin{align}
&
\nonumber
|\zeta|\leq 
C^{k^2}
\Lambda_{k,\beta}(g;\Sigma;B)
\|\Sigma\|(\|\Sigma\|\vee 1)^{k+3/2}
\\
&
\nonumber
\Bigl[n^{-\frac{k+\beta-\alpha(k+1+\beta)}{2}}
+ n^{-(1-\alpha)\beta/2}
\sqrt{\log n}+n^{-(1-\alpha)/2}\sqrt{\log n} 
\Bigr],
\end{align}
which implies that with the same probability and with a possibly different $C\geq 1$
\begin{align}
\label{zeta}
&
\nonumber
|\zeta|\leq
C^{k^2} L_g(B;\Sigma)
\Bigl[n^{-\frac{k+\beta-\alpha(k+1+\beta)}{2}}
+ n^{-(1-\alpha)\beta/2}\sqrt{\log n}
\Bigr].
\end{align}
It follows from the last bound that 
\begin{align}
&
\nonumber
\delta(\xi,\eta)\leq 2 n^{-3}
+ 
C^{k^2} L_g(B;\Sigma)
\Bigl[n^{-\frac{k+\beta-\alpha(k+1+\beta)}{2}}
+ n^{-(1-\alpha)\beta/2}\sqrt{\log n}
\Bigr],
\end{align}
where 
$$
\xi:=\frac{\sqrt{n}\Bigl(\langle{\mathcal D}g_k(\hat \Sigma),B\rangle - \langle {\mathcal D}g(\Sigma), B\rangle\Bigr)}
{\sigma_{g}(\Sigma;B)},
\ \ 
\eta:= \frac
{\sqrt{n}\langle D{\frak d}_k(\Sigma), \hat \Sigma-\Sigma\rangle}
{\sqrt{2}\|{\mathcal D}{\frak d}_k(\Sigma)\|_2}, 
$$
$\xi-\eta=\zeta$ and $\delta(\xi,\eta)$ is defined in Lemma \ref{xi_eta}.
It follows from bound (\ref{Ber-Ess_d_k}) and Lemma \ref{xi_eta}
that, for some $C\geq 1,$ bound (\ref{Ber-Ess_finis}) holds.

{\it Step 6}. It remains to prove Proposition \ref{main-inv_exp}.
When $t\in [1,n],$ bound \eqref{main-inv_exp_A} immediately follows from \eqref{repres_repres},
\eqref{expo_X}, \eqref{zeta1}, \eqref{conc_S_d_k_A} and \eqref{zeta3_A}.
To prove it for $t>n,$ first observe that 
\begin{equation}
\label{bd_DG_A}
|\langle {\mathcal D}g(\Sigma),B\rangle|
\leq \|\Sigma^{1/2}Dg(\Sigma)\Sigma^{1/2}\| \|B\|_1
\leq \|Dg\|_{L_{\infty}}\|\Sigma\| \|B\|_1.
\end{equation}
We will also prove that for some constant $C>1$ 
\begin{equation}
\label{bd_DG_B}
|\langle {\mathcal D}g_k(\Sigma),B\rangle|
\leq C^{k}\|Dg\|_{L_{\infty}}\|\Sigma\| \|B\|_1.
\end{equation}
To this end, note that, by \eqref{B_kDB_k}, 
\begin{align}
&
\nonumber
\|{\mathcal D} {\mathcal B}^k g(\Sigma)\|
\leq 
{\mathbb E}\sum_{I\subset \{1,\dots, k\}}
\|\Sigma^{1/2}A_{I}Dg(A_I^{\ast} \Sigma A_I)A_I^{\ast}\Sigma^{1/2}\|
\\
&
\nonumber
\leq 
\sum_{I\subset \{1,\dots, k\}}
\|\Sigma\|\|Dg\|_{L_{\infty}} {\mathbb E}\|A_I\|^2
\leq \|\Sigma\|\|Dg\|_{L_{\infty}}
\sum_{I\subset \{1,\dots, k\}}
{\mathbb E}\prod_{i\in I}\|W_i\|
\\
&
\leq \|\Sigma\|\|Dg\|_{L_{\infty}}\sum_{j=0}^k {k\choose j} ({\mathbb E}\|W\|)^j
= \|\Sigma\|\|Dg\|_{L_{\infty}}(1+{\mathbb E}\|W\|)^k.
\end{align}
For $d\leq n,$ we have ${\mathbb E}\|W-I\|\lesssim \sqrt{\frac{d}{n}}\leq C'$
for some $C'>0.$ Thus, 
$$1+{\mathbb E}\|W\|\leq 2+ {\mathbb E}\|W-I\|\leq 2+C'=:C.$$ 
Therefore, 
$
\|{\mathcal D} {\mathcal B}^k g(\Sigma)\|\leq C^k
\|Dg\|_{L_{\infty}}\|\Sigma\|.
$
In view of the definition of $g_k,$ this implies that \eqref{bd_DG_B} holds with some $C>1.$
It follows from \eqref{bd_DG_A} and \eqref{bd_DG_B} that, for some $C>1,$
\begin{equation}
\label{bd_t>n}
\biggl|\frac{\sqrt{n}\Bigl(\langle{\mathcal D}g_k(\hat \Sigma),B\rangle - \langle {\mathcal D}g(\Sigma), B\rangle\Bigr)}
{\sigma_{g}(\Sigma;B)}\biggr| 
\leq \frac{C^k\|B\|_1\|Dg\|_{L_{\infty}}\|\Sigma\|\sqrt{n}}{\sigma_g(\Sigma;B)}.
\end{equation}
For $t>n,$ the right hand side of bound \eqref{bd_t>n} is smaller than 
the right hand side of bound \eqref{main-inv_exp_A}. Thus, 
\eqref{main-inv_exp_A} holds for all $t\geq 1.$

\end{proof}

Next we prove Theorem \ref{th-main-function} and Proposition 
\ref{main-inv_function_exp}.

\begin{proof}
First suppose that, for some $\delta>0,$ $\sigma(\Sigma)\subset [2\delta,\infty).$
Let $\gamma_{\delta}(x)=\gamma (x/\delta),$ where $\gamma: {\mathbb R}\mapsto [0,1]$ is a nondecreasing $C^{\infty}$ function, $\gamma (x)=0, x\leq 1/2,$ $\gamma(x)=1, x\geq 1.$ Define $f_{\delta}(x)=f(x)\gamma_{\delta}(x), x\in {\mathbb R}.$
Then, $f(\Sigma)=f_{\delta}(\Sigma)$ which also implies that, for all $\Sigma$ 
with $\sigma(\Sigma)\subset [2\delta,\infty),$ $Df(\Sigma)=Df_{\delta}(\Sigma)$
and $\sigma_f(\Sigma;B)=\sigma_{f_{\delta}}(\Sigma;B).$

Let
$\varphi(x):= \int_0^x \frac{f_{\delta}(t)}{t} dt, x\geq 0$ and $\varphi (x)=0, x<0.$
Clearly, $f_{\delta}(x)= x\varphi^{\prime}(x), x\in {\mathbb R}.$
Let $g(C):= {\rm tr}(\varphi (C)), C\in {\mathcal B}_{sa}({\mathbb H}).$ Then,
clearly, $g$ is an orthogonally invariant function, 
$Dg(C)= \varphi^{\prime}(C), C\in {\mathcal B}_{sa}({\mathbb H})$ and 
$$
{\mathcal D}g(C)= C^{1/2}\varphi^{\prime}(C) C^{1/2}
=f_{\delta}(C), C\in {\mathcal C}_+({\mathbb H}).
$$ 
It is also easy to see that 
${\mathcal D}g_k(C)=(f_{\delta})_k(C), C\in {\mathcal C}_+({\mathbb H}).$
Using Corollary \ref{remark_diff} of Section \ref{Sec:Entire}, standard bounds for pointwise multipliers of functions in Besov spaces (\cite{Triebel}, Section 2.8.3) and characterization of Besov norms in terms of difference operators 
(\cite{Triebel}, Section 2.5.12), it is easy to check that 
\begin{align*}
&
\|Dg\|_{C^{s}} \leq 2^{k+1} \|\varphi^{\prime}\|_{B^{s}_{\infty,1}}
=2^{k+1} \biggl\|\frac{f(x)\gamma_{\delta}(x)}{x}\biggr\|_{B^{s}_{\infty,1}}
\lesssim 2^{k+1}
\biggl\|\frac{\gamma_{\delta}(x)}{x}\biggr\|_{B^{s}_{\infty,1}} \|f\|_{B^{s}_{\infty,1}}
\\
&
\lesssim 2^{k+1}
\frac{1}{\delta}\biggl\|\frac{\gamma(x/\delta)}{x/\delta}\biggr\|_{B^{s}_{\infty,1}} \|f\|_{B^{s}_{\infty,1}}
\lesssim 2^{k+1} (\delta^{-1-s}\vee \delta^{-1}) \|f\|_{B^{s}_{\infty,1}}.
\end{align*}
Denote 
$$
\eta:=\frac{\sqrt{n}\Bigl(\langle (f_{\delta})_k(\hat \Sigma),B\rangle - 
\langle f(\Sigma), B\rangle\Bigr)}
{\sigma_{f}(\Sigma;B)}.
$$
It follows from Theorem \ref{th-main-inv} that 
\begin{equation}
\label{bound_Del}
\Delta(\eta;Z)
\leq C^{k^2} M_{f_{\delta}}(B;\Sigma)
\Bigl[n^{-\frac{k+\beta-\alpha(k+1+\beta)}{2}}
+ n^{-(1-\alpha)\beta/2}\sqrt{\log n}\Bigr]
+\frac{C}{\sqrt{n}}
\end{equation}
with
$M_{f_{\delta}}(B;\Sigma)\lesssim 2^{k+1}M_{f,\delta}(B;\Sigma)$
and
$$
M_{f,\delta}(B;\Sigma):=\frac{\|B\|_1 (\delta^{-1-s}\vee \delta^{-1})\|f\|_{B_{\infty,1}^{s}}}{\sigma_{f}(\Sigma;B)}
(\|\Sigma\|\vee \|\Sigma^{-1}\|)\log^2 (2(\|\Sigma\|\vee \|\Sigma^{-1}\|))\|\Sigma\|(\|\Sigma\|\vee 1)^{k+3/2}.
$$

It will be shown that, under the assumption $\sigma(\Sigma)\subset [2\delta,\infty),$
the estimator $(f_{\delta})_k(\hat \Sigma)={\mathcal D}g_k(\hat \Sigma)$
 can be replaced by the estimator $f_k(\hat \Sigma).$ To this end, 
the following lemma will be proved. 

\begin{lemma}
Suppose, for some $\delta>0,$ $\sigma(\Sigma)\subset [2\delta,\infty)$ and, for a sufficiently 
large constant $C_1>1,$ 
\begin{equation}
\label{d_assump}
d\leq \frac{\log^2 (1+\delta/\|\Sigma\|)}{C_1^2(k+1)^2}n=:\bar d.
\end{equation}
Then, with probability 
at least $1-e^{-\bar d},$
\begin{equation}
\label{fkfkdelta}
\|f_k(\hat \Sigma)- (f_{\delta})_k (\hat \Sigma)\|
\leq (k^2 2^{k+1}+2) e^{-\bar d}\|f\|_{L_{\infty}}.
\end{equation}
\end{lemma}

\begin{proof}
Recall that, by (\ref{Bkrepr_1}),
$$
{\mathcal B}^{k} f(\Sigma) 
= {\mathbb E}_{\Sigma}\sum_{j=0}^k (-1)^{k-j}{k\choose j} f(\hat \Sigma^{(j)}),
$$
implying that 
$$
{\mathcal B}^{k} f(\hat \Sigma)-
{\mathcal B}^{k} f_{\delta}(\hat \Sigma) 
= {\mathbb E}_{\hat \Sigma}
\sum_{j=0}^k (-1)^{k-j}{k\choose j} \Bigl[f(\hat \Sigma^{(j+1)})-f_{\delta}(\hat \Sigma^{(j+1)})\Bigr].
$$
Note also that $f(\hat \Sigma^{(j+1)})=f_{\delta}(\hat \Sigma^{(j+1)})$ provided 
that $\sigma (\hat \Sigma^{(j+1)})\subset [\delta, \infty)$ 
(since $f(x)=f_{\delta}(x), x\geq \delta$).  
This easily implies the following bound:
\begin{equation}
\label{f_f_delta}
\|{\mathcal B}^{k} f(\hat \Sigma)-
{\mathcal B}^{k} f_{\delta}(\hat \Sigma)\| 
\leq 2^{k+1} \|f\|_{L_{\infty}} {\mathbb P}_{\hat \Sigma}
\Bigl\{
\exists j=1,\dots ,k+1: \sigma (\hat \Sigma^{(j)})\not\subset [\delta, \infty)\Bigr\}.
\end{equation}
To control the probability of the event 
$
G:= \Bigl\{
\exists j=1,\dots ,k+1: \sigma (\hat \Sigma^{(j)})\not\subset [\delta, \infty)\Bigr\},
$
consider the following event:
$$
E:= \Bigl\{ \|\hat \Sigma^{(j+1)}-\hat \Sigma^{(j)}\|< C_1 \|\hat \Sigma^{(j)}\|\sqrt{\frac{d}{n}}, j=1,\dots, k\Bigr\}.
$$
It follows from bound \eqref{operator_hatSigma_exp_dimension} (applied conditionally 
on $\hat \Sigma^{(j)}$) that,
for a proper choice of constant $C_1>0,$
\begin{align}
\label{prob_E}
&
{\mathbb P}_{\hat \Sigma}(E^c) \leq 
{\mathbb E}_{\hat \Sigma} \sum_{j=1}^k {\mathbb P}_{\hat \Sigma^{(j)}}
\Bigl\{
\|\hat \Sigma^{(j+1)}-\hat \Sigma^{(j)}\|\geq C_1 \|\hat \Sigma^{(j)}\|\sqrt{\frac{d}{n}}
\Bigr\} \leq ke^{-d}.
\end{align}
Note that, on the event $E,$
$
\|\hat \Sigma^{(j+1)}\| \leq \|\hat \Sigma^{(j)}\|\biggl(1+C_1\sqrt{\frac{d}{n}}\biggr), 
$
which implies by induction that 
$$
\|\hat \Sigma^{(j)}\|\leq \|\hat \Sigma\|\biggl(1+C_1\sqrt{\frac{d}{n}}\biggr)^{j-1},
j=1,\dots, k+1. 
$$
This also yields that, on the event $E,$ 
$$
\|\hat \Sigma^{(j)}-\hat \Sigma\|\leq 
\sum_{i=1}^{j-1}
\|\hat \Sigma^{(i+1)}-\hat \Sigma^{(i)}\| 
\leq 
\sum_{i=1}^{j-1}
\|\hat \Sigma^{(i)}\| C_1\sqrt{\frac{d}{n}}
$$
$$
\leq \|\hat \Sigma\|C_1\sqrt{\frac{d}{n}}
\sum_{i=1}^{j-1}\biggl(1+C_1\sqrt{\frac{d}{n}}\biggr)^{i-1}
\leq \|\hat \Sigma\| \biggl[\biggl(1+C_1\sqrt{\frac{d}{n}}\biggr)^{j-1}-1\biggr],
j=1,\dots, k+1.
$$
Consider also the event 
$
F := \Bigl\{\|\hat \Sigma -\Sigma\| \leq C_1\|\Sigma\|\sqrt{\frac{d}{n}}\Bigr\}
$
that holds with probability at least $1-e^{-d}$ with a proper choice of 
constant $C_1.$ On this event, $\|\hat \Sigma\|\leq \|\Sigma\|
\biggl(1+C_1\sqrt{\frac{d}{n}}\biggr).$ Therefore, on the event $E\cap F,$
$$
\|\hat \Sigma^{(j)}-\Sigma\| \leq 
\|\Sigma\| \biggl(1+C_1\sqrt{\frac{d}{n}}\biggr)\biggl[\biggl(1+C_1\sqrt{\frac{d}{n}}\biggr)^{j-1}-1\biggr] + \|\Sigma\|C_1\sqrt{\frac{d}{n}}
$$
$$
=\|\Sigma\| \biggl[\biggl(1+C_1\sqrt{\frac{d}{n}}\biggr)^{j}-1\biggr], j=1,\dots, k+1.
$$
Note that 
$$
\|\Sigma\|\biggl[\biggl(1+C_1\sqrt{\frac{d}{n}}\biggr)^{k+1}-1\bigg]
\leq \|\Sigma\|\biggl(\exp\biggl\{C_1(k+1)\sqrt{\frac{d}{n}}\biggr\}-1\biggr)
\leq \delta
$$
provided that condition \eqref{d_assump} holds.
Therefore, on the event $E\cap F,$ 
$\|\hat \Sigma^{(j)}-\Sigma\|\leq \delta, j=1,\dots, k+1.$ 
Since $\sigma(\Sigma)\subset [2\delta, \infty),$ this implies 
that $\sigma (\hat \Sigma^{(j)})\subset [\delta, \infty), j=1,\dots, k+1.$ In other
words, $E\cap F\subset G^c.$ 
Bound (\ref{f_f_delta}) implies that 
\begin{align*}
&
\|{\mathcal B}^{k} f(\hat \Sigma)-
{\mathcal B}^{k} f_{\delta}(\hat \Sigma)\| I_{F}
\leq 2^{k+1} \|f\|_{L_{\infty}} I_F {\mathbb E}_{\hat \Sigma}I_{G}
\\
&
\leq 2^{k+1} \|f\|_{L_{\infty}} I_F {\mathbb E}_{\hat \Sigma}I_{F^c\cup E^c}
\leq 2^{k+1} \|f\|_{L_{\infty}} I_F (I_{F^c}+{\mathbb E}_{\hat \Sigma}I_{E^c})
\\
&
=2^{k+1} \|f\|_{L_{\infty}} I_F{\mathbb P}_{\hat \Sigma}(E^c)
\leq k2^{k+1} e^{-d}\|f\|_{L_{\infty}} I_F.
\end{align*}
This proves that on the event $F$ of probability at least $1-e^{-d}$
$$
\|{\mathcal B}^{k} f(\hat \Sigma)-
{\mathcal B}^{k} f_{\delta}(\hat \Sigma)\| 
\leq 
k2^{k+1} e^{-d}\|f\|_{L_{\infty}}.
$$
Moreover, the same bound also holds for 
$\|{\mathcal B}^{j} f(\hat \Sigma)-{\mathcal B}^{j} f_{\delta}(\hat \Sigma)\|$
for all $j=1,\dots, k$ and the dimension $d$ in the above argument can be replaced
by an arbitrary upper bound $d'$ satisfying condition 
\eqref{d_assump} (in particular, by $d'=\bar d$). Thus, 
under condition \eqref{d_assump},
with probability at least $1-e^{-\bar d}$
$$
\|{\mathcal B}^{j} f(\hat \Sigma)-
{\mathcal B}^{j} f_{\delta}(\hat \Sigma)\| 
\leq 
j2^{j+1} e^{-\bar d}\|f\|_{L_{\infty}}, j=1,\dots, k
$$
and also 
$
\|f(\hat \Sigma)-f_{\delta}(\hat \Sigma)\|\leq 2 e^{-\bar d} \|f\|_{L_{\infty}}.
$
This immediately implies that, under the assumption $\sigma(\Sigma)\subset [2\delta,\infty)$ and condition \eqref{d_assump}, with probability 
at least $1-e^{-\bar d},$ bound \eqref{fkfkdelta} holds.

\end{proof}

Define 
$$
\xi:= \frac{\sqrt{n}\Bigl(\langle f_k(\hat \Sigma),B\rangle - 
\langle f(\Sigma), B\rangle\Bigr)}
{\sigma_{f}(\Sigma;B)}
$$
It follows from \eqref{fkfkdelta} that with probability at least $1-e^{-\bar d}$ 
$$
|\xi -\eta| \leq \frac{(k^2 2^{k+1}+2)\|f\|_{L_{\infty}}\|B\|_1}{\sigma_{f}(\Sigma;B)} \sqrt{n}e^{-\bar d}
$$
and we can conclude that, under conditions $d\geq 3 \log n$ and \eqref{d_assump}, the following bound holds for some $C>1:$ 
$$
\delta(\xi,\eta)\leq \frac{(k^2 2^{k+1}+2)\|f\|_{L_{\infty}}\|B\|_1}{\sigma_{f}(\Sigma;B)} \sqrt{n}e^{-d} +e^{-d}
\leq C^{k} \frac{\|f\|_{L_{\infty}}\|B\|_1}{\sigma_{f}(\Sigma;B)} n^{-2} + n^{-3}.
$$
Combining this with bound (\ref{bound_Del}) and using Lemma \ref{xi_eta}
yields that with some $C>1$
\begin{equation}
\label{kon-kon}
\Delta (\xi, Z)\leq 
C^{k^2} M_{f,\delta}(B;\Sigma)
\Bigl[n^{-\frac{k+\beta-\alpha(k+1+\beta)}{2}}
+ n^{-(1-\alpha)\beta/2}\sqrt{\log n}\Bigr]
+\frac{C}{\sqrt{n}} +
C^{k} \frac{\|f\|_{L_{\infty}}\|B\|_1}{\sigma_{f}(\Sigma;B)} n^{-2}.
\end{equation}
It remains to choose $\delta: = \frac{1}{2\|\Sigma^{-1}\|}$ (which implies that
$\sigma(\Sigma)\subset [2\delta,\infty)$). Since 
\begin{equation}
\label{onbard}
\frac{\log^2 (1+\delta/\|\Sigma\|)}{C_1^2(k+1)^2} \geq 
\frac{\log^2 (1+1/2(\|\Sigma\|\vee \|\Sigma^{-1}\|)^2)}{C_1^2 s^2}
\geq \frac{c_{1,s}}{(\|\Sigma\|\vee \|\Sigma^{-1}\|)^4} 
\end{equation}
for a sufficiently small constant $c_{1,s},$
condition \eqref{d_assump} follows from assumption \eqref{assump_on_d_s} on $d.$ 
The bound of Theorem \ref{th-main-function} immediately follows from 
(\ref{kon-kon}).

It remains to prove Proposition \ref{main-inv_function_exp}.
It follows from bound \eqref{fkfkdelta} that, for $t\in [1,\bar d]$ and $d\geq 3\log n,$
$d\leq \bar d,$
with probability at least $1-e^{-t},$ 
\begin{equation}
\label{leqbard}
\|f_k(\hat \Sigma)- (f_{\delta})_k (\hat \Sigma)\|
\leq (k^2 2^{k+1}+2) n^{-3}\|f\|_{L_{\infty}}
\leq (k^2 2^{k+1}+2) \|f\|_{L_{\infty}} \sqrt{\frac{t}{n}}. 
\end{equation}
Due to a trivial bound 
$
\|f_k(\hat \Sigma)- (f_{\delta})_k (\hat \Sigma)\|\leq 2^{k+1}\|f\|_{L_{\infty}}
$
and bound \eqref{onbard},
we get that, for $t\geq \bar d,$
\begin{equation}
\label{geqbard}
\|f_k(\hat \Sigma)- (f_{\delta})_k (\hat \Sigma)\|\leq 
\frac{2^{k+1}\|f\|_{L_{\infty}}}{\sqrt{\bar d}}
\sqrt{t}
\leq 
\frac{2^{k+1}\|f\|_{L_{\infty}}(\|\Sigma\|\vee \|\Sigma\|^{-1})^2}{\sqrt{c_{1,s}}}
\sqrt{\frac{t}{n}}.
\end{equation}
It follows from \eqref{leqbard} and \eqref{geqbard} that 
there exists a constant $C>1$ such that, for all $t\geq 1$ with probability 
at least $1-e^{-t},$
\begin{equation*}
\|f_k(\hat \Sigma)- (f_{\delta})_k (\hat \Sigma)\|
\leq 
C^k \|f\|_{L_{\infty}}(\|\Sigma\|\vee \|\Sigma\|^{-1})^2
\sqrt{\frac{t}{n}}.
\end{equation*}
This implies that, for some $C>1,$ with the same probability
\begin{equation}
\label{xi-eta_expon}
|\xi -\eta|\leq C^k\frac{\|B\|_1\|f\|_{L_{\infty}}(\|\Sigma\|\vee \|\Sigma\|^{-1})^2}{\sigma_f(\Sigma;B)}\sqrt{t}.
\end{equation}
Applying bound of Proposition \ref{main-inv_exp} to 
${\mathcal D}g_k = (f_{\delta})_k,$ we get that for some 
constant $C>1$ and for all $t\geq 1$ with probability at least 
$1-e^{-t}$
$
|\eta| \leq C^{k^2} M_{f,\delta}(B;\Sigma)\sqrt{t}.
$ 
Combining this with bound \eqref{xi-eta_expon} yields that for some 
$C>1$ and all $t\geq 1$ with probability at least $1-e^{-t}$ 
$
|\xi| \leq C^{k^2} M_{f,\delta}(B;\Sigma)
\sqrt{t},
$ 
which, taking into account that $\delta=\frac{1}{2\|\Sigma^{-1}\|},$ completes the proof of Proposition \ref{main-inv_function_exp}.

\end{proof}

Finally, we are ready to prove Theorem \ref{th-main-function-uniform} stated 
in the introduction. 

\begin{proof} 
If $d<3\log n,$ the claims of Theorem \ref{th-main-function-uniform}
easily follow from Corollary \ref{r_small}. 
Under the assumption $d\geq 3 \log n,$ 
the proof of \eqref{normal_approximation_uniform_reduced_bias}
immediately follows from the bound of Theorem \ref{th-main-function}
(it is enough to take the supremum over the class of covariances 
$S(d_n;a)\cap \{\sigma_f(\Sigma;B)\geq \sigma_0\}$
and over all the operators $B$ with $\|B\|_1\leq 1,$
and to pass to the limit as $n\to\infty$).

To prove \eqref{normal_approximation_loss}, 
we apply lemmas \ref{loss-bd} and \ref{Lemma:Bernstein-loss} to r.v. 
$\xi:= \xi(\Sigma):=\frac{\sqrt{n}(\langle f_k(\hat \Sigma),B\rangle - \langle f(\Sigma), B\rangle)}
{\sigma_{f}(\Sigma;B)}$ and $\eta:=Z.$
Using bounds \eqref{main-inv_exp_function_A} and \eqref{Bernstein-loss},
we get that 
$
{\mathbb E}\ell^2(\xi)\leq 2e \sqrt{2\pi}c_1^2 e^{2c_2^2 \tau^2},
$
where $\tau:= 2C^{k^2} M_f(B;\Sigma).$ 
Using bounds \eqref{ellxietaell}, \eqref{Bernstein-loss}, easy bounds on
${\mathbb E}\ell^2(Z),$ ${\mathbb P}\{|Z|\geq A\},$ 
and the bound of
Theorem \ref{th-main-function}, we get 
\begin{align*}
&
|{\mathbb E}\ell(\xi)-{\mathbb E}\ell(Z)|
\leq 4 c_1^2 e^{2c_2 A^2}\Bigl[C^{k^2} M_f(B;\Sigma)
\Bigl(n^{-\frac{k+\beta-\alpha(k+1+\beta)}{2}}
+ n^{-(1-\alpha)\beta/2}\sqrt{\log n}\Bigr)
+\frac{C}{\sqrt{n}}\Bigr]
\\
&
+\sqrt{2e}(2\pi)^{1/4}c_1 e^{c_2^2 \tau^2} 
e^{-A^2/(2\tau^2)} + c_1 e^{c_2^2} e^{-A^2/4}.
\end{align*}
It remains to take the supremum over the class of covariances 
$S(d_n;a)\cap \{\sigma_f(\Sigma;B)\geq \sigma_0\}$
and over all the operators $B$ with $\|B\|_1\leq 1,$ 
and to pass to the limit first as $n\to\infty$ and then as $A\to\infty.$

\end{proof}

\section{Lower bounds}
\label{Sec:Lower bounds}

Our main goal in this section is to prove Theorem \ref{th:main_lower_bound} 
stated in Subsection \ref{SubSec:Overview}.

\begin{proof} The main part of the proof is based on an application of van Trees inequality and follows 
the same lines as the proof of a minimax lower bound for estimation of linear functionals of principal 
components in \cite{Koltchinskii_Nickl}.
We will need the following lemma (that could be of independent interest) showing the Lipschitz property of the function $\Sigma\mapsto \sigma_f^2(\Sigma;B).$ It holds for an arbitrary 
separable Hilbert space ${\mathbb H}$ (not necessarily finite-dimensional). 

\begin{lemma}
\label{Lemma:sigmaf_lipsch}
Suppose, for some $s\in (1,2],$ 
$f\in B_{\infty,1}^s({\mathbb R}).$
Then 
\begin{align}
\label{sigmaf_lipsch}
&
\nonumber
\Bigl|\sigma_f^2(\Sigma+H;B)-\sigma_f^2(\Sigma;B)\Bigr| 
\\
&
\leq 
\|f'\|_{L_{\infty}} (2\|\Sigma\|+\|H\|)\|B\|_1^2\Bigl[2\|f'\|_{L_{\infty}} \|H\|  
+8\|f\|_{B_{\infty,1}^s} \|\Sigma\| \|H\|^{s-1}\Bigr].
\end{align}
\end{lemma}

\begin{proof}
Note that 
\begin{align*}
&
\sigma_f^2(\Sigma;B)= 2 \Bigl\|\Sigma^{1/2}Df(\Sigma;B)\Sigma^{1/2}\Bigr\|_2^2
\\
&
= 2 {\rm tr}\Bigl(\Sigma^{1/2}Df(\Sigma;B)\Sigma Df(\Sigma;B)\Sigma^{1/2} \Bigr)
=2 {\rm tr}\Bigl(\Sigma Df(\Sigma;B)\Sigma Df(\Sigma;B)\Bigr).
\end{align*}
This implies that 
\begin{align}
\label{sigmaf_0}
&
\nonumber
\sigma_f^2(\Sigma+H;B)-\sigma_f^2(\Sigma;B)
\\
&
\nonumber
=2 {\rm tr}\Bigl(H Df(\Sigma+H;B)(\Sigma+H) Df(\Sigma+H;B)\Bigr)
\\
&
\nonumber
+2 {\rm tr}\Bigl(\Sigma (Df(\Sigma+H;B)-Df(\Sigma;B))(\Sigma+H) Df(\Sigma+H;B)\Bigr)
\\
&
\nonumber
+2 {\rm tr}\Bigl(\Sigma Df(\Sigma;B) H Df(\Sigma+H;B)\Bigr)
\\
&
+2 {\rm tr}\Bigl(\Sigma Df(\Sigma;B) \Sigma (Df(\Sigma+H;B)-Df(\Sigma;B)\Bigr).
\end{align}
We then have 
\begin{align}
\label{sigmaf_1}
&
\nonumber
\Bigl|2 {\rm tr}\Bigl(H Df(\Sigma+H;B)(\Sigma+H) Df(\Sigma+H;B)\Bigr)\Bigr|
\\
&
\nonumber
\leq 2 \|Df(\Sigma+H;B)(\Sigma+H) Df(\Sigma+H;B)\|_1\|H\|
\\
&
\nonumber
\leq 2 \|\Sigma+H\| \|Df(\Sigma+H;B)Df(\Sigma+H;B)\|_1\|H\|
\\
&
\nonumber
\leq 2 \|\Sigma+H\| \|Df(\Sigma+H;B)\|_2^2\|H\|
\\
&
\leq 2 \|f'\|_{L_{\infty}}^2 (\|\Sigma\|+\|H\|) \|H\|\|B\|_2^2. 
\end{align}
Similarly, it could be shown that 
\begin{align}
\label{sigmaf_2}
\Bigl|2 {\rm tr}\Bigl(\Sigma Df(\Sigma;B) H Df(\Sigma+H;B)\Bigr)\Bigr|
\leq 2 \|f'\|_{L_{\infty}}^2 \|\Sigma\| \|H\|\|B\|_2^2.
\end{align}
Also, we have 
\begin{align*}
&
2 {\rm tr}\Bigl(\Sigma (Df(\Sigma+H;B)-Df(\Sigma;B))(\Sigma+H) Df(\Sigma+H;B)\Bigr)
\\
&
=\langle (Df(\Sigma+H)-Df(\Sigma))(B), C\rangle
= \langle (Df(\Sigma+H)-Df(\Sigma))(C), B\rangle
\\
&
= \langle Df(\Sigma+H;C)-Df(\Sigma;C), B\rangle, 
\end{align*}
where $C:= (\Sigma+H)Df(\Sigma;B) \Sigma + \Sigma Df(\Sigma;B)(\Sigma+H).$
Using bound \eqref{Djf_lip}, this implies 
\begin{align}
\label{sigmaf_3}
&
\nonumber
\Bigl|2 {\rm tr}\Bigl(\Sigma (Df(\Sigma+H;B)-Df(\Sigma;B))(\Sigma+H) Df(\Sigma+H;B)\Bigr)\Bigr|
\\
&
\nonumber
=|\langle Df(\Sigma+H;C)-Df(\Sigma;C), B\rangle|
\leq \|Df(\Sigma+H;C)-Df(\Sigma;C)\|\|B\|_1
\\
&
\nonumber
\leq 4\|f\|_{B_{\infty,1}^s}\|C\| \|H\|^{s-1}\|B\|_1
\leq 8\|f\|_{B_{\infty,1}^s} \|\Sigma\| \|\Sigma+H\| \|Df(\Sigma;B)\| \|H\|^{s-1}\|B\|_1
\leq 
\\
&
\leq 8\|f'\|_{L_{\infty}} \|f\|_{B_{\infty,1}^s} \|\Sigma\|(\|\Sigma\|+\|H\|)\|H\|^{s-1}
\|B\|_1^2.
\end{align}
Similarly,
\begin{align}
\label{sigmaf_4}
\Bigl|2 {\rm tr}\Bigl(\Sigma Df(\Sigma;B) \Sigma (Df(\Sigma+H;B)-Df(\Sigma;B)\Bigr)\Bigr|
\leq 
8\|f'\|_{L_{\infty}} \|f\|_{B_{\infty,1}^s} \|\Sigma\|^2 \|H\|^{s-1}
\|B\|_1^2.
\end{align}
Substituting bound \eqref{sigmaf_1}, \eqref{sigmaf_2}, \eqref{sigmaf_3} and \eqref{sigmaf_4}
into \eqref{sigmaf_0}, we get \eqref{sigmaf_lipsch}.

\end{proof}

For given $a'\in (1,a)$ and $\sigma_0'>\sigma_0,$ assume that $\frak{B}_f(d_n;a';\sigma_0')\neq \emptyset$ (otherwise, the proof becomes trivial)  
and, for $B$ with $\|B\|_1\leq 1$ such that 
$\mathring{{\mathcal S}}_{f,B}(d_n;a';\sigma_0')\neq \emptyset,$
consider $\Sigma_0\in \mathring{{\mathcal S}}_{f,B}(d_n;a';\sigma_0').$  
For $H\in {\mathbb B}_{sa}({\mathbb H})$ 
and $c>0,$ define 
$$
\Sigma_t:= \Sigma_0+\frac{t H}{\sqrt{n}}\ \ {\rm and}\ \ 
{\mathcal S}_{c,n}(\Sigma_0,H):=\{\Sigma_t: t\in [-c,c]\}.
$$ 
In what follows, $H$ will be chosen so that 
\begin{equation}
\label{cond_HHHH}
\|H\|\leq \|f'\|_{L_{\infty}} a^2.
\end{equation}
Recall that the set $\mathring{{\mathcal S}}_{f,B}(d_n;a;\sigma_0)$ is open in operator norm topology,
so, $\Sigma_0$ is its interior point. Moreover, let $\delta>0$ and suppose that $\|\Sigma-\Sigma_0\|<\delta.$ If $\delta < a-a',$ then $\|\Sigma\|< a.$ If $\delta< \frac{a-a'}{2a^2},$ then it is easy
to check that $\|\Sigma^{-1}\|<a.$ Also, using the bound of Lemma \ref{Lemma:sigmaf_lipsch},
it is easy to show that, for $B$ with $\|B\|_1\leq 1,$ the condition 
\begin{equation}
\label{cond_delt_AA}
\|f'\|_{L_{\infty}} (2a+\delta)\Bigl[2\|f'\|_{L_{\infty}} \delta 
+8\|f\|_{B_{\infty,1}^s} a \delta^{s-1}\Bigr]
\leq (\sigma_0')^2- \sigma_0^2 
\end{equation}
implies that $\sigma(\Sigma;B)>\sigma_0.$
Thus, for a small enough $\delta=\delta(f,s,a,a',\sigma_0,\sigma_0')\in (0,1)$ satisfying 
assumptions $\delta <\frac{a-a'}{2a^2}$ and \eqref{cond_delt_AA}, we have 
$$
B(\Sigma_0;\delta):=\{\Sigma: \|\Sigma-\Sigma_0\|< \delta\}\subset 
\mathring{{\mathcal S}}_{f,B}(d_n;a;\sigma_0).
$$ 
For given $c$ and $\delta,$ for $H$ satisfying \eqref{cond_HHHH} and for all large enough $n$ 
(more specifically, for $n>\frac{c^2 a^4 \|f'\|_{L_{\infty}}^2}{\delta^2}$),
we have
\begin{equation} 
\label{assumption_on_H_0}
\frac{c\|H\|}{\sqrt{n}}<\delta,
\end{equation}
implying that 
$
{\mathcal S}_{c,n}(\Sigma_0,H)\subset B(\Sigma_0;\delta)\subset \mathring{{\mathcal S}}_{f,B}(d_n;a;\sigma_0).
$ 
Define 
$$
\varphi (t):= \langle f(\Sigma_t),B\rangle, t\in [-c,c].
$$
Clearly, $\varphi$ is continuously differentiable function 
with derivative
\begin{equation}
\label{deriv_varphi}
\varphi'(t)= \frac{1}{\sqrt{n}}\langle Df(\Sigma_t;H), B\rangle, t\in [-c,c].
\end{equation}
Consider the following parametric model:
\begin{equation}
\label{parametric_model}
X_1,\dots , X_n \ \ {\rm i.i.d.}\ \sim N(0;\Sigma_t), t\in [-c,c].
\end{equation}
It is well known that the Fisher information matrix for model $X\sim N(0;\Sigma)$
with non-singular covariance $\Sigma$ is $I(\Sigma)=\frac{1}{2}(\Sigma^{-1}\otimes \Sigma^{-1})$
(see, e.g., \cite{Eaton}). This implies that the Fisher information for model $X\sim N(0,\Sigma_t), t\in [-c,c]$
is 
$
I(t)= \langle I(\Sigma_t)\Sigma'_t, \Sigma'_t\rangle = \frac{1}{n}\langle I(\Sigma_t)H,H\rangle
$ 
and for model \eqref{parametric_model} it is 
$$
I_n(t)= nI(t)= \langle I(\Sigma_t)H,H\rangle=
\frac{1}{2}\langle (\Sigma_t^{-1}\otimes \Sigma_t^{-1})H, H\rangle
$$
$$
= \frac{1}{2}\langle \Sigma_t^{-1}H\Sigma_t^{-1}, H\rangle = 
\frac{1}{2}{\rm tr}(\Sigma_t^{-1}H\Sigma_t^{-1}H) = 
\frac{1}{2} \|\Sigma_t^{-1/2}H\Sigma_t^{-1/2}\|_2^2.
$$
We will now use well known van Trees inequality (see, e.g., \cite{GillLevit}) that provides a lower bound 
on the average risk of an arbitrary estimator $T(X_1,\dots, X_n)$ of a smooth 
function $\varphi(t)$ of parameter $t$ of model \eqref{parametric_model} with respect 
to a smooth prior density $\pi_c$ on $[-c,c]$ such that $J_{\pi_c}:= \int_{-c}^c \frac{(\pi_c'(t))^2}{\pi_c(t)}dt<\infty$
and $\pi_c(c)=\pi_c(-c)=0.$ It follows from this inequality that 
\begin{align}
\label{van Trees}
&
\nonumber
\sup_{t\in [-c,c]}{\mathbb E}_t (T_n(X_1,\dots, X_n)-g(t))^2
\\
&
\geq \int_{-c}^c {\mathbb E}_t (T_n(X_1,\dots, X_n)-g(t))^2\pi_c(t)dt
\geq \frac{\Bigl(\int_{-c}^c \varphi'(t)\pi_c(t)dt\Bigr)^2}
{\int_{-c}^c I_n(t)\pi_c(t)dt+ J_{\pi_c}}.
\end{align}
A common choice of prior is $\pi_c(t):= \frac{1}{c}\pi\Bigl(\frac{t}{c}\Bigr)$ for a smooth density
$\pi$ on $[-1,1]$ with $\pi(1)=\pi(-1)=0$ and $J_{\pi}:= \int_{-1}^1 \frac{(\pi'(t))^2}{\pi(t)}dt<\infty.$ In this case,
$J_{\pi_c}=c^{-2}J_{\pi}.$
Next we provide bounds on the numerator and the denominator of the right hand side of \eqref{van Trees}.

For the numerator, we get
\begin{align*}
&
\Bigl(\int_{-c}^c \varphi'(t)\pi_c(t)dt\Bigr)^2 =  \Bigl(\int_{-c}^c 
[\varphi' (0)+(\varphi'(t)-\varphi'(0))]\pi(t/c)dt/c\Bigr)^2 
\\
&
\geq (\varphi'(0))^2 + 2 \varphi'(0) \int_{-c}^c (\varphi'(t)-\varphi'(0))\pi(t/c)dt/c
\\
&
\geq (\varphi'(0))^2 - 2 |\varphi'(0)| \int_{-c}^c |\varphi'(t)-\varphi'(0)|\pi(t/c)dt/c.
\end{align*}
Using \eqref{deriv_varphi} along with the following bound (based on bound \eqref{Djf_lip}),
\begin{align*}
&
\nonumber
|\varphi'(t)-\varphi'(0)|
\leq 
\frac{1}{\sqrt{n}}
\|Df(\Sigma_t;H)-Df(\Sigma_0;H)\|\|B\|_1
\\
&
\leq 
\frac{4}{\sqrt{n}}\|f\|_{B_{\infty,1}^s} \|\Sigma_t-\Sigma_0\|^{s-1}\|H\|\|B\|_1
\leq 
\frac{4}{n^{s/2}}\|f\|_{B_{\infty,1}^s} \|H\|^{s}\|B\|_1 |t|^{s-1},
\end{align*}
we get:
\begin{align}
\label{lower_numer_1}
&
\nonumber
\Bigl(\int_{-c}^c \varphi'(t)\pi_c(t)dt\Bigr)^2\geq 
\frac{1}{n}\langle Df(\Sigma_0;H),B\rangle^2
\\
&
\nonumber
-\frac{2}{\sqrt{n}}|\langle Df(\Sigma_0;H),B\rangle|
\frac{4}{n^{s/2}}\|f\|_{B_{\infty,1}^s} \|H\|^{s}\|B\|_1\int_{-c}^c |t|^{s-1}\pi(t/c)dt/c
\\
&
\nonumber
=
\frac{1}{n}\langle Df(\Sigma_0;H),B\rangle^2- 
\frac{8\|f\|_{B_{\infty,1}^s} \|H\|^{s}\|B\|_1 c^{s-1}}{n^{(1+s)/2}}|\langle Df(\Sigma_0;H),B\rangle|
\int_{-1}^1 |t|^{s-1}\pi(t)dt
\\
&
\geq 
\frac{1}{n}\langle Df(\Sigma_0;H),B\rangle^2- 
\frac{8\|f\|_{B_{\infty,1}^s} \|H\|^{s}\|B\|_1 c^{s-1}}{n^{(1+s)/2}}|\langle Df(\Sigma_0;H),B\rangle|.
\end{align}
Observing that 
\begin{align*}
\langle Df(\Sigma_0;H),B\rangle= \langle Df(\Sigma_0;B), H\rangle
= \langle \Sigma_0^{-1/2} D \Sigma_0^{-1/2}, \Sigma_0^{-1/2}H\Sigma_0^{-1/2}\rangle, 
\end{align*}
where $D:= \Sigma_0 Df(\Sigma_0;B) \Sigma_0,$
we can rewrite bound \eqref{lower_numer_1} as 
\begin{align}
\label{lower_numer_1_X}
&
\nonumber
\Bigl(\int_{-c}^c \varphi'(t)\pi_c(t)dt\Bigr)^2\geq 
\frac{1}{n}\langle \Sigma_0^{-1/2} D \Sigma_0^{-1/2}, \Sigma_0^{-1/2}H\Sigma_0^{-1/2}\rangle^2
\\
&
-
\Bigl|\langle \Sigma_0^{-1/2} D \Sigma_0^{-1/2}, \Sigma_0^{-1/2}H\Sigma_0^{-1/2}\rangle\Bigr|
\frac{8\|f\|_{B_{\infty,1}^s} \|H\|^{s}\|B\|_1 c^{s-1}}{n^{(1+s)/2}}.
\end{align}

To bound the denominator, we need to control $I_n(t)=\frac{1}{2}{\rm tr}(\Sigma_t^{-1}H\Sigma_t^{-1}H)$
in terms of $I_n(0)=\frac{1}{2}{\rm tr}(\Sigma_0^{-1}H\Sigma_0^{-1}H).$ To this end, note that 
$$
\Sigma_t^{-1}= \Sigma_0^{-1} + C \Sigma_0^{-1}, 
$$
where 
$
C:= \biggl(I+ \frac{t\Sigma_0^{-1}H}{\sqrt{n}}\biggr)^{-1}-I.
$
Suppose $H$ satisfies the assumption
\begin{equation}
\label{assumption_on_H_1}
\frac{c\|\Sigma_0^{-1}H\|}{\sqrt{n}}\leq 
\frac{1}{2},
\end{equation}
which also implies that 
$
\|C\|\leq 2|t| \frac{\|\Sigma_0^{-1}H\|}{\sqrt{n}}\leq  1.
$
Note also that 
$$
{\rm tr}(\Sigma_t^{-1}H\Sigma_t^{-1}H)=
{\rm tr}(\Sigma_0^{-1}H\Sigma_0^{-1}H)
+ 2{\rm tr}(C\Sigma_0^{-1}H\Sigma_0^{-1}H)+ {\rm tr}(C\Sigma_0^{-1}H C\Sigma_0^{-1}H).
$$ 
Therefore,
\begin{align*}
&
I_n(t) \leq I_n(0) + \|C\|\|\Sigma_0^{-1}H \Sigma_0^{-1}H\|_1 
+ \frac{1}{2} \|C\Sigma_0^{-1}H\|_2 \|H\Sigma_0^{-1}C\|_2
\\
&
\leq I_n(0)+ \Bigl(\|C\|+ \frac{\|C\|^2}{2}\Bigr) \|\Sigma_0^{-1}H\|_2^2 
\leq I_n(0) + 3 \frac{|t|\|\Sigma_0^{-1}H\|_2^3}{\sqrt{n}}
\end{align*}
and 
\begin{align}
\label{Fisher_up}
&
\nonumber
\int_{-c}^c I_n(t)\pi(t/c)dt/c \leq I_n(0)+ 3 \frac{\|\Sigma_0^{-1}H\|_2^3}{\sqrt{n}}
\int_{-c}^c |t|\pi(t/c)dt/c
\\
&
\leq
\frac{1}{2}\|\Sigma_0^{-1/2}H\Sigma_0^{-1/2}\|_2^2 +
3c \frac{\|\Sigma_0^{-1}H\|_2^3}{\sqrt{n}}.
\end{align}
Substituting bounds \eqref{lower_numer_1_X} and \eqref{Fisher_up} into inequality \eqref{van Trees},
we get 
\begin{align}
\label{van Trees_A}
&
\nonumber
\sup_{t\in [-c,c]}n {\mathbb E}_t (T_n(X_1,\dots, X_n)-g(t))^2
\\
&
\geq \frac
{\langle \Sigma_0^{-1/2} D \Sigma_0^{-1/2}, \Sigma_0^{-1/2}H\Sigma_0^{-1/2}\rangle^2
-
\Bigl|\langle \Sigma_0^{-1/2} D \Sigma_0^{-1/2}, \Sigma_0^{-1/2}H\Sigma_0^{-1/2}\rangle\Bigr|
\frac{8\|f\|_{B_{\infty,1}^s} \|H\|^{s}\|B\|_1 c^{s-1}}{n^{(s-1)/2}}
}
{
\frac{1}{2}\|\Sigma_0^{-1/2}H\Sigma_0^{-1/2}\|_2^2 +
3c \frac{\|\Sigma_0^{-1}H\|_2^3}{\sqrt{n}}+ \frac{J_{\pi}}{c^2}
}.
\end{align}
Note that 
\begin{equation*}
\|\Sigma_0^{-1/2}D\Sigma_0^{-1/2}\|_2^2 = \|\Sigma_0^{1/2}Df(\Sigma_0;B)\Sigma_0^{1/2}\|_2^2
=\frac{1}{2} \sigma_f^2(\Sigma_0;B).
\end{equation*}
In what follows, we use $H:=D,$ which clearly satisfies condition \eqref{cond_HHHH}
since, for $\Sigma_0\in \mathring{\mathcal S}_{f,B}(d_n;a;\sigma_0)$ and $\|B\|_1\leq 1,$
\begin{align}
\label{bd_DDD}
&
\nonumber
\|D\| = \|\Sigma_0 Df(\Sigma_0;B)\Sigma_0\| \leq \|\Sigma_0\|^2 \|Df(\Sigma_0;B)\|
\\
&
\leq  \|f'\|_{L_{\infty}} \|B\|_2 \|\Sigma\|_0^2 \leq  a^2 \|f'\|_{L_{\infty}} \|B\|_2
\leq a^2 \|f'\|_{L_{\infty}}.
\end{align}
We also have 
\begin{align}
\label{bd_DDD'''}
&
\nonumber
\|\Sigma_0^{-1}D\|_2^2 = {\rm tr}(Df(\Sigma_0;B) \Sigma_0^2 Df(\Sigma_0,B))
\\
&
\leq \|\Sigma_0\|^2 \|Df(\Sigma_0;B) \|_2^2 
\leq \|\Sigma_0\|^2\|f'\|_{L_{\infty}}^2 \|B\|_2^2 
\leq a^2\|f'\|_{L_{\infty}}^2 \|B\|_2^2 \leq 
a^2\|f'\|_{L_{\infty}}^2, 
\end{align}
implying that assumptions \eqref{assumption_on_H_0} and \eqref{assumption_on_H_1} hold for 
$H=D$ provided that 
\begin{equation}
\label{how_large_n}
n> \frac{4c^2 a^4 \|f'\|_{L_{\infty}}^2}{\delta^2}.
\end{equation}
With this choice of $H,$ \eqref{van Trees_A} implies 
\begin{align}
\label{van Trees_B}
&
\sup_{t\in [-c,c]}\frac{n {\mathbb E}_t (T_n(X_1,\dots, X_n)-g(t))^2}{\sigma_f^2(\Sigma_0;B)}
\geq 1-\frac
{3c \frac{\|\Sigma_0^{-1}D\|_2^3}{\sqrt{n}}+ 
\frac{4\|f\|_{B_{\infty,1}^s} \|D\|^{s}\|B\|_1 c^{s-1}}{n^{(s-1)/2}}
+\frac{J_{\pi}}{c^2}}
{
\frac{1}{4}\sigma_f^2(\Sigma_0;B)
+3c \frac{\|\Sigma_0^{-1}D\|_2^3}{\sqrt{n}}+ \frac{J_{\pi}}{c^2}
}.
\end{align}
It follows from \eqref{van Trees_B}, \eqref{bd_DDD} and \eqref{bd_DDD'''} that for $B$ satisfying 
$\|B\|_1\leq 1$
\begin{align}
\label{van Trees_C}
&
\sup_{t\in [-c,c]}\frac{n {\mathbb E}_t (T_n(X_1,\dots, X_n)-g(t))^2}{\sigma_f^2(\Sigma_0;B)}
\geq 1-\gamma_{n,c}(f;a;\sigma_0)
\end{align}
where
\begin{align*}
\gamma_{n,c}(f;a;\sigma_0):=\frac
{\frac{3 a^3 \|f'\|_{L_{\infty}}^3 c}{\sqrt{n}}+
\frac{4 a^{2s}\|f\|_{B_{\infty,1}^s} \|f'\|_{L_{\infty}}^s c^{s-1}}{n^{(s-1)/2}}
+ \frac{J_{\pi}}{c^2}}
{\frac{1}{4}\sigma_0^2}.
\end{align*}
Denote 
$
\sigma^2 (t):= \sigma_f^2(\Sigma_t;B), t\in [-c,c].
$
By Lemma \ref{Lemma:sigmaf_lipsch}, 
$$
|\sigma^2(t)-\sigma^2(0)|\leq 
\|f'\|_{L_{\infty}} \biggl(2\|\Sigma_0\|+\frac{|t|\|D\|}{\sqrt{n}}\biggr)\|B\|_1^2
\Bigl[2\|f\|_{L_{\infty}} \frac{|t|\|D\|}{\sqrt{n}}  
+8\|f\|_{B_{\infty,1}^s} \|\Sigma_0\| \frac{|t|^{s-1}\|D\|^{s-1}}{n^{(s-1)/2}}\Bigr].
$$
Note that assumption \eqref{assumption_on_H_1} on $H=D$ implies 
that 
$$
\frac{c\|D\|}{\sqrt{n}}= \frac{c\|\Sigma_0\Sigma_0^{-1}H\|}{\sqrt{n}}
\leq \frac{c\|\Sigma_0^{-1}H\|\|\Sigma_0\|}{\sqrt{n}}\leq \frac{\|\Sigma_0\|}{2}.
$$
Using bound \eqref{bd_DDD}, we get that, for all $t\in [-c,c]$ and $B$ with $\|B\|_1\leq1,$
\begin{align*}
&
|\sigma^2(t)-\sigma^2(0)|
\leq 
\frac{6 c a^3 \|f'\|_{L_{\infty}}^3}{n^{1/2}}  
+ \frac{24 c^{s-1}a^{2s} \|f'\|^s_{L_{\infty}}\|f\|_{B_{\infty,1}^s}}{n^{(s-1)/2}}
=:\lambda_{n,c} (f;a).
\end{align*}
which implies that 
\begin{align}
\label{ratio_sigma}
&
\sup_{t\in [-c,c]} \frac{\sigma^2(t)}{\sigma^2(0)}
\leq 
1+ \frac{\lambda_{n,c} (f;a)}{\sigma_0^2}.
\end{align}
It follows from \eqref{van Trees_C} and \eqref{ratio_sigma} that 
\begin{align*}
&
\nonumber
\sup_{t\in [-c,c]}\frac{n {\mathbb E}_t (T_n(X_1,\dots, X_n)-g(t))^2}{\sigma^2(t)}
\biggl(1+ \frac{\lambda_{n,c} (f;a)}{\sigma_0^2}\biggr)
\\
&
\nonumber 
\geq \sup_{t\in [-c,c]}\frac{n {\mathbb E}_t (T_n(X_1,\dots, X_n)-g(t))^2}{\sigma^2(t)}
\sup_{t\in [-c,c]} \frac{\sigma^2(t)}{\sigma^2(0)}
\\
&
\nonumber
\geq \sup_{t\in [-c,c]}\frac{n {\mathbb E}_t (T_n(X_1,\dots, X_n)-g(t))^2}{\sigma_f^2(\Sigma_0;B)}
\geq 1-\gamma_{n,c}(f;a;\sigma_0), 
\end{align*}
which yields that for all $B\in {\frak B}_f(d_n;a';\sigma_0')$ 
\begin{align}
\label{van Trees_E}
&
\nonumber 
\sup_{\Sigma \in \mathring{\mathcal S}_{f,B}(d_n;a;\sigma_0)}
\frac{n {\mathbb E}_{\Sigma}(T_n(X_1,\dots, X_n)-\langle f(\Sigma),B\rangle)^2}{\sigma_f^2(\Sigma;B)}
\\
&
\geq 
\sup_{t\in [-c,c]}\frac{n {\mathbb E}_t (T_n(X_1,\dots, X_n)-g(t))^2}{\sigma^2(t)}
\geq 
\frac{1-\gamma_{n,c}(f;a;\sigma_0)}{1+ \frac{\lambda_{n,c} (f;a)}{\sigma_0^2}}.
\end{align}
It remains to observe that 
$$
\lim_{c\to \infty}\limsup_{n\to \infty}\gamma_{n,c}(f;a;\sigma_0)=0\ \ {\rm and}\ \ 
\lim_{c\to \infty}\limsup_{n\to \infty}\lambda_{n,c}(f;a)=0
$$
to complete the proof.

\end{proof}

\begin{remark}
It follows from the proof that the following local version of \eqref{main_lower_bound} also holds: for all  
$a'\in (1,a),$ $\sigma_0'>\sigma_0,$
\begin{equation}
\label{main_lower_bound_local}
\lim_{\delta\to 0}\liminf_{n\to\infty} \inf_{T_n} 
\inf_{B\in {\frak B}_f(d_n;a';\sigma_0'
)}
\inf_{\Sigma_0\in \mathring{{\mathcal S}}_{f,B}(d_n;a';\sigma_0')}
\sup_{\|\Sigma-\Sigma_0\|<\delta}
\frac{n {\mathbb E}_{\Sigma}(T_n-\langle f(\Sigma),B\rangle)^2}{\sigma_f^2(\Sigma;B)}
\geq 1.
\end{equation}
\end{remark}

\bigskip
\footnotesize
\noindent\textit{Acknowledgments.}
The author is very thankful to Richard Nickl, Alexandre Tsybakov 
and Mayya Zhilova for several helpful conversations and to Anna Skripka for careful reading of Section 2 of the paper and making several useful suggestions. A part of this work was done while on leave 
to the University of Cambridge. The author is thankful to the Department of Pure Mathematics and Mathematical Statistics of this university for its hospitality. Finally, the author is thankful to the referees
for numerous helpful comments.  
The research was partially supported by NSF Grants DMS-1810958, DMS-1509739 and CCF-1523768.

\end{document}